\documentclass[reqno]{amsart}
\usepackage[margin = 1in]{geometry}
\usepackage{amsmath, amssymb, amsthm, fancyhdr, verbatim, graphicx}
\usepackage{enumerate}
\usepackage[cmtip, all, 2cell]{xy}
\UseAllTwocells
\usepackage[usenames,dvipsnames]{xcolor}
\usepackage{mathrsfs}
\usepackage{tikz-cd}
\usetikzlibrary{decorations.pathmorphing}
\usepackage{framed, hyperref}
\usepackage[titletoc]{appendix}
\usepackage{bbm}
\usepackage{lipsum}
\usepackage{adjustbox}

\numberwithin{equation}{section}

\DeclareFontFamily{U}{matha}{\hyphenchar\font45}
\DeclareFontShape{U}{matha}{m}{n}{
      <5> <6> <7> <8> <9> <10> gen * matha
      <10.95> matha10 <12> <14.4> <17.28> <20.74> <24.88> matha12
      }{}
\DeclareSymbolFont{matha}{U}{matha}{m}{n}
\DeclareFontFamily{U}{mathx}{\hyphenchar\font45}
\DeclareFontShape{U}{mathx}{m}{n}{
      <5> <6> <7> <8> <9> <10>
      <10.95> <12> <14.4> <17.28> <20.74> <24.88>
      mathx10
      }{}
\DeclareSymbolFont{mathx}{U}{mathx}{m}{n}

\DeclareMathSymbol{\obot}{2}{matha}{"6B}

\RequirePackage{xspace}

\newcommand{\sL}{\ensuremath{\mathscr{L}}\xspace}
\newcommand{\sM}{\ensuremath{\mathscr{M}}\xspace}

\newcommand{\sP}{\ensuremath{\mathscr{P}}\xspace}

\newcommand{\rD}{\ensuremath{\mathrm{D}}\xspace}

\newcommand{\rH}{\ensuremath{\mathrm{H}}\xspace}

\usepackage{color}

\newcommand{\F}{\mathbf{F}}

\newcommand{\G}{\mathbf{G}}

\newcommand{\wt}[1]{\widetilde{#1}}

\newcommand{\Q}{\mathbf{Q}}
\newcommand{\Z}{\mathbf{Z}}

\newcommand{\mf}[1]{\mathfrak{#1}}

\newcommand{\cl}{\overline}

\newcommand{\ul}[1]{\underline{#1}}
\newcommand{\ol}[1]{\overline{#1}}
\newcommand{\wh}[1]{\widehat{#1}}

\newcommand{\Cal}[1]{\mathcal{#1}}
\DeclareMathOperator{\et}{\text{\'{e}t}}
\newcommand{\mbf}[1]{\mathbf{#1}}

\newcommand{\co}{\colon}
\newcommand{\mrm}[1]{\mathrm{#1}}
\newcommand{\msf}[1]{\mathsf{#1}}

\newcommand{\bbm}[1]{\mathbbm{#1}}
\newcommand{\PP}{\mathbf{P}}

\newcommand{\DD}{\mathbf{D}}

\newcommand{\TT}{\mathbb{T}}
\newcommand{\inj}{\hookrightarrow}
\newcommand{\surj}{\twoheadrightarrow}

\newcommand{\tw}[1]{\sm{\langle #1 \rangle}}
\newcommand{\twbrace}[1]{\{#1\}}
\newcommand{\pair}[1]{\langle #1 \rangle}

\newcommand{\di}{\diamond}

\newcommand{\td}{\triangledown}
\newcommand{\tu}{\vartriangle}

\DeclareMathOperator{\Tot}{Tot}
\DeclareMathOperator{\Perf}{Perf}

\DeclareMathOperator{\GL}{GL}

\DeclareMathOperator{\Frob}{Frob}

\DeclareMathOperator{\N}{\mathbf{N}}
\DeclareMathOperator{\Tr}{Tr}

\DeclareMathOperator{\Hom}{Hom}

\newcommand{\cHom}{\Cal{H}om}

\newcommand{\cRHom}{\Cal{R}\Cal{H}om}

\DeclareMathOperator{\rank}{rank}

\DeclareMathOperator{\Spec}{Spec\,}

\DeclareMathOperator{\End}{End}

\DeclareMathOperator{\Bun}{Bun}
\DeclareMathOperator{\Ext}{Ext}
\DeclareMathOperator{\Pic}{Pic}

\DeclareMathOperator{\Id}{Id}

\DeclareMathOperator{\Gr}{Gr}

\DeclareMathOperator{\QCoh}{QCoh}

\DeclareMathOperator{\Sym}{Sym}

\DeclareMathOperator{\pt}{pt}
\DeclareMathOperator{\Fib}{Fib}

\DeclareMathOperator{\Fix}{Fix}

\DeclareMathOperator{\Sht}{Sht}

\DeclareMathOperator{\pr}{pr}

\DeclareMathOperator{\Hk}{Hk}

\DeclareMathOperator{\CH}{CH}

\DeclareMathOperator{\Map}{Map}

\DeclareMathOperator{\RHom}{RHom}
\DeclareMathOperator{\FT}{FT}

\DeclareMathOperator{\cc}{\mathfrak{c}}
\DeclareMathOperator{\dd}{\mathfrak{d}}
\DeclareMathOperator{\ee}{\mathfrak{e}}

\newcommand{\sm}[1]{{\scriptstyle  #1}}

\newcommand{\limit}{\varprojlim}

\DeclareMathOperator{\ev}{ev}
\DeclareMathOperator{\coev}{coev}

\newcommand{\arith}{\mrm{arith}}

\newcommand{\incl}{\hookrightarrow}
\newcommand{\isom}{\stackrel{\sim}{\to}}

\newcommand{\Qsh}[1]{\Q_{#1}}

\renewcommand{\j}[1]{\langle{#1}\rangle}
\newcommand{\bu}{\bullet}
\newcommand{\ov}{\overline}

\newcommand\xr{\xrightarrow}

\newcommand\ot{\otimes}

\renewcommand\c{\circ}

\renewcommand\sp{\mathfrak{sp}}

\renewcommand\a\alpha
\renewcommand\b\beta
\newcommand\g\gamma
\renewcommand\d\delta

\newcommand{\io}{\iota}

\newcommand{\ph}{\varphi}
\newcommand{\s}{\sigma}

\newcommand{\y}{\eta}
\newcommand{\z}{\zeta}

\newcommand{\om}{\omega}

\newcommand\sh{\sharp}

\newcommand\frL{\mathfrak{L}}
\newcommand\frM{\mathfrak{M}}

\newcommand\frc{\mathfrak{c}}

\newcommand\frq{\mathfrak{q}}

\newcommand\frs{\mathfrak{s}}

\newcommand\Corr{\mathrm{Corr}}
\newcommand\CoCorr{\mathrm{CoCorr}}

\def\upp{\textup{upp}}
\def\low{\textup{low}}
\def\out{\textup{out}}

\newcommand{\SH}{\mrm{SH}}
\newcommand{\DA}[1]{\mrm{DA}_{\et}({#1}; \Q)}
\newcommand{\Dmot}[1]{\cD_{\mrm{mot}}({#1};\Q)}
\newcommand{\Dmotg}[1]{\cD_{\mrm{mot},\mrm{gm}}({#1};\Q)}
\DeclareMathOperator{\DM}{DM}
\DeclareMathOperator{\LZ}{LZ}
\DeclareMathOperator{\Av}{Av}

\newcommand{\fsp}{\mathfrak{sp}}
\newcommand{\spc}{\mathrm{sp}}

\newcommand{\rFT}{\mrm{FT}^{\mrm{ren}}}
\newcommand{\aFT}{\mrm{FT}^{\mrm{arith}}}

\DeclareMathOperator{\red}{red}

\newcommand\nc{\newcommand}
\nc\renc{\renewcommand}

\nc{\A}{\bA}
\nc{\V}{\Tot}
\nc{\Ex}{\mrm{Ex}}
\nc{\gys}{\mrm{gys}}
\nc{\un}{\mbf{1}}
\nc{\vb}[1]{\langle{#1}\rangle}
\nc{\unit}{\mrm{unit}}
\nc{\counit}{\mrm{counit}}
\nc{\DVect}{\on{DVect}}
\nc{\MGL}{\mrm{MGL}}
\nc{\MGLmod}{\on{\mbf{D}_\MGL}}
\nc{\Nis}{\mrm{Nis}}

\nc{\Gm}{{\bG_{m}}}
\nc{\act}{\mrm{act}}
\nc{\Chom}{\on{C}_\bullet}
\nc{\oH}{{\mrm{H}}}

\newcommand{\quo}[1]{{}^{\dagger}{#1}}

\usepackage{xparse}

\usepackage{xspace}

\nc{\ssec}{\subsection}
\nc{\sssec}{\subsubsection}
\nc{\on}{\operatorname}
\nc{\term}[1]{#1\xspace}

\usepackage{tikz}
\usetikzlibrary{matrix}
\usepackage{tikz-cd}

\tikzset{
  commutative diagrams/.cd,
  arrow style=tikz,
  diagrams={>=latex}}
\makeatletter
\tikzset{
  column sep/.code=\def\pgfmatrixcolumnsep{\pgf@matrix@xscale*(#1)},
  row sep/.code   =\def\pgfmatrixrowsep{\pgf@matrix@yscale*(#1)},
  matrix xscale/.code=%
    \pgfmathsetmacro\pgf@matrix@xscale{\pgf@matrix@xscale*(#1)},
  matrix yscale/.code=%
    \pgfmathsetmacro\pgf@matrix@yscale{\pgf@matrix@yscale*(#1)},
  matrix scale/.style={/tikz/matrix xscale={#1},/tikz/matrix yscale={#1}}}
\def\pgf@matrix@xscale{1}
\def\pgf@matrix@yscale{1}
\makeatother

\usepackage[inline]{enumitem}
\setlist[enumerate,1]{label={(\alph*)},itemsep=\parskip}

\newlist{enumcompress}{enumerate}{1}
\setlist[enumcompress,1]{
  label={(\alph*)},
  itemsep=0.3\parskip,
  leftmargin=*,
  align=left,
  topsep=0.5em
}

\newlist{thmlist}{enumerate}{2}
\setlist[thmlist,1]{
  label={\em(\roman*)}, ref={(\roman*)},
  itemsep=0.5em,
  align=right,widest=vi)}
\setlist[thmlist,2]{
  label={\em(\alph*)}, ref={(\alph*)},
  itemsep=0.75em,
  labelsep=0em,labelindent=0em,leftmargin=*,align=left,widest=vi),
  topsep=0.75em}

\newlist{thmlistbis}{enumerate}{1}
\setlist[thmlistbis,1]{
  label={\em(\roman*~\textit{bis})},
  ref={(\roman*}~\textit{bis}\upshape{)},
  itemsep=0.5em,
  align=right, widest=vi)}

\newlist{defnlist}{enumerate}{2}
\setlist[defnlist,1]{
  label={(\roman*)}, ref={(\roman*)},
  itemsep=0.5em,
  align=right, widest=vi)}

\setlist[defnlist,2]{
  label={(\alph*)}, ref={(\alph*)},
  itemsep=0.75em,
  labelsep=0em,labelindent=0em,leftmargin=*,align=left,widest=vi),
  topsep=0.75em}

\newlist{inlinelist}{enumerate*}{1}
\setlist[inlinelist,1]{label={(\alph*)}}

\newlist{inlinedefnlist}{enumerate*}{1}
\definecolor{green}{HTML}{38550C}
\setlist[inlinedefnlist,1]{label={\color{green}(\roman*)}}

\nc{\inftyCat}{\term{$\infty$-category}}
\nc{\inftyCats}{\term{$\infty$-categories}}

\nc{\inftyGrpd}{\term{$\infty$-groupoid}}
\nc{\inftyGrpds}{\term{$\infty$-groupoids}}

\usepackage{microtype}
\usepackage{textcomp}

\nc{\cA}{\ensuremath{\mathcal{A}}\xspace}
\nc{\cB}{\ensuremath{\mathcal{B}}\xspace}
\nc{\cC}{\ensuremath{\mathcal{C}}\xspace}
\nc{\cD}{\ensuremath{\mathcal{D}}\xspace}
\nc{\cE}{\ensuremath{\mathcal{E}}\xspace}
\nc{\cF}{\ensuremath{\mathcal{F}}\xspace}
\nc{\cG}{\ensuremath{\mathcal{G}}\xspace}
\nc{\cH}{\ensuremath{\mathcal{H}}\xspace}
\nc{\cI}{\ensuremath{\mathcal{I}}\xspace}
\nc{\cJ}{\ensuremath{\mathcal{J}}\xspace}
\nc{\cK}{\ensuremath{\mathcal{K}}\xspace}
\nc{\cL}{\ensuremath{\mathcal{L}}\xspace}
\nc{\cM}{\ensuremath{\mathcal{M}}\xspace}
\nc{\cN}{\ensuremath{\mathcal{N}}\xspace}
\nc{\cO}{\ensuremath{\mathcal{O}}\xspace}
\nc{\cP}{\ensuremath{\mathcal{P}}\xspace}
\nc{\cQ}{\ensuremath{\mathcal{Q}}\xspace}
\nc{\cR}{\ensuremath{\mathcal{R}}\xspace}
\nc{\cS}{\ensuremath{\mathcal{S}}\xspace}
\nc{\cT}{\ensuremath{\mathcal{T}}\xspace}
\nc{\cU}{\ensuremath{\mathcal{U}}\xspace}
\nc{\cV}{\ensuremath{\mathcal{V}}\xspace}
\nc{\cW}{\ensuremath{\mathcal{W}}\xspace}
\nc{\cX}{\ensuremath{\mathcal{X}}\xspace}
\nc{\cY}{\ensuremath{\mathcal{Y}}\xspace}
\nc{\cZ}{\ensuremath{\mathcal{Z}}\xspace}

\nc{\bA}{\ensuremath{\mathbf{A}}\xspace}
\nc{\bB}{\ensuremath{\mathbf{B}}\xspace}
\nc{\bC}{\ensuremath{\mathbf{C}}\xspace}
\nc{\bD}{\ensuremath{\mathbf{D}}\xspace}
\nc{\bE}{\ensuremath{\mathbf{E}}\xspace}
\nc{\bF}{\ensuremath{\mathbf{F}}\xspace}
\nc{\bG}{\ensuremath{\mathbf{G}}\xspace}
\nc{\bH}{\ensuremath{\mathbf{H}}\xspace}
\nc{\bI}{\ensuremath{\mathbf{I}}\xspace}
\nc{\bJ}{\ensuremath{\mathbf{J}}\xspace}
\nc{\bK}{\ensuremath{\mathbf{K}}\xspace}
\nc{\bL}{\ensuremath{\mathbf{L}}\xspace}
\nc{\bM}{\ensuremath{\mathbf{M}}\xspace}
\nc{\bN}{\ensuremath{\mathbf{N}}\xspace}
\nc{\bO}{\ensuremath{\mathbf{O}}\xspace}
\nc{\bP}{\ensuremath{\mathbf{P}}\xspace}
\nc{\bQ}{\ensuremath{\mathbf{Q}}\xspace}
\nc{\bR}{\ensuremath{\mathbf{R}}\xspace}
\nc{\bS}{\ensuremath{\mathbf{S}}\xspace}
\nc{\bT}{\ensuremath{\mathbf{T}}\xspace}
\nc{\bU}{\ensuremath{\mathbf{U}}\xspace}
\nc{\bV}{\ensuremath{\mathbf{V}}\xspace}
\nc{\bW}{\ensuremath{\mathbf{W}}\xspace}
\nc{\bX}{\ensuremath{\mathbf{X}}\xspace}
\nc{\bY}{\ensuremath{\mathbf{Y}}\xspace}
\nc{\bZ}{\ensuremath{\mathbf{Z}}\xspace}

\nc{\bbA}{\ensuremath{\mathbb{A}}\xspace}
\nc{\bbB}{\ensuremath{\mathbb{B}}\xspace}
\nc{\bbC}{\ensuremath{\mathbb{C}}\xspace}
\nc{\bbD}{\ensuremath{\mathbb{D}}\xspace}
\nc{\bbE}{\ensuremath{\mathbb{E}}\xspace}
\nc{\bbF}{\ensuremath{\mathbb{F}}\xspace}
\nc{\bbG}{\ensuremath{\mathbb{G}}\xspace}
\nc{\bbH}{\ensuremath{\mathbb{H}}\xspace}
\nc{\bbI}{\ensuremath{\mathbb{I}}\xspace}
\nc{\bbJ}{\ensuremath{\mathbb{J}}\xspace}
\nc{\bbK}{\ensuremath{\mathbb{K}}\xspace}
\nc{\bbL}{\ensuremath{\mathbb{L}}\xspace}
\nc{\bbM}{\ensuremath{\mathbb{M}}\xspace}
\nc{\bbN}{\ensuremath{\mathbb{N}}\xspace}
\nc{\bbO}{\ensuremath{\mathbb{O}}\xspace}
\nc{\bbP}{\ensuremath{\mathbb{P}}\xspace}
\nc{\bbQ}{\ensuremath{\mathbb{Q}}\xspace}
\nc{\bbR}{\ensuremath{\mathbb{R}}\xspace}
\nc{\bbS}{\ensuremath{\mathbb{S}}\xspace}
\nc{\bbT}{\ensuremath{\mathbb{T}}\xspace}
\nc{\bbU}{\ensuremath{\mathbb{U}}\xspace}
\nc{\bbV}{\ensuremath{\mathbb{V}}\xspace}
\nc{\bbW}{\ensuremath{\mathbb{W}}\xspace}
\nc{\bbX}{\ensuremath{\mathbb{X}}\xspace}
\nc{\bbY}{\ensuremath{\mathbb{Y}}\xspace}
\nc{\bbZ}{\ensuremath{\mathbb{Z}}\xspace}

\nc{\mit}[1]{\ensuremath{\mathit{#1}}\xspace}
\nc{\mcal}[1]{\ensuremath{\mathcal{#1}}\xspace}
\nc{\msc}[1]{\ensuremath{\mathscr{#1}}\xspace}

\nc{\sub}{\subseteq}
\nc{\too}{\longrightarrow}
\nc{\hook}{\hookrightarrow}
\nc{\hooklongrightarrow}{\lhook\joinrel\longrightarrow}
\nc{\hooklong}{\hooklongrightarrow}
\nc{\hooklongleftarrow}{\longleftarrow\joinrel\rhook}
\nc{\twoheadlongrightarrow}{\relbar\joinrel\twoheadrightarrow}
\nc{\longrightleftarrows}{\ \raisebox{0.3ex}{\(\mathrel{\substack{\xrightarrow{\rule{1em}{0em}} \\[-1ex] \xleftarrow{\rule{1em}{0em}}}}\)}\ }

\renc{\cl}{\mrm{cl}}
\renc{\ge}{\geqslant}
\renc{\le}{\leqslant}

\nc{\id}{\mathrm{id}}

\DeclareMathOperator{\rk}{\mathrm{rk}}
\nc{\uHom}{\underline{\smash{\Hom}}}
\DeclareMathOperator{\Maps}{\on{Maps}}
\nc{\uEnd}{\underline{\smash{\End}}}
\renc{\lim}{\varprojlim}
\nc{\Cofib}{\on{Cofib}}
\nc{\initial}{\varnothing}

\DeclareMathOperator*{\fibprod}{\times}

\renc{\setminus}{\smallsetminus}

\usepackage{mathtools}
\newcommand\restr[2]{{
  \left.\kern-\nulldelimiterspace 
  #1 
  \vphantom{\big|} 
  \right|_{#2} 
  }}

\newcommand{\thmref}[1]{Theorem~\ref{#1}}

\newcommand{\subsectionref}[1]{\S\ref{#1}}

\newcommand{\lemref}[1]{Lemma~\ref{#1}}
\newcommand{\propref}[1]{Proposition~\ref{#1}}
\newcommand{\corref}[1]{Corollary~\ref{#1}}

\newcommand{\remref}[1]{Remark~\ref{#1}}

\renewcommand{\eqref}[1]{(\ref{#1})}

\newcommand{\itemref}[1]{\ref{#1}}

\newtheorem{thm}{Theorem}[subsection]
\newtheorem{lemma}[thm]{Lemma}
\newtheorem{prop}[thm]{Proposition}
\newtheorem{cor}[thm]{Corollary}

\newtheorem{defn-prop}[thm]{Definition-Proposition}
\newtheorem*{defthm*}{Definition/Theorem}
\newtheorem{lem}[thm]{Lemma}

\theoremstyle{remark}
\newtheorem{remark}[thm]{Remark} 
\newtheorem{defn}[thm]{Definition}
\newtheorem{const}[thm]{Construction}
\newtheorem{example}[thm]{Example}

\newtheorem{rem}[thm]{Remark}
\newtheorem{exam}[thm]{Example}

\newtheorem{notat}[thm]{Notation}

\newtheorem*{exam*}{Example}

\makeatletter
\def\th@remark{%
  \thm@headfont{\bfseries}%
  \normalfont 
  \thm@preskip \thm@preskip 
  \thm@postskip\thm@preskip
}
\def\imod#1{\allowbreak\mkern5mu({\operator@font mod}\,\,#1)}
\makeatother

\numberwithin{equation}{subsection}

\title[]{Modularity of higher theta series II: \\ Chow group of the generic fiber}

\author{Tony Feng}
\address{University of California Berkeley, Department of Mathematics, Berkeley, CA 94720, USA}
\email{fengt@berkeley.edu}

\author{Adeel A. Khan}
\address{Institute of Mathematics, Academia Sinica, Taipei 10617, Taiwan}
\email{adeelkhan@gate.sinica.edu.tw}

\begin{document}

\begin{abstract}
Higher theta series on moduli spaces of Hermitian shtukas were constructed by Feng--Yun--Zhang and conjectured to be modular, parallel to classical conjectures in the Kudla program. In this paper we prove the modularity of higher theta series after restriction to the generic locus. The proof is an upgrade, using motivic homotopy theory, of earlier work of Feng--Yun--Zhang which established generic modularity of $\ell$-adic realizations. In the process, we develop some general tools of broader utility. One such is the \emph{motivic sheaf-cycle correspondence}, a categorical trace formalism for extracting computations in the Chow group from computations in Voevodsky's derived category of motives. Another new tool is the \emph{derived homogeneous Fourier transform}, which we use to implement a form of Fourier analysis for motives. 
\end{abstract}

\maketitle

\tableofcontents

\section{Introduction}

In this paper we synthesize two threads of research in the theory of algebraic cycles. The first thread comes from the lineage of the Birch and Swinnerton-Dyer Conjecture, and broadly speaking concerns the relationship between algebraic cycles on arithmetic moduli spaces and special values of $L$-functions. The second thread comes from the lineage of the Milnor Conjecture, through the $\A^1$-homotopy theory introduced by Voevodsky, and concerns a cohomological approach to motives. As these two domains have traditionally had little overlap, the Introduction will be aimed more broadly than usual so as to be accessible to audiences from both. 

\subsection{Number theory background}\label{ssec: nt background} The Birch and Swinnerton-Dyer Conjecture, and its generalizations such as the Beilinson--Bloch Conjecture and Bloch--Kato Conjecture\footnote{There is another ``Bloch--Kato Conjecture'' proved by Voevodsky, which generalizes the Milnor Conjecture about the relationship between Galois cohomology and Milnor K-theory. This is \emph{not} what we are referring to here, although it will become relevant later through the connection to $\A^1$-homotopy theory. We will exclusively use the phrase ``Norm residue isomorphism'' to refer to the Bloch--Kato Conjecture proved by Voevodsky.}, predict a deep relation between algebraic cycles and $L$-functions. (See \cite[\S 1.2]{Liu16} for a brief introduction to these Conjectures, or \cite{BK90} for a more extensive reference.) The classical work of Riemann \cite{Riemann} and Hecke \cite{Hecke2} founded the paradigm of accessing special values of $L$-functions through integral representations as periods of automorphic forms. The work of Gross--Zagier \cite{GZ86} introduced the idea of accessing the \emph{first derivative} of $L$-functions at special points\footnote{In turn, Gross emphasized to us the importance of Stark's work, in which the derivatives of Artin $L$-functions appear, as an inspiration for \cite{GZ86}.} as ``periods'' of geometric incarnations of automorphic forms on arithmetic moduli spaces called Shimura varieties. This opened the door to the rank 1 case of the Birch and Swinnerton-Dyer Conjecture; the higher rank case remains wide open. 
 
Theta functions are certain examples of automorphic forms built as generating functions for counting problems associated to lattices. Kudla introduced the concept of \emph{arithmetic theta functions} as an incarnation of theta functions in the arithmetic geometry of Shimura varieties. The so-called \emph{Kudla program} outlined in \cite{Kud04} (building on \cite{Kudla1997a, Kudla1997}) refers to a strategy to represent the first derivative of standard $L$-functions of cuspidal automorphic representations in terms of arithmetic geometry, giving a higher-dimensional generalization of the Gross--Zagier formula. This program involves several major conjectures: 
\begin{enumerate}
\item[(1)] The \emph{arithmetic Siegel--Weil formula}, which would relate the arithmetic volumes of arithmetic theta functions to the first derivatives of Siegel--Eisenstein series.
\item[(2)] The \emph{modularity of arithmetic theta functions}, which would enable the construction of arithmetic theta lifts. 
\item[(3)] The \emph{arithmetic inner product formula}, which would relate heights of arithmetic theta lifts to the first derivative of standard $L$-functions. 
\end{enumerate}
Remarkably, all of these problems have seen major progress in recent years. The works of Li--Zhang \cite{LZ1, LZ2}, Garcia--Sankaran \cite{GS19}, and Liu \cite{Liu11} establish the local arithmetic Siegel-Weil formula for the non-singular Fourier coefficients. The modularity has been proved in many cases; we postpone a more detailed discussion to \S \ref{ssec: discussion}. The arithmetic inner product formula was proved in many situations by Li--Liu \cite{LL1, LL2}. A modern introduction to these ideas, along with a more complete survey of recent developments, can be found in \cite{LiIHES}. 

In the function field context, the ``higher Gross--Zagier formula'' of Yun--Zhang \cite{YZ, YZII} revealed the possibility of accessing not only zeroth and first derivatives, but \emph
{all higher derivatives} of $L$-functions as periods of geometric incarnations of automorphic forms on moduli spaces of shtukas. In \cite{FYZ}, a \emph{higher Siegel--Weil formula} was established for non-singular terms, constituting the first step of a ``higher'' version of Kudla's program in this context. In \cite{FYZ2}, \emph{higher theta series} were constructed, and in this paper we prove their generic modularity (following the strategy of \cite{FYZ3}, which proved the generic modularity of the $\ell$-adic realization). This opens the door to ``higher theta-lifting'', of which one possible next application could be a ``higher arithmetic inner product formula'', but we view the modularity property as interesting in its own right. We note that the number field analogue of the modularity theorem has already had diverse applications in arithmetic geometry unrelated to the Kudla program.

\subsection{Main result} Let $X' \rightarrow X$ be an \'{e}tale double cover of smooth projective curves over a finite field $\F_q$ of characteristic $p>2$. Fix integers $n\ge m \geq 1$, and $r \geq 0$.

We recall the following definitions from \cite[\S 4.5]{FYZ2}:
\begin{itemize}
\item Let $\Bun_{GU^-(2m)}$ be the moduli stack of triples $(\cG, \frM, h)$ where $\cG$ is a vector bundle of rank $2m$ over  $X'$, $\frM$ is a line bundle over $X$, and $h$ is a skew-Hermitian isomorphism 
\[
h: \cG\isom \s^{*}\cG^{\vee}\ot \nu^{*}\frM=\s^{*}\cG^{*}\ot\nu^{*}(\om_{X}\ot \frM).
\] 
\item Let $\Bun_{\wt P_{m}}$ be the moduli stack of quadruples $(\cG,\frM, h,\cE)$ where $(\cG,\frM,h)\in \Bun_{GU(2m)}$, and $\cE\subset \cG$ is a Lagrangian sub-bundle (of rank $m$). 
\item Let $\Sht_{GU(n)}^r$ be the moduli stack of rank $n$ similitude Hermitian shtukas with $r$ legs. 
\end{itemize}
In \cite[\S 4]{FYZ2}, we constructed the \emph{higher theta series}
\[
\wt Z^{r}_{m}: \Bun_{\wt P_{m}}(k) \to \CH_{r(n-m)}(\Sht^{r}_{GU(n)}),
\]
a function assigning a cycle class on $\Sht^r_{GU(n)}$ to every quadruple $(\cG,\frM, h,\cE)$ as above.

\subsubsection{The Modularity Conjecture} 

 The map $\Bun_{\wt P_{m}}(k) \rightarrow \Bun_{GU(2m)}(k)$, given by forgetting the Lagrangian sub-bundle $\cE \subset \cG$, is surjective; and \cite[Modularity Conjecture 4.15]{FYZ2} predicts that $\wt{Z}^r_m$ descends through this map to a function $Z^{r}_{m}: \Bun_{GU(2m)}(k) \to \CH_{r(n-m)}(\Sht^{r}_{GU(n)})$, as in the diagram below. 
\[
\begin{tikzcd}
\Bun_{\wt P_{m}}(k)  \ar[d, twoheadrightarrow] \ar[dr, "\wt Z^{r}_{m}"] \\
 \Bun_{GU(2m)}(k) \ar[r, dashed, "Z^{r}_{m}"'] &  \CH_{r(n-m)}(\Sht^{r}_{GU(n)})
\end{tikzcd}
\]
In other words, the Modularity Conjecture says that the function $\wt Z^{r}_{m}$, which a priori depends on $(\cG, \frM, h, \cE)$, is actually independent of the Lagrangian sub-bundle $\cE \subset \cG$. 

\subsubsection{The generic locus}\label{sssec: generic locus} The stack $\Sht^{r}_{GU(n)}$ admits a ``leg map'' $\Sht^{r}_{GU(n)} \rightarrow (X')^r$. Let $\y=\Spec F'\to X'$ be the generic point. Let $\y^{r}=\Spec(F'\ot_{k}\cdots\ot_k F')\to (X')^{r}$. Note that $\y^{r}$ contains the generic point of $(X')^{r}$ but it also contains many more points such as the generic point of the diagonal $X' \inj (X')^r$. We refer to $\Sht^{r}_{GU(n)} \times_{(X')^r} \eta^{r}$ as the ``generic locus'' of $\Sht^r_{GU(n)}$. 

\subsubsection{The generic modularity theorem}We have a restriction map 
\[
\CH_{r(n-m)}(\Sht^{r}_{GU(n)}) \rightarrow \CH_{r(n-m)}(\Sht^{r}_{GU(n)} \times_{(X')^r} \eta^{r}).
\]

\begin{thm}[Modularity on the generic locus]\label{thm: main}
The composition 
\[
\Bun_{\wt P_{m}}(k)  \xrightarrow{\wt Z^{r}_{m}}  \CH_{r(n-m)}(\Sht^{r}_{GU(n)}) \rightarrow \CH_{r(n-m)}(\Sht^{r}_{GU(n)} \times_{(X')^r} \eta^{r})
\]
descends through $\Bun_{\wt P_{m}}(k)  \surj \Bun_{GU(2m)}(k)$. In other words, its value on $(\cG, \frM, h, \cE) \in \Bun_{\wt P_{m}}(k) $ is independent of the Lagrangian sub-bundle $\cE \subset \cG$. 
\end{thm}

\begin{remark}
For application to the Kudla program as outlined in \S \ref{ssec: nt background}, the generic form of modularity established in Theorem \ref{thm: main} is sufficient for arithmetic theta lifting and the arithmetic inner product formula, according to the paradigm of \cite{LL1, LL2}. 
\end{remark}

\subsection{Discussion}\label{ssec: discussion} We discuss some related results to Theorem \ref{thm: main}. Modularity of arithmetic theta series in the Chow group of Shimura varieties (of the generic fiber), which is analogous to the case $r=1$ of Theorem \ref{thm: main}, is known for unitary Shimura varieties of signature $(n-1,1)$ and orthogonal Shimura varieties of signature $(n-2,2)$ when the underlying CM field is norm-Euclidean, thanks to work of Borcherds \cite{Bor99}, Zhang \cite{Zh09}, and Bruinier--Westerholt-Raum \cite{BWR15}.\footnote{These results were then refined to obtain modularity in the Chow groups of integral models, for unitary Shimura varieties of signature $(n-1,1)$ by Bruinier--Howard--Kudla--Rapoport--Yang \cite{BHKRY1} in the divisor case, and for orthogonal Shimura varieties of signature $(n-2,2)$ by Howard--Madapusi \cite{HM22} in all codimensions.}  However, the methods behind those results seem completely inapplicable in our situation (there is a brief discussion of this in \cite[\S 1.2]{FYZ3}), hence the proofs are essentially disjoint, even at the level of ideas. 

The cohomological version of Theorem \ref{thm: main}, modularity of the $\ell$-adic realization on the generic locus, is established in \cite{FYZ3}, and our proof very much builds on the ideas of \emph{loc. cit.}, as will be discussed further in \S \ref{ssec: intro proof}. In the Shimura variety context, the analogous modularity of the Betti realization was proved much earlier in work of Kudla--Millson \cite{KM90}, and in complete generality. By contrast, modularity in the Chow group of the generic fiber is more difficult and the known results more restrictive. This is because there are various tools for accessing cohomology, whereas modularity in the Chow group amounts to the fundamentally difficult problem of producing motivic data. The cases in which this has been accessible are limited to those where the Shimura variety supports special cycles that are divisors, for then one can use the (intricate!) theory of Borcherds lifting to write down explicit functions which attest to the necessary relations between divisors. Beyond the case where the Shimura variety supports divisors, nothing seems to be known towards modularity; the motivic data that must be produced is of a more complicated and subtler nature. We note that Kudla has observed in \cite{Kud21} that assuming (presumably difficult) conjectures of Beilinson--Bloch on the injectivity of Abel--Jacobi maps, and due to an incidental miracle of Hodge diamonds, the modularity in the orthogonal case (on the generic fiber) is implied by modularity in Betti cohomology; Maeda proved an analogous statement in the unitary case \cite{Mae22}. We do not know if any such phenomenon occurs in the function field context; our proof completely bypasses such questions. 

In the function field context there is no analogue of Borcherds lifting, and we do not know any direct way to write down the explicit functions that attest to the necessary relations. We note that the arithmetic theta series live in codimension $mr$, so when $r>1$ we are beyond the existence of codimension 1 special cycles, putting us in the realm where the modularity is completely unknown in the Shimura variety context. Nevertheless, our proof produces the necessary relations. Implicitly we are constructing motivic data; for example, for $r=1$ and $m=1$ our proof implies the existence of certain rational functions on moduli of shtukas that should perhaps be considered as the correct analogues of Borcherds products. But instead of writing down this motivic data explicitly, we produce it as the output of a machine that we call the \emph{motivic sheaf-cycle correspondence}. 

It could be interesting to investigate potential applications of the motivic data produced in this way, such as the functions produced in the $m=1, r=1$ case, which seem to play the role of Borcherds products on moduli of shtukas. This approach to constructing units is vaguely reminiscent of the approach to modular units and Beilinson--Flach classes via the Manin--Drinfeld Theorem.

\subsection{Commentary on the proof}\label{ssec: intro proof}
The proof of Theorem \ref{thm: main} is patterned on the proof of \cite[Theorem 1.1.1]{FYZ3}, which established modularity after $\ell$-adic realization. In fact we refer the reader to the Introduction and \S 2.4 of \emph{loc. cit}. for a guide to the strategy of the proof. We will only describe the improvements of this present paper relative to \cite{FYZ3}. 

\subsubsection{Motivic sheaf-cycle correspondence} The classical sheaf-function correspondence is a formalism for extracting functions from sheaves via the trace of Frobenius. One innovation of \cite{FYZ3} is a ``sheaf-cycle'' correspondence that extracts the $\ell$-adic realization of cycles from sheaves. More precisely, the strategy of \cite{FYZ3} is to express the $\ell$-adic realization of higher theta series as the ``trace'' of a cohomological correspondence between $\ell$-adic sheaves, and then to deduce modularity from some appropriate form of modularity for cohomological correspondences. By definition the trace operation produces elements of Ext groups between $\ell$-adic sheaves, hence can only see $\ell$-adic realizations.

In this paper we upgrade the $\ell$-adic sheaf-cycle correspondence of \cite{FYZ3} to a \emph{motivic sheaf-cycle correspondence} that directly extracts Chow classes from a suitable notion of motivic sheaves.
Specifically, we will work with the triangulated category of motivic sheaves introduced by V.~Voevodsky in the course of his proofs of the Norm Residue isomorphism\footnote{This was conjectured by Bloch--Kato, but we avoid calling it the Bloch--Kato Conjecture since it is completely different from the other Bloch--Kato Conjecture which appeared in \S \ref{ssec: nt background}.} and the Beilinson--Lichtenbaum Conjecture. After further developments by J.~Ayoub, D.-C.~Cisinski, F.~D\'eglise, F.~Morel, and others, we have at our disposal a robust theory of triangulated categories of motivic sheaves over arbitrary base schemes, equipped with Grothendieck's six operations.
We may regard Voevodsky's category as the derived category of the hypothetical \emph{abelian} category of perverse motivic sheaves; while Grothendieck's Standard Conjectures obstruct the existence of this abelian category, or equivalently of the perverse motivic t-structure (see \cite{Bei12}), Voevodsky's insight was that the putative derived category can in fact be constructed independently of intractable questions about algebraic cycles.  For the purposes of our construction of a motivic sheaf-cycle correspondence, the key property of the derived category of motives is that the Ext groups calculate (higher) Chow groups.

We discuss some differences between motivic sheaves and $\ell$-adic sheaves. In the language of Ayoub \cite{Ay14}, $\ell$-adic sheaves are a ``transcendental'' invariant: they have strong finiteness properties, behave well in families, and are relatively computable; but their relationship to algebraic cycles is tenuous (highly conjectural at best). By contrast, motivic cohomology is what Ayoub calls an ``algebro-geometric invariant'', which is built directly out of objects of interest in algebraic geometry (e.g., algebraic cycles), but behaves ``chaotically'': it does not have good finiteness properties, it varies violently in families, and it is not amenable to computation. In particular, it is (a priori) ill-defined to form the trace of an endomorphism on motivic cohomology groups, since these groups are usually infinite-dimensional. This presents a challenge for the sheaf-cycle correspondence, which is implemented by formation of trace. 

The solution is to expand the meaning of the trace. Dold--Puppe \cite{DP80} codified the ``trace of an endomorphism of a dualizable object in a symmetric monoidal category'' as a generalization of the trace in linear algebra. In the category of vector spaces (over a given field), the dualizable objects are precisely the finite-dimensional vector spaces, and the \emph{categorical trace} in the sense of Dold--Puppe is equivalent to the usual trace. However, in a more general symmetric monoidal category, the trace is divorced from its linear algebraic origins, and can be formed without finite-dimensionality conditions. This is precisely how we make sense of the trace for motivic sheaves. We define a motivic version of the Lu--Zheng 2-category from \cite{LZ22}, in which \emph{universally strongly locally acyclic} motives form dualizable objects; dualizability provides the analogue of a finiteness property which allows to form the trace even without finite-dimensionality of motivic cohomology. We use Lu--Zheng's approach to prove a ``relative Verdier--Lefschetz formula'', which supplies compatibility of the motivic sheaf-cycle correspondence with proper pushforwards. We also introduce a dual version of the Lu--Zheng category to prove a ``relative local term formula'', which supplies compatibility of the motivic sheaf-cycle correspondence with smooth pullbacks; we note that derived algebraic geometry is crucial to formulate the pullback compatibility.

\subsubsection{Motivic Fourier analysis} The modularity of $\ell$-adic cohomological correspondences in \cite{FYZ3} comes from a derived generalization of Deligne--Laumon's $\ell$-adic Fourier transform, which allows to execute a sheaf-theoretic version of Poisson's argument for modularity of the classical theta function. To carry out such arguments, we need to develop a theory of this ``derived Fourier analysis'' for \emph{motivic sheaves}. 

While the $\ell$-adic Fourier transform requires an Artin-Schreier sheaf, hence only exists on spaces in characteristic $p$, there is a variant of Fourier analysis which is more robust, in the sense of being defined in very general geometric and sheaf-theoretic contexts. This variant is based on Laumon's theory of the \emph{homogeneous Fourier transform} \cite{Laumon}. In anticipation of future applications, we invest effort into developing the homogeneous Fourier transform in great generality in \S \ref{sec: derived homogeneous FT}, encompassing motivic sheaves but also all other known 6-functor formalisms. It turns out that the homogeneous Fourier theory is enough for our applications to the modularity of higher theta functions. (Jakob Scholbach informed us that he is writing up a motivic lift of the Deligne-Laumon Fourier theory; this would also be enough for our applications in the present paper, but at least one of the authors has followup applications in mind that require our more general context.) 

 We note that, as in \cite{FYZ3}, we need a \emph{derived} expansion of this theory, which encompasses generalizations of vector bundles called derived vector bundles. This derived generalization presents significant technical difficulties, for which we refer to \S \ref{sec: derived homogeneous FT} and \cite[\S 6 and Appendix A]{FYZ3} for further discussion.

\subsection{Outline} We give a brief outline of the paper. \S \ref{sec: motivic sheaves} establishes some preliminaries on motivic sheaves, especially the notion of universal strong local acyclicity (USLA) and its consequences. Then \S \ref{sec: coh corr} -- \S \ref{sec: local terms} develop the motivic sheaf-cycle correspondence and the tools to calculate with it. Next \S \ref{sec: derived homogeneous FT} constructs the derived homogeneous Fourier transform and establishes its properties, and \S \ref{sec: Fourier for motives} studies its interaction with the motivic sheaf-cycle correspondence. Finally, \S \ref{sec: generic modularity} assembles everything to prove Theorem \ref{thm: main}. 

\subsection{Acknowledgements} 
TF thanks Aravind Asok and Elden Elmanto for conversations about motivic sheaves, and especially Zhiwei Yun and Wei Zhang for their many insights shared in the collaboration on \cite{FYZ, FYZ2, FYZ3}, on which this paper builds. 

AK thanks Denis-Charles Cisinski and Tasuki Kinjo for invaluable discussions about Fourier transforms, and Fangzhou Jin for answering questions about his paper \cite{Jin23}.

We are grateful to Joseph Ayoub, Benedict Gross, Chao Li, Yifeng Liu, Jakob Scholbach, and Wei Zhang for comments and corrections on a draft of this paper. 

TF was supported by the NSF (DMS-2302520). AK acknowledges support from the grants AS-CDA-112-M01 (Academia Sinica), NSTC 110-2115-M-001-016-MY3, and NSTC 112-2628-M-001-0062030.

\section{Notation and conventions} The notation is consistent with that of \cite{FYZ3} (and therefore inconsistent with \cite{FYZ, FYZ2} in some ways, as noted there.) 

  \subsection{Spaces}

    Unless noted otherwise, we always work in the category of derived Artin stacks. Hence when we say ``Cartesian square'' we mean what might be called ``derived Cartesian square'' (sometimes we keep the adjective ``derived'' for emphasis), unless noted otherwise (the exception is in \S \ref{sec: local terms}).

    For a derived Artin stack $A$, we denote by $A_{\cl}$ its classical truncation.
    By definition, a map of derived Artin stacks is a \emph{closed embedding} or \emph{proper} if the induced map of classical truncations has this property.
    For example, the inclusion of the classical truncation $A_{\cl} \hook A$ is a closed embedding.

  \subsection{Perfect complexes}\label{ssec: notate perf}

    Let $S$ be a derived Artin stack.
    We let $\Perf(S)$ be the $\infty$-category of perfect complexes on $S$  (see e.g. \cite[\S 3]{GRI}).
    For $\cE \in \Perf(S)$ we write $\cE^* = \cRHom_S(\cE, \cO_S)$ for the linear dual of $\cE$.

    By a \emph{cochain complex of locally free sheaves} on $S$ of amplitude $[a,b]$, we will mean a diagram
    \begin{equation*}
      \cE^a
      \xrightarrow{d^{a}} \cdots
      \xrightarrow{d^{-2}} \cE^{-1}
      \xrightarrow{d^{-1}} \cE^0
      \xrightarrow{d^0} \cE^1
      \xrightarrow{d^1} \cdots
      \xrightarrow{d^{b-1}} \cE^b
    \end{equation*}
    in $\Perf(S)$ where each $\cE^i$ is of tor-amplitude $[0,0]$, together with null-homotopies $d^{i} \circ d^{i-1} \cong 0$ for all $i$.
    A \emph{morphism} of cochain complexes $\phi \co \cE^\bullet \to \cF^\bullet$ is a collection of morphisms $\phi^i \co \cE^i \to \cF^i$ (extending by zero if the complexes are not of the same amplitude), together with a homotopies $\phi^{i+1} \circ d^i \cong d^i \circ \phi^i$ as well as compatibilities between the null-homotopies $\phi^{i+1} \circ (d^{i} \circ d^{i-1}) \cong 0$ and $(d^{i} \circ d^{i-1}) \circ \phi^{i-1} \cong 0$.

    By taking iterated cofibres, a cochain complex $\cE^\bullet$ gives rise to a perfect complex $\cE \in \Perf(S)$ of tor-amplitude $[a,b]$ (note that we are using \emph{cohomological} grading even for tor-amplitude).
    Similarly, a morphism of cochain complexes gives rise to a morphism of perfect complexes.
    We refer to the cochain complex $\cE^\bullet$ as a \emph{global presentation} for $\cE$.
    More generally, we refer to a diagram of cochain complexes as a global presentation for the induced diagram of perfect complexes.

    When $S$ is affine, or more generally admits the derived resolution property in the sense of \cite[\S 1.7]{kstack}, every perfect complex admits a global presentation.
    For a general derived Artin stack $S$, every perfect complex $\cE \in \Perf(S)$ admits a global presentation smooth-locally on $S$.

  \subsection{Cotangent complexes}

    For a map $f \co X \to Y$ of derived Artin stacks, we denote by $\bL_f \coloneqq \bL_{X/Y} \in \Perf(X)$ the relative cotangent complex.

    Let $f \co X \to Y$ be a map of derived Artin stacks that is locally finitely presented on classical truncations.
    The map $f$ is \emph{étale} if the relative cotangent complex $\bL_f$ vanishes (i.e., is isomorphic to $0 \in \Perf(X)$).
    The map $f$ is \emph{smooth} (resp. \emph{quasi-smooth}) if the relative cotangent complex $\bL_f$ is perfect of tor-amplitude $[0, \infty)$ (resp. $[-1, \infty)$).
    
    Note that unlike properness, these properties cannot in general be detected on classical truncations.
    Moreover, while a smooth map is also smooth on classical truncations, quasi-smoothness is typically destroyed by classical truncation.
    See \cite[\S 2]{KhanRydh} for some background on quasi-smoothness.

    When $\bL_f$ is perfect, we write $d(f)$ for its virtual rank (or Euler characteristic), and call it the \emph{relative dimension} of $f$.

  \subsection{Derived vector bundles}\label{ssec: notate dvb}

    Let $S$ be a derived Artin stack.
    
    Given a perfect complex $\cE \in \Perf(S)$, we denote by $\V(\cE)$ the derived stack of sections of $\cE$, as in \cite[\S 6.1.1]{FYZ3}.
    We refer to \emph{$\Tot(\cE)$} as the \emph{derived vector bundle} associated with $\cE$.\footnote{We caution that some other sources (including \cite{khan2023derived}) use the dual convention, using Grothendieck's $\bV(-)$ construction.}
    In terms of the functor of points, $\Tot(\cE)$ is the derived stack over $S$ sending an $S$-scheme $u : T \to S$ to the mapping space $\Map_{\QCoh(T)}(\cO_T, u^*\cE)$.
    
    If $\cE$ is of tor-amplitude $\ge 0$, then we have
    \begin{equation*}
      \Tot(\cE) \cong \ul{\smash{\Spec}}_S(\Sym_{\cO_S}(\cE^*)).
    \end{equation*}
    Thus in that case the projection $\Tot(\cE) \to S$ is affine (but smooth if and only if $\cE$ is of tor-amplitude $[0,0]$).

    On the other hand, if $\cE$ is of tor-amplitude $\le 0$, then the projection $\Tot(\cE) \to S$ is smooth (but representable if and only if $\cE$ is of tor-amplitude $[0,0]$).

    The $\infty$-category $\DVect(S)$ of \emph{derived vector bundles} over $S$ is the essential image of the fully faithful functor $\cE \mapsto \V(\cE)$ from $\Perf(S)$ to the $\infty$-category of derived stacks over $S$ with $\bG_m$-action.    

    Throughout we use calligraphic letters such as $\cE$ for perfect complexes, and Roman letters such as $E$ for the corresponding total spaces. We will denote the dual derived vector bundle to $E=\Tot(\cE)$ by $\wh{E} = \Tot(\cE^*)$.

    Using the equivalence $\Tot(-) \co \Perf(S) \xrightarrow{\sim} \DVect(S)$, we can make sense of global presentations of (diagrams of) derived vector bundles just as in \S \ref{ssec: notate perf}.
    That is, a global presentation for $E = \Tot(\cE)$ is a global presentation for $\cE \in \Perf(S)$.

  \subsection{$\infty$-categories}

    In an $\infty$-category $\msf{C}$, we use the notation $\Map(c, c')$ for the mapping space between objects $c,c' \in \msf{C}$. We use the notation $\Hom(c,c') \coloneqq \pi_0 \Map(c,c')$, which is the group of morphisms from $c$ to $c'$ in the homotopy category of $\msf{C}$. We denote $\Ext^i(c,c') \coloneqq \Hom(c, c'[i])$.

  \subsection{Motives}

    We refer to \S \ref{ssec: dmot} for the precise definition of motivic sheaves adopted in this paper, and then \S \ref{ssec: notation for motives} for additional relevant notation.

\section{Motivic sheaf theory}\label{sec: motivic sheaves}

In this section we establish some general material on motivic sheaves and motivic cohomology. We define the notion of \emph{(universally) strongly locally acyclic} motivic sheaves and their properties; this part is similar to work of Jin \cite{Jin23} which is itself a motivic version of work of Lu--Zheng \cite{LZ22}. However, these earlier works focus on the case of schemes while for applications we need the generality of derived Artin stacks, so we formulate the statements in this generality, and give proofs when they need to be modified from the case of schemes. 

\subsection{The derived category of motives}\label{ssec: dmot}

For a derived Artin stack $S$, we have the stable \inftyCat $\Dmot{S}$ of motivic sheaves on $S$ with rational coefficients.

Recall that for a scheme $S$, the motivic stable homotopy category $\SH(S)$ along with the six-functor formalism for the assignment $S \mapsto \SH(S)$ was constructed by Morel and Voevodsky \cite{Voe98, MV99,deligne2001voevodsky} and developed further by Ayoub \cite{Ay07a, Ay07b} and Cisinski--Déglise \cite{CD19} (see also \cite[App.~C]{HoyoisLefschetz} or \cite{khansix} for non-noetherian bases).
The six-functor formalism descends to the étale-localized and rationalized categories $\SH_{\et}(S; \Q)$, and we take $\Dmot{-} \coloneqq \SH_{\et}(-; \Q)$ by definition on schemes.\footnote{%
  Since some older references operate with triangulated categories or model categories, we clarify that we will always use the $\infty$-categorical incarnation of $\Dmot{S}$.
  See e.g. \cite{khansix} for the construction of the six operations at the $\infty$-categorical level.
}
This is also known as Ayoub's category $\DA{S}$ of étale motives with rational coefficients (see \cite{Ay14} for an introduction).
This category has been defined and studied in various other guises, which are described and compared in \cite{CD19}:
\begin{itemize}
\item (Beilinson motives) By \cite[Theorem 16.2.13]{CD19}, $\Dmot{S}$ is equivalent to the category of \emph{Beilinson motives} over $S$ in the sense of \cite[\S 14]{CD19}.
In particular, if $S$ is noetherian and finite-dimensional, then by \cite[Theorem 5.2.2]{CD16} $\Dmot{S}$ is equivalent to the category $\DM_h(S; \Q)$ of \emph{$h$-motives} (with rational coefficients).
\item (Morel motives) By \cite[Theorem 16.2.18]{CD19}, $\Dmot{S}$ is equivalent to the category of \emph{Morel motives} over $S$ in the sense of \cite[\S 16.2]{CD19}.
\item (Voevodsky motives) If $S$ is excellent and geometrically unibranch, then by \cite[Theorem 16.1.4]{CD19} $\Dmot{S}$ is equivalent to the category $\DM(S; \Q)$ of \emph{Voevodsky motives} over $S$ (with rational coefficients).
\item ($H\bQ$-linear motivic spectra) For a commutative ring $\Lambda$, let $H\Lambda_S \in \SH(S)$ denote the $\Lambda$-linear motivic Eilenberg--MacLane spectrum as defined in \cite{zbMATH07015021}.
For $\Lambda=\bQ$, $H\bQ_S$ is isomorphic to the Beilinson motivic cohomology spectrum of \cite[Definition~14.1.2]{CD19} by \cite[Theorem~7.14]{zbMATH07015021}.
In particular, by \cite[Theorem~14.2.9]{CD19}, the \inftyCat $\on{D}_{H\Lambda}(S)$ of modules over $H\Lambda_S$ is equivalent to the \inftyCat of Beilinson motives over $S$, and hence to $\Dmot{S}$.
\end{itemize}

The generalization of $\Dmot{-}$ to derived algebraic spaces and derived Artin stacks is developed in \cite[Appendix A]{KhanI}. To explicate this, we remark that $\SH(S)$ and hence $\Dmot{S}$ is invariant under passing to the classical truncation $S_\cl$ by \cite{khanlocalization}. Then $\Dmot{-}$ is extended from derived schemes to derived Artin stacks by right Kan extension. Explicitly, this means that if $S$ is a derived Artin stack then
\[
\Dmot{S} = \limit \Dmot{T}
\]
where the limit is over the category of smooth morphisms $T \rightarrow S$ from derived schemes $T$.
If $T \twoheadrightarrow S$ is a smooth \emph{atlas} from a derived scheme, then $\Dmot{S}$ agrees with the category of Cartesian sheaves on the simplicial derived scheme $T_{\bu} = \{ T \times_S \ldots \times_S T\}$.
The six-functor formalism also extends to derived Artin stacks by \cite[Thm.~A.5]{KhanI}.

\subsection{Notations for motives}\label{ssec: notation for motives}
For a derived Artin stack $A$, we denote by $\Qsh{A}$ (or just $\Q$ if the context is clear) the unit of the symmetric monoidal category $\Dmot{A}$. 

The category $\Dmot{-}$ contains a ``Tate motive'' $\Q(1)$. For $\cK \in \Dmot{A}$, we write $\cK \tw{i} \coloneqq \cK[2i](i)$ for the indicated shift and Tate twist. 

For $\cK, \cK' \in \Dmot{A}$, we abbreviate 
\[
\Hom_A(\cK, \cK') \coloneqq \Hom_{\Dmot{A}}(\cK, \cK'). 
\]

For a map $f \co A \rightarrow S$ of derived Artin stacks, we denote by $\DD_{A/S}(-)$ the relative Verdier dual functor, 
\[
\DD_{A/S}(\cK) \coloneqq \cRHom_A(\cK, f^! \Qsh{S}). 
\]
We also abbreviate $\DD_{A/S} \coloneqq \DD_{A/S}(\Qsh{A})$ for the relative dualizing complex of $f$. 

\subsection{Geometric motives}\label{ssec: geometric motives} Given a smooth map of derived Artin stacks $f \co T \rightarrow S$, there is a functor
\[
f_{\sh} \co \Dmot{T} \rightarrow \Dmot{S}
\]
which is \emph{left} adjoint to the pullback $f^* \co \Dmot{S} \rightarrow \Dmot{T}$. If $S$ is a derived scheme, the subcategory $\Dmotg{S} \subset \Dmot{S}$ of \emph{geometric motives} is the thick subcategory generated by $f_{\sh} \Qsh{T}\tw{i}$ as $f  \co T \rightarrow S$ ranges over smooth morphisms of derived schemes and $i$ ranges over all integers.

If $S$ is a derived stack, then we say that a motive $\cK \in \Dmot{S}$ is \emph{geometric} if it is geometric after pullback to some (equivalently, any) atlas $S' \surj S$ where $S'$ is a derived scheme. We denote by $\Dmotg{S} \subset \Dmot{S}$ the full subcategory of geometric motives. 

\begin{example}
For any derived Artin stack $A$, the unit $\Qsh{A} \in \Dmotg{A}$ is geometric. 
\end{example}

\begin{remark}[Preservation under six functors]\label{rem: preservation of constructibility}
For $f \co S' \rightarrow S$ a map of derived schemes of finite type over a quasi-excellent scheme, the property of being geometric is preserved by the functors $f_!, f_*, f^*, f^!$ (see \cite[Theorem~15.2.1]{CD19}). It then follows that for a map $f \co A' \rightarrow A$ of derived Artin stacks locally of finite type over a quasi-excellent scheme, geometricity is  preserved by the functors $f^*$ and $f^!$; and geometricity is preserved by the functors $f_!$ and $f_*$ \emph{if} $f$ is representable in derived schemes. 

Finally, we note that $\cRHom(-,-)$ and $- \otimes -$ preserve geometric motives on schemes, and are compatible with smooth base change, hence they preserve geometric motives on derived Artin stacks. 
\end{remark}

\subsection{The effective homotopy t-structure}

For a derived scheme $S$, let $\Dmot{S}^{\le 0} \subseteq \Dmot{S}$ denote the full subcategory generated under colimits and extensions by objects of the form $a_!a^!(\Q_S)$, for $a \co X \to S$ a smooth morphism from a scheme.
This forms the connective part of the \emph{effective homotopy t-structure} on $\Dmot{S}$ (see \cite[Sect.~13, App.~B]{BachmannHoyois}).
The coconnective part $\Dmot{S}^{\ge 0}$ is thus spanned by those $\cK \in \Dmot{S}$ for which the groups
\begin{equation*}
  \rH^{-n}(X; \cK) \cong \Hom_{\Dmot{S}}(\Q_S[n],a_*a^*(\cK)) \cong \Hom_{\Dmot{S}}(a_!a^!(\Q_S)[n],\cK)
\end{equation*}
vanish for all $n>0$ and all smooth morphisms $a \co X \to S$ with $X$ a scheme.

For a derived Artin stack $S$, we say that an object $\cK \in \Dmot{S}$ belongs to $\Dmot{S}^{\le 0}$, resp. $\Dmot{S}^{\ge 0}$, if $u^*(\cK)$ belongs to $\Dmot{U}^{\le 0}$, resp. $\Dmot{U}^{\ge 0}$ for some smooth atlas $u \co U \twoheadrightarrow S$.
The proof of \cite[Prop.~5.3]{equilisse} applies verbatim to show that this defines a t-structure on $\Dmot{S}$.

\begin{lem}\label{lem:Qplus heart}
  For every derived Artin stack $S$ locally of finite type over a field $k$, the unit $\Q_S \in \Dmot{S}$ belongs to the heart of the effective homotopy t-structure.
\end{lem}
\begin{proof}
  We may assume that $S$ is a scheme.
  We have $\Q_S \in \Dmot{S}^{\le 0}$ by definition, so it remains to show that for every scheme $Y$ which is smooth over $S$, the spectrum
  \begin{equation*}
    R\Gamma(Y, \Q_Y)
  \end{equation*}
  is connective, i.e., $\rH^{-n}(Y; \Q_Y) \cong 0$ for $n>0$.
  For any prime $\ell \ne \on{char}(k)$, there exists by de Jong--Gabber an $\ell$dh-hypercover $Y'_\bullet \to Y$ where each $Y'_n$ is a regular scheme.
  Since motivic cohomology with rational coefficients satisfies $\ell$dh descent by \cite[Thm.~1.2]{Geisser}, and connectivity is stable under limits, we may assume that $Y$ is regular.
  In this case we have
  \begin{equation*}
    \rH^{-n}(Y; \Q)
    \cong \rH^{-n}(Y; \mathrm{KGL}_\Q^{(0)})
    \cong \Gr^0_\gamma (\on{K}_n(Y)_\Q)
  \end{equation*}
  where $\on{K}_n(Y)_\Q \coloneqq \pi_n(\on{K}(Y)) \otimes \Q$ and $\on{K}(Y)$ is the algebraic K-theory spectrum of $Y$ (see \cite[\S 14.1]{CD19}).
  But $\on{Fil}^1_\gamma \on{K}_n(Y)_\Q = \on{K}_n(Y)_\Q$ holds for $n>0$ by definition of the augmented $\lambda$-ring structure on $\on{K}_n(Y)_\Q$ (see e.g. \cite[IV, \S 5, p.~345]{WeibelK}).
\end{proof}

\subsection{Chow groups as motivic Borel-Moore homology}\label{ssec: Chow}
Let $A$ be a derived Artin stack locally of finite type over a field $\F$.

We define the \emph{Chow groups} of $A$ (with rational coefficients) by
\[
\CH_{i}(A) \coloneqq \rH^{-2i}(A; \pi^! \Qsh{\Spec(F)} (-i) )
\cong \rH^0(A; \pi^!\Qsh{\Spec(F)}\vb{-i}), \quad \text{ for } i \in \bZ
\]
where $\pi \co A \rightarrow \Spec (\F)$ is the structural morphism.
This definition agrees with the rationalization of the classical definition of Chow groups \emph{under assumptions} that $A$ is ``reasonable''.\footnote{Our point of view is that when $A$ is unreasonable, then our definition of $\CH_i(A)$ is the ``correct'' one, being well-behaved from various technical perspectives.} More precisely, according to \cite[Example 2.10]{KhanI}, when $A$ is a classical $1$-Artin stack of finite type over $k$ with affine stabilizers, this recovers the Chow group (with $\Q$-coefficients) of Kresch \cite{Kr99}. If $A$ is a derived Artin stack, then by the derived invariance of $\Dmot{A}$, the inclusion of the classical truncation $A_{\cl} \inj A$ induces isomorphisms $\CH_i(A_{\cl}) \cong \CH_{i}(A)$. 

We define the \emph{Chow cohomology groups} of $A$ (with rational coefficients) by
\[
\CH^i(A) \coloneqq \rH^{2i}(A; \Qsh{A}(i)) \cong \rH^0(A; \Qsh{A}\vb{i}), \quad \text{ for } i \in \bZ.
\]
We caution that these map to, but are typically \emph{not} the same as, Fulton's operational Chow cohomology groups \cite[\S 17]{Ful98}, even for classical quasi-projective schemes (unless $A$ is smooth).

More generally, for a locally of finite type morphism $f \co A \rightarrow B$ of derived Artin stacks over a field $\F$, we define the \emph{relative Chow groups of $f$} to be
\[
\CH_{i}(A/B) \coloneqq \rH^{-2i}(A; f^! \Qsh{B}(-i))
\cong \rH^0(A; f^!\Qsh{B}\vb{-i}) , \quad \text{for } i \in \bZ.
\]
We have $\CH_{i}(A/A) \simeq \CH^{-i}(A)$ for all $i$.
Again, these are a refinement of the operational or bivariant Chow groups of \cite[\S 17]{Ful98}.

\subsection{Functoriality of Chow groups}
\label{ssec:Gys}

\sssec{Proper pushforward}

The Chow groups are covariantly functorial with respect to proper morphisms.
That is, if $f \co A \to B$ is a proper morphism of derived Artin stacks, then we have pushforward maps
\begin{equation}
f_* \co \CH_i(A) \to \CH_i(B)
\end{equation}
and more generally $f_* \co \CH_i(A/C) \to \CH_i(B/C)$ if $f$ is defined over some $C$.
In terms of the six functors, these are induced by the natural transformation $f_* f^! \to \id$, counit of the adjunction $(f_*, f^!)$.

\sssec{Gysin pullback}

Let $f \co A \rightarrow B$ be a quasi-smooth map of derived Artin stacks, of relative dimension $d(f)$.
There are (virtual) Gysin pullback maps
\begin{equation}\label{eq: Gysin pullback}
f^! \co \CH_i(B) \to \CH_{i+d(f)}(A),
\end{equation}
and more generally $f^! : \CH_i(B/C) \to \CH_{i+d(f)}(A/C)$ if $B$ is defined over $C$.
These are functorial and satisfy a base change formula with respect to proper pushforwards.

The maps \eqref{eq: Gysin pullback} are induced by a natural transformation
\begin{equation}\label{eq: gys}
\mrm{gys}_f \co f^* \rightarrow f^! \tw{-d(f)}
\end{equation}
called the Gysin transformation, constructed in \cite[\S 3]{KhanI}.
It satisfies various natural compatibilities detailed in \cite[\S 3.2]{KhanI} or \cite[\S 3.4]{FYZ3}.
For example, when $f$ is smooth, the Gysin transformation recovers the Poincaré duality isomorphism $f^* \cong f^! \tw{-d(f)}$.

In particular, one has a relative (virtual) fundamental class
\begin{equation}
[A/B] \coloneqq [f] \in \CH_{d(f)}(A/B)
\end{equation}
defined as the Gysin pullback of the unit in $\CH_0(B/B) \simeq \CH^0(B)$.
Equivalently, it is determined by the morphism
\begin{equation}\label{eq: Gysin}
\Qsh{A} \cong f^* \Qsh{B} \xrightarrow{[f]} f^! \Qsh{B} \tw{-d(f)}
\end{equation}
obtained by evaluating the Gysin transformation \eqref{eq: gys} on $\Qsh{B}$.

\subsection{USLA motives} Let $S$ be a derived Artin stack locally of finite type over a field. 

\begin{defn}\label{def: USLA}
Let $f \co A \rightarrow S$ be a map of derived Artin stacks. Following \cite[Definition 3.1.1]{Jin23}, we say that $\cK_A \in \Dmot{A}$ is \emph{strongly locally acyclic} (SLA) over $S$ if for any schematic map of derived Artin stacks $g \co T \rightarrow S$, inducing the Cartesian square
\begin{equation}\label{eq:prochronic}
\begin{tikzcd}
B \ar[r, "g'"] \ar[d, "f'"] &  A \ar[d, "f"] \\
T \ar[r, "g"] & S
\end{tikzcd}
\end{equation}
and any $\cK_T \in \Dmot{T}$, the canonical map
\begin{equation}\label{eq: USLA}
\cK_A \otimes f^* g_* \cK_T \rightarrow g'_* ((g')^* \cK_A \otimes (f')^* \cK_T)
\end{equation}
is an isomorphism. 

We say that $\cK_A \in \Dmot{A}$ is \emph{universally strongly locally acyclic} (USLA) over $S$ if for any morphism $S' \rightarrow S$, the $*$-pullback of $\cK_A$ to $S' \times_S A'$ is SLA over $S'$. 
\end{defn}



The property of being (U)SLA can be checked locally in the smooth topology on the source and target.

\begin{lemma}\label{lem: USLA smooth local}
Maintain the notation of Definition \ref{def: USLA}. 

(1) Let $h \co A' \rightarrow A$ be a smooth, surjective morphism of derived Artin stacks. Then $\cK_A$ is (U)SLA over $S$ if and only if $h^* \cK_A$ is (U)SLA over $S$.  

(2) Let $h \co S' \rightarrow S$ be a smooth, surjective morphism of derived Artin stacks. Let $h_A$ be the base change of $h$ to $A$. Then $\cK_A$ is (U)SLA over $S$ if and only if $h^*_A \cK_A$ is (U)SLA over $S'	$.
\end{lemma}

\begin{proof}
(1) Since $h$ is surjective, \eqref{eq: USLA} is an isomorphism if and only if 
\begin{equation}\label{eq: USLA smooth local 1}
h^* (\cK_A \otimes f^* g_* \cK_T ) \rightarrow h^* g'_* ((g')^* \cK_A \otimes (f')^* \cK_T)
\end{equation}
is an isomorphism. Given a schematic map $g \co T \rightarrow S$ of derived Artin stacks, we have a commutative diagram 
\begin{equation}\label{eq: composition USLA}
\begin{tikzcd}
B' \ar[r, "g''"] \ar[d, "h'"] & A' \ar[d, "h"]  \\
B \ar[r, "g'"] \ar[d, "f'"] &  A \ar[d, "f"] \\
T \ar[r, "g"] & S
\end{tikzcd}
\end{equation}
where all squares are derived Cartesian and $h,h'$ are smooth. Using smooth base change, this induces a commutative diagram 
\[
\begin{tikzcd}
h^* (\cK_A \otimes f^* g_* \cK_T ) \ar[r] \ar[d, "\sim"]  &  h^* g'_* ((g')^* \cK_A \otimes (f')^* \cK_T) \ar[d, "\sim"]  \\
h^* \cK_A \otimes h^* f^* g_* \cK_T \ar[r] \ar[u]   \ar[d, "\sim"]  & g''_* (h')^* ((g')^* \cK_A \otimes (f')^* \cK_T)   \ar[u] \ar[d, "\sim"] \\
(h^* \cK_A) \otimes (f \circ h)^* g_* \cK_T \ar[u]  \ar[r]   & g''_* (( g'')^* (h^* \cK_A) \otimes (f' \circ h')^* \cK_T) 
\end{tikzcd}
\]
and comparing the top and bottom rows shows that \eqref{eq: USLA smooth local 1} is equivalent to $h^* \cK_A$ being SLA over $S$. Running the same argument over all base changes shows that $\cK_A$ is USLA over $S$ if and only if $h^* \cK_A$ is USLA over $S$.

(2) The argument is similar. 
\end{proof}

\begin{example}\label{ex: USLA over point} If $S$ is a point (by which we mean the spectrum of a field) then every object of $\Dmot{A}$ is USLA over $S$. Indeed, if $A$ is a derived scheme this follows from \cite[2.1.14]{JY21}, and the general case follows by induction and cohomological descent.
\end{example}

The (U)SLA property is preserved by direct image along proper morphisms. 

\begin{lemma}\label{lem: proper pushforward preserves ULA}
Let $h \co A' \rightarrow A$ be a proper morphism of derived Artin stacks over $S$. Let $\cK'_A \in \Dmot{A'}$ be (U)SLA over $S$. Then $h_! \cK_A' \in \Dmot{A}$ is (U)SLA over $S$. 
\end{lemma}

\begin{proof}
By proper base change, it suffices to show that $f_! \cK_A'$ is SLA over $S$ if $\cK_A'$ is SLA over $S$. Consider the commutative diagram \eqref{eq: composition USLA}
where all squares are derived Cartesian and $h,h'$ are proper. We want to show that the map 
\[
h_! 	\cK_A' \otimes  f^* g_* \cK_T \rightarrow  g'_*  ((g')^* h_! \cK_A' \otimes (f')^* \cK_T)
\]
is an isomorphism. We have a commutative diagram 
\[
\begin{tikzcd}
h_! 	\cK_A' \otimes  f^* g_* \cK_T  \ar[r]  &  g'_*  ((g')^* h_! \cK_A' \otimes (f')^* \cK_T) \ar[d, "\sim", "\text{proper base change}"']  \\
 &  g'_* (h'_! (g'')^* \cK_A' \otimes (f')^* \cK_T)   \ar[u]  \\
  &  g'_* h'_! ((g'')^*  \cK_A' \otimes (h')^* (f')^*  \cK_T)  \ar[d, equals, "\text{$h$ proper } \implies"'] \ar[u, "\sim"', "\text{projection formula}"] \\
 h_! (\cK_A' \otimes h^* f^* g_* \cK_T) \ar[r]  \ar[uuu, "\sim"', "\text{projection formula}"] & h_! g''_* ((g'')^*  \cK_A' \otimes (h')^* (f')^*  \cK_T ) 
\end{tikzcd}
\]
The bottom horizontal map is an isomorphism by definition of $\cK_A'$ being SLA over $S$, hence so is the top horizontal map. 
\end{proof}

\subsection{Relative K\"{u}nneth formulae} Let $S$ be a derived Artin stack.

\begin{notat}Let $A_0, A_1$ be derived stacks over $S$ and $\cK_0 \in \Dmot{A_0}$, $\cK_1 \in \Dmot{A_1}$. We write 
\[
\cK_0 \boxtimes_S \cK_1 \coloneqq \pr_0^* \cK_0 \otimes \pr_1^* \cK_1 \in \Dmot{A_0 \times_S A_1}. 
\]
\end{notat}

\begin{lemma}\label{lem: kunneth}
Let $f_0 \co A_0 \rightarrow B_0$ and $f_1 \co A_1 \rightarrow B_1$ be locally finite type morphisms of derived Artin stacks over $S$.
Then the commutative diagram
\[
\begin{tikzcd}
A_0 \ar[d, "f_0"] & A_0 \times_{S} A_1 \ar[l] \ar[d, "f_0 \times_S f_1"] \ar[r] & A_1 \ar[d, "f_1"] \\
B_0 & \ar[l] B_0 \times_{S} B_1 \ar[r] & B_1
\end{tikzcd}
\]
induces an isomorphism
\begin{equation}\label{eq: ! push kunneth}
f_{0!} \cK_0 \boxtimes_S f_{1!} \cK_1 \xrightarrow{\sim} (f_0 \times_S f_1)_! (\cK_0 \boxtimes_S \cK_1),
\end{equation}
natural in $\cK_0 \in \Dmot{A_0}$ and $\cK_1 \in \Dmot{A_1}$.
\end{lemma}

\begin{proof}
The proof of \cite[Lemma 2.2.3]{JY21} works verbatim. 
\end{proof}

Suppose we have a commutative diagram of derived Artin stacks
\begin{equation}\label{eq: relative kunneth diagram}
\begin{tikzcd}
T \times_S A \ar[r] \ar[d]  & S' \times_S A \ar[d]   \ar[r] & A \ar[d] \\
T \ar[r, "f"]  &  S'  \ar[r] & S
\end{tikzcd}
\end{equation}
There is a natural map 
\begin{equation}\label{eq: relative projection *}
f_* \cK_T \boxtimes_S \cK_A  \rightarrow  (f \times_S \Id_A)_* (\cK_T \boxtimes_S  \cK_A)  \in \Dmot{S' \times_S A}
\end{equation}
defined by adjunction from the composition of maps
\[
 (f \times_S \Id_A)^* (f_* \cK_T \boxtimes_S \cK_A)  \cong f^* f_* \cK_T \boxtimes_S \cK_A  \xrightarrow{\mrm{counit}} \cK_T \boxtimes_S \cK_A \in \Dmot{T \times_S A}.
\]

\begin{lemma}\label{lem: relative Kunneth} Let notation be as in diagram \eqref{eq: relative kunneth diagram}. Let $\cK_A \in \Dmot{A}$ be USLA over $S$. 

(1) If $f \co T \rightarrow S'$ is schematic, then for any $\cK_T \in \Dmot{T}$, the canonical morphism
\[
f_* \cK_T \boxtimes_S \cK_A  \xrightarrow{\eqref{eq: relative projection *}} (f \times_S \Id_A)_* (\cK_T \boxtimes_S  \cK_A) \in \Dmot{S' \times_S A}
\]
is an isomorphism. 

(2) If $f \co T \rightarrow S'$ is locally of finite type, then for any $\cK_{S'} \in \Dmot{S'}$, the natural morphism (adjoint to \eqref{eq: ! push kunneth}) 
\[
f^! \cK_{S'} \boxtimes_S \cK_A \rightarrow (f \times_S \Id_A)^! (\cK_{S'} \boxtimes_S  \cK_A) \in \Dmot{T \times_S A}
\]
is an isomorphism. 
\end{lemma}

\begin{proof} The proof is the essentially the same as that of \cite[Lemma 3.1.4]{Jin23}. 

(1) This is exactly the definition of $\cK_A$ being SLA after base change along $S' \rightarrow S$.

(2) The statement can be checked smooth-locally on $T$.
Thus we may assume $f$ is schematic and moreover factors through a closed immersion and a smooth morphism. If $f$ is smooth, then the result follows from the Poincaré duality isomorphism \S \ref{ssec:Gys} and its compatibility with base change. If $f$ is a closed embedding, then write $j \co U \inj S'$ for the complementary open. Abbreviate $f_A \coloneqq (f \times \Id_A)$ and $j_A \coloneqq (j \times \Id_A)$. We have a map of excision sequences in $\Dmot{T \times_S A}$: 
\[
\begin{tikzcd}
(f^! \cK_{S'} )\boxtimes_S \cK_A \ar[d] \ar[r] & (f^* \cK_{S'} ) \boxtimes_S \cK_A  \ar[d] \ar[r] & (f^* j_* j^* \cK_{S'}) \boxtimes_S \cK_A \ar[d] \\
f_A^! (\cK_{S'} \boxtimes_S \cK_A) \ar[r] & f_A^* (\cK_{S'} \boxtimes_S \cK_A) \ar[r] &  f_A^* j_{A*} j_A^* (\cK_{S'} \boxtimes_S \cK_A)
\end{tikzcd}
\]
The middle vertical map is obviously an isomorphism. The right vertical map is an isomorphism by item (1) applied to $U \inj S'$. Therefore the left vertical map is an isomorphism. 
\end{proof}

\begin{cor}\label{cor: ULA duality}
Let $\cK_A \in \Dmot{A}$ be USLA over $S$. Then for every derived Artin stack $T$ over $S$, the canonical map 
\[
\DD_{T/S} \boxtimes_S  \cK_A \rightarrow \pr_A^! \cK_A \in \Dmot{T \times_S A}
\]
is an isomorphism. 
\end{cor}

\begin{proof}
Apply Lemma \ref{lem: relative Kunneth}(2) with $S' = S$, $f$ the morphism $T \rightarrow S$, and $\cK_S = \bQ_S$.
\end{proof}

\begin{prop}\label{prop: ULA hom tensor} Let $\cK_A \in \Dmot{B}$ be USLA over $S$ and $\cK_B \in \Dmotg{B}$. Then
the canonical morphism 
\begin{equation}\label{eq: ULA hom tensor}
(\DD_{B/S} \cK_B) \boxtimes_S \cK_A \rightarrow \cRHom_{B \times_S A} (\pr_B^*  \cK_B, \pr_A^! \cK_A)
\end{equation}
is an isomorphism. 
\end{prop}

\begin{proof} By smooth base change, the map \eqref{eq: ULA hom tensor} can be checked to be an isomorphism smooth locally on $A$ and $B$. Since the hypotheses are also stable under smooth base change (using \S \ref{lem: USLA smooth local}), we may assume that $A$ and $B$ are schemes. By the definition of geometric motives, it suffices to check this on for $\cK_B$ of the form $f_{\sharp} \Qsh{T}$ for smooth $f \co T \rightarrow B$ (since the statement is evidently compatible with shifts and Tate twists). We refer to the commutative diagram 
\[
\begin{tikzcd}
T \times_S A \ar[d, "\pr_T"] \ar[r, "f_A"]  & B \times_S A \ar[d, "\pr_B"] \ar[r, "\pr_A"] & A \ar[d] \\
T \ar[r, "f"] & B \ar[r] & S
\end{tikzcd}
\]
Using that $f_{\sharp} \Qsh{T} \cong f_!  f^! \Qsh{B} \cong f_! \Qsh{T} \tw{d(f)}$, proper base change gives
\begin{align*}
\cRHom_{B \times_S A}(\pr_B^* f_{\sharp} \Qsh{T}, \pr_A^! \cK_A)  & \cong \cRHom_{B \times_S A} (f_{A!} \pr_T^* \Qsh{T} \tw{d(f)}, \pr_A^! \cK_A)  \\
& \cong f_{A*} \cRHom_{T \times_S A} (\pr_T^*  \Qsh{T} \tw{d(f)}, f_A^! \pr_A^! \cK_A)  
\end{align*}
where we write $f_A \co T \times_S A \rightarrow B \times_S A$ for the pullback of $f$. Since $\cK_A$ is assumed to be USLA over $S$, from Corollary \ref{cor: ULA duality} we have 
\[
\cRHom_{T \times_S A} (\pr_T^* \Qsh{T} \tw{d(f)}, f_A^! \pr_A^! \cK_A)  \cong \DD_{T/S} (f^! \Qsh{B}) \boxtimes_S  \cK_A.
\]
Again since $\cK_A$ is USLA over $S$, Lemma \ref{lem: relative Kunneth}(1) applies to give 
\[
f_{A*} (\DD_{T/S}(f^! \Qsh{B}) \boxtimes_S  \cK_A)    \cong (f_* \DD_{T/S} f^! \Qsh{B}  )\boxtimes_S  \cK_A  \cong \DD_{B/S}(f_! f^! \Qsh{B}) \boxtimes_S \cK_A,
\]
as desired. 
\end{proof}

\section{Cohomological correspondences}\label{sec: coh corr}
In this section we establish some general material related to cohomological correspondences. In \S \ref{ssec: base change natural transformations}, we recall the notion of ``pushable'' and ``pullable'' squares from \cite[\S 4, \S 5]{FYZ3} and the base change natural transformations that they entail. In \S \ref{ssec: correspondences} we recall the notion of cohomological correspondence, and in \S \ref{ssec: functoriality for cohomological correspondences} we formulate the notion of pushforward and pullback for cohomological correspondences. Finally in \S \ref{ssec: base change for cc} we state the Base Change Theorem for cohomological correspondences. The constructions and proofs carry over verbatim from $\ell$-adic sheaves as considered in \cite[\S 3, \S 4]{FYZ3} to $\Dmot{-}$, so we just formulate the statements without proof.

\subsection{Cohomological correspondences}\label{ssec: correspondences}

Let $A_{0}$ and $A_{1}$ be  derived Artin stacks. A {\em correspondence between $A_{0}$ and $A_{1}$} is a diagram of derived Artin stacks
\[
\begin{tikzcd}
A_0 & C \ar[l, "c_0"'] \ar[r, "c_1"]  & A_1
\end{tikzcd}
\]
where $c_1$ is locally of finite type.
A \emph{map of correspondences} from $(A_0 \xleftarrow{c_0} C \xrightarrow{c_1} A_1)$ to $(B_0 \xleftarrow{d_0} D \xrightarrow{d_1} B_1)$ is a commutative diagram 
\[
\begin{tikzcd}
A_0 \ar[d] & C \ar[l, "c_0"'] \ar[r, "c_1"]  \ar[d] & A_1 \ar[d] \\
B_0 & D \ar[l, "d_0"'] \ar[r, "d_1"] & B_1 
\end{tikzcd}
\]

Let $\cK_0 \in \Dmot{A_0}$, $\cK_1 \in \Dmot{A_1}$. A \emph{cohomological correspondence from $\cK_0$ to $\cK_1$ supported on $C$} is a map 
$c_0^* \cK_0 \rightarrow c_1^! \cK_1$ in $\Dmot{C}$. The vector space of such is denoted  
\begin{equation}
\Corr_{C}(\cK_{0}, \cK_{1}) \coloneqq\Hom_{C}(c_0^* \cK_0 , c_1^! \cK_1).
\end{equation}

\subsubsection{Fixed points of a self-correspondence}\label{sssec: fixed points}
Suppose that we have a fixed isomorphism $A_0 \xrightarrow{\sim} A_1$, which we will sometimes use to identify $A_0$ with $A_1$; however, it will also be convenient to distinguish them at times. Let $\Delta  \co A_0 \rightarrow A_0 \times A_1$ be the diagonal embedding. Define $\Fix(C)$ as the fibered product 
\begin{equation}
\begin{tikzcd}
\Fix(C) \ar[r, "\Delta'"] \ar[d, "c'"] & C \ar[d, "c"] \\
A_0 \ar[r, "\Delta"] & A_0 \times A_1
\end{tikzcd}
\end{equation}
where $c = (c_0, c_1)$. 

\subsection{Base change transformations}\label{ssec: base change natural transformations}
In order to discuss the functoriality of cohomological correspondences, we make a brief detour on base change transformations.

\subsubsection{Pushable and pullable squares} The notions of \emph{pushable} and \emph{pullable} squares were defined in \cite[Definition 3.1.1]{FYZ3} in order to codify situations where base change natural transformations can be constructed. Later we realized that the notion of pushable square appears at least implicitly in \cite{Zh15} for the same reason. We repeat the definitions for the convenience of the reader.

Let
\begin{equation}\label{eq: commutative square}
\begin{tikzcd}
A \ar[r, "g'"] \ar[d, "f'"'] & B \ar[d, "f"] \\
C \ar[r, "g"] & D 
\end{tikzcd}
\end{equation}
be a commutative square of derived Artin stacks. 
Denote by $\wt B=C\times_{D}B$ the derived fibered product so that the square \eqref{eq: commutative square} decomposes into a commutative diagram
\begin{equation}\label{eq: pushable square}
\xymatrix{ A\ar[dr]^{a}\ar@/_1pc/[ddr]_{f'}\ar@/^1pc/[drr]^{g'} \\
& \wt B\ar[d]^{\wt f}\ar[r]^{\wt g} & B\ar[d]^{f}\\
& C \ar[r]^{g}& D
}
\end{equation} 
where the bottom right square is derived Cartesian. 

\begin{defn}\label{def: pushable pullable} The square \eqref{eq: commutative square} is called 
\begin{itemize}
\item {\em pushable}, if $a$ is proper.
\item {\em pullable}, if $a$ is quasi-smooth. In this case, the {\em defect} of the square is by definition the relative dimension $d(a)$.
\end{itemize}
\end{defn}

\begin{remark}
Note that pushability is a purely topological notion: it can be checked on classical truncations (and even on underlying reduced stacks). By contrast, pullability is sensitive to the derived structure, and most of the pullable squares that arise for us would not be pullable on classical truncations. 
\end{remark}

\begin{example}\label{ex: pushable/pullable examples}
If $f$ is separated and $f'$ is proper, then \eqref{eq: pushable square} is pushable. If $f$ is smooth and $f'$ is quasi-smooth, then \eqref{eq: pushable square} is pullable.
\end{example}

\begin{remark}\label{rem: top pullable}
We will also have occasion to consider the following variant: we say \eqref{eq: commutative square} is \emph{topologically pullable} if the morphism $a$ is a finite radicial surjection (e.g. if it is an isomorphism on reduced classical truncations).
In this case the defect is zero by convention.
\end{remark}

\subsubsection{Proper base change}

Suppose \eqref{eq: commutative square} is Cartesian after taking classical truncations and then underlying reduced stacks. Then there is a proper base change natural isomorphism 
\begin{equation}\label{eq: pbc di}
g^* f_! \xrightarrow{\di} f'_! (g')^*
\end{equation}
of functors $\Dmot{B} \rightarrow \Dmot{C}$ which we label by ``$\di$''. We use the same notation for the natural isomorphism
\[ 
f'_* (g')^! \xrightarrow{\di} g^! f_* .
\]
By adjunction, \eqref{eq: pbc di} induces natural transformations 
\[
f_! (g')_* \xrightarrow{\di} g_* f'_!
\]
and
\[
(g')^* f^! \xrightarrow{\di} (f')^! g^*
\]
which we will also label by ``$\di$''.

\subsubsection{Push-pull base change transformation}\label{sssec: push-pull} Suppose \eqref{eq: commutative square} is pushable. Then we have a natural transformation of functors $\Dmot{B} \rightarrow \Dmot{C}$
\begin{equation}\label{eq: push-pull gen} 
g^{*}f_{!} \xrightarrow{\td} f'_{!}(g')^{*}
\end{equation}
defined as the composition
\[
g^{*}f_{!}\xr{\di} \wt f_{!}\wt g^{*}\xrightarrow{\unit(a)} \wt f_{!}a_{*}a^{*}\wt g^{*}=\wt f_{!}a_{!}a^{*}\wt g^{*}=f'_{!}(g')^{*}
\]
Here we used that $\mrm{fsupp}_a \co a_! \rightarrow a_*$ is invertible because $a$ is proper. We sometimes denote this base change transformation by $\td$.

\subsubsection{Push-push base change transformations}\label{sssec: push-push}

Suppose \eqref{eq: commutative square} is pushable. Then we have a natural transformation of functors $\Dmot{A} \rightarrow \Dmot{D}$
\begin{equation}\label{eq: push-push gen}
f_{!}g'_{*}\to g_{*}f'_{!}
\end{equation}
defined as the composition
\[
f_{!}g'_{*}=f_{!}\wt g_{*}a_{*}\xr{\di}g_{*}\wt f_{!}a_{*}=g_{*}\wt f_{!}a_{!}\to g_{*}f'_{!}.
\]
Again we used that $\mrm{fsupp}_a \co a_! \rightarrow a_*$ is invertible because $a$ is proper. We sometimes denote this base change transformation by $\td$.

\subsubsection{Pull-pull base change transformation}\label{ssec: pull-pull}

Suppose that \eqref{eq: commutative square} is pullable with defect $\d$. Then we have a natural transformation of functors $\Dmot{D} \rightarrow \Dmot{A}$
\begin{equation}\label{eq: pull-pull gen}
(f')^{*}g^{!}  \xrightarrow{\tu} (g')^{!} f^{*}\tw{-\d}
\end{equation}
defined as the composition 
\[
(f')^{*}g^{!} =  a^{*} \wt f^{*}g^{!} \xr{a^* \di}a^{*} \wt g^{!} f^{*}\xr{[a]}a^{!} \wt g^{!} f^{*}\tw{-\d}=(g')^{!} f^{*}\tw{-\d}.
\]
We often denote such a natural transformation induced by a pullable square by $\tu$. 

\begin{remark} \label{rem: p-p} If \eqref{eq: commutative square} is pullable, the map \eqref{eq: pull-pull gen} induces by adjunction a map
\begin{equation}\label{eq: p-p}
g'_{!}(f')^{*} \xrightarrow{\tu}  f^{*}g_{!}\tw{-\d}.
\end{equation}
\end{remark}

\begin{remark}\label{rem:top pullable tu}
If \eqref{eq: commutative square} is topologically pullable as in \remref{rem: top pullable}, then we have a canonical isomorphism $a^* \cong a^!$ by topological invariance (see \cite[Remark~2.1.13]{zbMATH07194962}).
Therefore we may define a natural transformation $\tu \colon (f')^*g^! \to (g')^!f^*$ just as in \eqref{eq: pull-pull gen}.
\end{remark}

\subsubsection{Compatibility with compositions}

The natural transformations $\td$ and $\tu$ are compatible with compositions in the following sense. Suppose we have a commutative diagram 
\begin{equation}\label{eq: commutative composition}
\begin{tikzcd}
A \ar[r, "g''"] \ar[d, "f'"'] & B \ar[d, "f"] \\
C \ar[r, "g'"] \ar[d, "h'"'] & D \ar[d, "h"]  \\
E \ar[r, "g"] & F 
\end{tikzcd}
\end{equation}
According to \cite[Lemma 3.2.2]{FYZ3} and \cite[Lemma 3.5.3]{FYZ3}:
\begin{enumerate}
\item If both the upper square and the lower square are pushable, then the outer square formed by $(A,B,E,F)$ is also pushable.
\item If both up the upper square and the lower square are pullable, say of defects $\delta_{\upp}$ and $\delta_{\low}$, then the outer square is also pullable, with defect $\delta_{\out} \coloneqq \delta_{\upp} + \delta_{\low}$. 
\end{enumerate}

Suppose the outer and lower squares in \eqref{eq: commutative composition} are pushable. Then by the same argument as in proof of \cite[Proposition 3.2.3]{FYZ3}, we have the following commutative diagrams of natural transformations $\Dmot{B} \rightarrow \Dmot{E}$, resp. $\Dmot{A} \rightarrow \Dmot{F}$: 
\begin{equation}\label{eq: comp push-push}
\xymatrix{g^* h_! f_! \ar@{=}[d] \ar[r]^-{\td f_{!}} & h'_! (g')^* f_! \ar[r]^-{h'_{!}\td} & h'_! f'_! (g'')^*\ar@{=}[d]\\
g^{*}(h\c f)_{!}\ar[rr]^-{\td} &&   (h'\c f')_{!}(g'')^{*}
} \hspace{1cm} \xymatrix{h_! f_! g''_* \ar@{=}[d] \ar[r]^-{h_! \td} & h_! g'_* f'_! \ar[r]^-{\td f'_!} & g_* h'_! f'_! \ar@{=}[d]\\
(h\c f)_{!} g''_* \ar[rr]^-{\td} &&   g_* (h'\c f')_{!}
}
\end{equation}

Suppose the upper and lower squares in \eqref{eq: commutative composition} are pullable, of defects $\delta_{\upp}$ and $\delta_{\low}$. Then by the same argument as in proofs of \cite[Proposition 3.5.4]{FYZ3}, we have the following commutative diagram of natural transformations $\Dmot{F} \rightarrow \Dmot{A}$:
\begin{equation}\label{eq: comp pull}
\xymatrix{(f')^{*}(h')^{*}g^{!}\ar[r]^-{f'^{*}\tu}\ar@{=}[d] & (f')^{*}(g')^{!}h^{*}\tw{-\d_{\low}}\ar[r]^-{\tu h^{*}} & (g'')^{!}f^{*}h^{*}\tw{-\d_{\low}-\d_{\upp}}\ar@{=}[d]
\\
(h'\c f')^{*}g^{!} \ar[rr]^-{\tu} && (g'')^{!}(h\c f)^{*}\tw{-\d_{\out}}}
\end{equation}

\subsection{Functoriality for cohomological correspondences}\label{ssec: functoriality for cohomological correspondences} We tabulate some situations where cohomological correspondences can be pushed forward or pulled back. 

\subsubsection{Pushforward functoriality for cohomological correspondences}\label{ssec: pushforward functoriality for CC}

Suppose we have a map of correspondences
\begin{equation}\label{eq: pushforward cohomological correspondences diagram}
\begin{tikzcd}
A_0  \ar[d, "f_0"'] & C \ar[l, "c_0"'] \ar[r, "c_1"]  \ar[d, "f"]  & A_1  \ar[d, "f_1"] \\
B_0  & D \ar[l, "d_0"'] \ar[r, "d_1"]  & B_1  
\end{tikzcd}
\end{equation}

\begin{defn}\label{def: corr pushable} The map of correspondences \eqref{eq: pushforward cohomological correspondences diagram} is called {\em left pushable} if the square with vertices $(C, A_{0}, D, B_{0})$ is pushable in the sense of Definition \ref{def: pushable pullable}. 
\end{defn}

Assume \eqref{eq: pushforward cohomological correspondences diagram} is left pushable. Then for any cohomological correspondence $c_0^* \cK_0 \xrightarrow{\cc} c_1^! \cK_1$, there is a ``pushforward correspondence'' $f_{!}(\cc): d_0^* f_{0!} \cK_0  \rightarrow d_1^!  f_{1!} \cK_1$, defined as the composition 
\[
d_0^* f_{0!} \cK_0 \xrightarrow{\td}	 f_! c_0^* \cK_0 \xrightarrow{\cc} f_! c_1^! \cK_1 \rightarrow d_1^! f_{1!} \cK_1
\]
where the rightmost map is the natural base change transformation. Thus $\cc \mapsto f_!(\cc)$ defines a linear map
\begin{equation}
f_{!}: \Corr_{C}(\cK_{0}, \cK_{1})\to \Corr_{D}(f_{0!}\cK_{0}, f_{1!}\cK_{1}).
\end{equation}

\subsubsection{Pullback functoriality for cohomological correspondences} \label{ssec: pullback functoriality for CC}
Consider the diagram of correspondences in \eqref{eq: pushforward cohomological correspondences diagram}.

\begin{defn}\label{def: corr pullable} The diagram of correspondences \eqref{eq: pushforward cohomological correspondences diagram} is called {\em right pullable} if the square with vertices $(C, A_{1}, D, B_{1})$ is pullable in the sense of Definition \ref{def: pushable pullable}. 

In this case, we also say that the map of correspondences $f:C\to D$ is right pullable, with {\em defect} $\d_{f}$ defined to be the defect of the square $(C,A_{1}, D,B_{1})$, i.e., the relative dimension of the quasi-smooth map $\wt{c}_1 \co C\to D\times_{A_1} B_{1}$. 
\end{defn}

Suppose \eqref{eq: pushforward cohomological correspondences diagram} is right pullable. Then for any cohomological correspondence $d_0^* \cK_0 \xrightarrow{\cc} d_1^! \cK_1$ there is a ``pullback correspondence'' $f^{*}(\cc): c_0^* f_{0}^{*} \cK_0  \rightarrow c_1^!  f_{1}^{*} \cK_1\tw{-\d_f}$ defined as the composition 
\[
c_0^* f_0^*  \cK_0 = f^* d_0^*  \cK_0 \xrightarrow{\cc} f^* d_1^! \cK_1 \xrightarrow{\tu}   c_1^!  f_{1}^* \cK_1 \tw{-\d_f}.
\]
Thus $\cc \mapsto f^*(\cc)$ defines a linear map
\begin{equation}
f^*: \Corr_{D}(\cK_{0}, \cK_{1})\to \Corr_{C}(f_{0}^{*}\cK_{0}, f_{1}^{*}\cK_{1}\tw{-\d_f}).
\end{equation}

\begin{remark}\label{rem:right top pullable}
Similarly, we say the map of correspondences $f : C \to D$ is \emph{right topologically pullable} (with defect $\d_f = 0$) if the square with vertices $(C, A_1, D, B_1)$ is topologically pullable in the sense of Remark~\ref{rem: top pullable}.
We can then similarly define a pullback operation
\begin{equation}
f^*: \Corr_{D}(\cK_{0}, \cK_{1})\to \Corr_{C}(f_{0}^{*}\cK_{0}, f_{1}^{*}\cK_{1})
\end{equation}
using Remark~\ref{rem:top pullable tu}.
\end{remark}

\subsection{Base change for cohomological correspondences}\label{ssec: base change for cc}
In this subsection we formulate a base change result for cohomological correspondences (Theorem~\ref{thm: descent of pushforward correspondence}), following \cite[\S 5]{FYZ3}. 

\subsubsection{Setup}\label{sssec: BC assumptions}

Suppose we are given a commutative diagram of derived Artin stacks 
\begin{equation}\label{eq: two cubes}
\adjustbox{scale=0.8,center}{
\xymatrix{ U_{0}\ar[ddd]^{\pi_{0}}\ar[dr]^{f_{0}} && \ar[ll]_-{a_{0}}\ar[rr]^-{a_{1}}C_{U}\ar[ddd]^{\pi}\ar[dr]^{f} && U_{1}\ar[ddd]^{\pi_{1}}\ar[dr]^{f_{1}} \\
& V_{0}\ar[ddd]^{g_{0}} && \ar[ll]_-{b_{0}}\ar[rr]^-{b_{1}} C_{V}\ar[ddd]^{g} && V_{1}\ar[ddd]^{g_{1}}\\
\\
S_{0}\ar[dr]^{z_{0}} && \ar[ll]^-{h_{0}}\ar[rr]_-{h_{1}} C_{S}\ar[dr]^{z} && S_{1}\ar[dr]^{z_{1}}\\
& W_{0} && \ar[ll]_-{c_{0}}\ar[rr]^-{c_{1}} C_{W} && W_{1} 
}
}
\end{equation}
satisfying the following conditions:  
\begin{enumerate}
\item The middle vertical parallelogram
\begin{equation}
\adjustbox{scale = 0.7, center}{
\xymatrix{C_{U}\ar[ddd]^{\pi}\ar[dr]^{f}\\
& C_{V}\ar[ddd]^{g}\\
\\
C_{S}\ar[dr]^{z}\\
& C_{W}}
}
\end{equation}
is derived Cartesian.

\item The three squares in the following diagram are pushable:
\begin{equation}\label{eq:  UVW derived Cartesian}
\adjustbox{scale=0.7,center}{
\xymatrix{
U_{0}\ar[ddd]^{\pi_{0}}\ar[dr]^{f_{0}} && \ar[ll]_{a_{0}}C_{U}\ar[dr]^{f} && \\
& V_{0}\ar[ddd]^{g_{0}} &&  \ar[ll]_{b_{0}}C_{V}&& \\
\\
S_{0}\ar[dr]^{z_{0}} && \ar[ll]_(.6){h_{0}}C_{S}\ar[dr]^{z} &&\\
& W_{0} && \ar[ll]_{c_{0}}C_{W} && 
}
}
\end{equation}

\item The three squares in the following diagram are pullable:
\begin{equation}
\adjustbox{scale=0.7,center}{
\xymatrix{&& \ar[rr]^-{a_{1}}C_{U}\ar[ddd]^{\pi} && U_{1}\ar[ddd]^{\pi_{1}}\ar[dr]^{f_{1}} \\
&&& \ar[rr]^-{b_{1}} C_{V}\ar[ddd]^{g} && V_{1}\ar[ddd]^{g_{1}}\\
\\
&& \ar[rr]^(.4){h_{1}} C_{S} && S_{1}\ar[dr]^{z_{1}}\\
&&& \ar[rr]^-{c_{1}} C_{W} && W_{1} 
}
}
\end{equation}
Moreover, the right square $(U_{1}, V_{1}, S_{1}, W_{1})$ above has defect zero.
\end{enumerate}

\subsubsection{} We view $C_{S}$ as a correspondence between $S_{0}$ and $S_{1}$, and similarly for $C_{U}, C_{V}$ and $C_{W}$.  Let $\cK_{i}\in \Dmot{S_{i}}$ for $i\in \{0,1\}$ and $\frs\in \Corr_{C_{S}}(\cK_{0}, \cK_{1})$.

\subsubsection{Push $\circ$ pull} By assumption, the back face of \eqref{eq: two cubes} is pullable as a map of correspondences $\pi: C_{U}\to C_{S}$, so the map
\begin{equation}\label{eq: bc cc 1}
\pi^{*}: \Corr_{C_{S}}(\cK_{0}, \cK_{1})\to \Corr_{C_{U}}(\pi_{0}^{*}\cK_{0}, \pi_{1}^{*}\cK_{1}\tw{-\d_{\pi}})
\end{equation}
is defined (where the defect $\d_{\pi}$ is defined in Definition \ref{def: corr pushable}). By assumption, the top face of \eqref{eq: two cubes} is pushable as a map of correspondences $f: C_{U}\to C_{V}$, so the map
\begin{equation}\label{eq: bc cc 2}
f_{!}: \Corr_{C_{U}}(\pi^{*}_{0}\cK_{0}, \pi_{1}^{*}\cK_{1}\tw{-\d_{\pi}})\to \Corr_{C_{V}}(f_{0!}\pi_{0}^{*}\cK_{0}, f_{1!}\pi_{1}^{*}\cK_{1}\tw{-\d_{\pi}})
\end{equation}
is defined. The composition of \eqref{eq: bc cc 1} and \eqref{eq: bc cc 2} applied to $\mf{s} \in \Corr_{C_{S}}(\cK_{0}, \cK_{1})$ gives an element
\begin{equation}
f_{!}\pi^{*}(\frs)\in \Corr_{C_{V}}(f_{0!}\pi_{0}^{*}\cK_{0}, f_{1!}\pi_{1}^{*}\cK_{1}\tw{-\d_{\pi}}).
\end{equation}

\subsubsection{Pull $\circ$ push} Similarly, since the bottom face of the diagram \eqref{eq: two cubes} is left pushable and the front face is right pullable, the cohomological correspondence
\begin{equation}\label{eq: bc cc 3}
g^{*}z_{!}(\frs)\in \Corr_{C_{V}}(g_{0}^{*}z_{0!}\cK_{0}, g_{1}^{*}z_{1!}\cK_{1}\tw{-\d_{g}})
\end{equation}
is defined. 

\subsubsection{} We are now ready to formulate the base change theorem, expressing the compatibility of push $\circ$ pull and pull $\circ$ push.

By assumption, the square $(U_{0},V_{0}, S_{0}, W_{0})$ in \eqref{eq: two cubes} is pushable, so we get a base change natural transformation
\begin{equation}\label{eq: bc cc 4}
g^{*}_{0}z_{0!}\xr{\td} f_{0!}\pi_{0}^{*}: \Dmot{S_{0}} \to \Dmot{V_{0}}.
\end{equation}
By assumption, the square $(U_{1},V_{1}, S_{1}, W_{1})$ in \eqref{eq: two cubes} is pullable with defect zero, so we get a base change natural transformation
\begin{equation}\label{eq: bc cc 5}
\pi_{1}^{*}z_{1}^{!}\xr{\tu} f_{1}^{!}g^{*}_{1}: \Dmot{W_{1}} \to \Dmot{U_{1}}.
\end{equation}
By adjunction (cf. Remark \ref{rem: p-p}), \eqref{eq: bc cc 5} gives a base change natural transformation
\begin{equation}\label{eq: match target}
 f_{1!}\pi_{1}^{*}\to g^{*}_{1}z_{1!}: \Dmot{S_{1}} \to \Dmot{V_{1}}.
\end{equation}
We have an equality of defects $\d_{\pi}=\d_{g}$ \cite[Lemma 5.1.1]{FYZ3}.

\begin{example}
Suppose $(U_{0},V_{0}, S_{0}, W_{0})$ and $(U_{1},V_{1}, S_{1}, W_{1})$ are derived Cartesian. In this case, the sources and targets of $f_{!}\pi^{*}(\frs)$ and $g^{*}z_{!}(\frs)$ are identified by the proper base change isomorphisms
\begin{equation}\label{eq: match source target Cart}
f_{0!}\pi_{0}^{*}\cK_{0}\xr{\di} g^{*}_{0}z_{0!}\cK_{0} \quad \text{and} \quad f_{1!}\pi_{1}^{*}\cK_{1}\tw{-\d_{\pi}}\xr{\di} g^{*}_{1}z_{1!}\cK_{1}\tw{-\d_{g}}.
\end{equation}
\end{example}

\begin{thm}[Base change for cohomological correspondences]\label{thm: descent of pushforward correspondence}
Let the notation be as in \S \ref{sssec: BC assumptions}.
Then for every $\cK_{0}\in \Dmot{S_0}$, $\cK_{1}\in \Dmot{S_1}$, and $\frs\in \Corr_{C_{S}}(\cK_{0}, \cK_{1})$, the following diagram commutes:
\begin{equation}
\xymatrix{g^{*}_{0}z_{0!}\cK_{0}\ar[rr]^-{g^{*}z_{!}(\frs)}\ar[d]_{\eqref{eq: bc cc 4}} & & g^{*}_{1}z_{1!}\cK_{1}\tw{-\d_{g}} \\
f_{0!}\pi_{0}^{*}\cK_{0} \ar[rr]^-{f_{!}\pi^{*}(\frs)} && f_{1!}\pi_{1}^{*}\cK_{1}\tw{-\d_{\pi}}\ar[u]_{\eqref{eq: match target}}
}
\end{equation}
(Here we use \cite[Lemma 5.1.1]{FYZ3} to match the twists.) 

In particular, when both $(U_{0},V_{0}, S_{0}, W_{0})$ and $(U_{1},V_{1}, S_{1}, W_{1})$ are derived Cartesian, we have an equality of cohomological correspondences on $C_{V}$
\begin{equation}
f_{!}\pi^{*}(\frs)=g^{*}z_{!}(\frs)
\end{equation}
under the isomorphisms \eqref{eq: match source target Cart}.
\end{thm}

\begin{proof}
The proof of \cite[Theorem 5.1.3]{FYZ3} works verbatim. 
\end{proof}

\section{The Lu--Zheng categorical trace}\label{sec: lu-zheng}In this section we adapt the framework of Lu--Zheng \cite{LZ22}, which gave a new perspective on ULA sheaves and Lefschetz-Verdier formulas in the $\ell$-adic setting, to the derived category of motives. For a derived Artin stack $S$ over a field, we define symmetric monoidal 2-categories $\LZ(S)_!$ and $\LZ(S)^*$, in which objects are motivic sheaves on derived Artin stacks over $S$, morphisms are cohomological correspondences, and 2-morphisms are either pushforward or pullback of cohomological correspondences. Strongly universally locally acyclic motives are dualizable, so one can define the \emph{categorical trace} of their endomorphisms. We use this to study relative Lefschetz-Verdier pairings and their behavior under pullback and pushforward.

Compared to \cite{LZ22} we introduce some technical enhancements, which would apply equally well to the $\ell$-adic setting and simplify some proofs in \cite{FYZ3}. 
\begin{itemize}
\item We introduce an extended version of the Lu--Zheng category whose morphisms include higher Exts. This means that the trace of an endomorphism (of a dualizable object) can be valued in higher degree Chow groups, which is responsible for eventually promoting the sheaf-function correspondence to a sheaf-\emph{cycle} correspondence. 
\item We introduce a variant of the Lu--Zheng category adapted to pullbacks instead of pushforwards. This is eventually used to prove the compatibility of the sheaf-cycle correspondence with pullbacks. For this, it is essential to incorporate derived algebraic geometry (unlike the pushforward version, where all the content is already found in the classical truncation). 
\item We work with derived Artin stacks (as opposed to schemes), a generality which is needed in applications. 
\end{itemize}

\subsection{Motivic Lu--Zheng categories}

Let $S$ be a locally finite type derived Artin stack over a field. Following work of Lu--Zheng \cite{LZ22}, we will define two 2-categories, which we denote $\LZ(S)_!$ and $\LZ(S)^*$. 

\subsubsection{Objects and morphisms} Both $\LZ(S)_!$ and $\LZ(S)^*$ have the following objects and 1-morphisms:
\begin{itemize}
\item The objects are pairs $(A, \cK_A)$ where $A$ is a locally finite type derived Artin stack over $S$, and $\cK_A \in \Dmot{A}$. 
\item A \emph{morphism} from $(A_0, \cK_0)$ to $(A_1, \cK_1)$ is a triple $(c, i, \cc)$ where $c = (A_0 \xleftarrow{c_0} C \xrightarrow{c_1} A_1)$ is a correspondence, $i \in \Z$, and $\cc \in \Corr_C(\cK_0, \cK_1 \tw{-i})$ is a cohomological correspondence from $\cK_0$ to $\cK_1\tw{-i}$ with respect to $c$.  
\end{itemize}
In particular, $\LZ(S)_!$ and $\LZ(S)^*$ have the same objects and 1-morphisms. 

\begin{notat}We will generally use a lower-case Roman letter such as $c$ or $d$, with no subscripts, as shorthand for a correspondence involving the corresponding upper-case Roman letter, such as $C$ or $D$. 
\end{notat}

The composition of morphisms $(c, i, \cc) \co (A_0, \cK_0) \rightarrow (A_1, \cK_1) $ and $(d, j, \dd) \co  (A_1, \cK_1) \rightarrow (A_2, \cK_2)$ is $(e, i+j, \ee)$, where $e$ is the outer correspondence in the diagram 
\[
\begin{tikzcd}
& & E \ar[dl, "d_0'"'] \ar[dr, "c_1'"] \\
& C \ar[dl, "c_0"'] \ar[dr, "c_1"]  & & D \ar[dl, "d_0"'] \ar[dr, "d_1"] \\
A_0 & & A_1  && A_2 \end{tikzcd}
\]
where the diamond is (derived) Cartesian, and $\ee$ is the composition 
\[
(d_0')^*  c_0^* \cK_1 \xrightarrow{\cc}  (d_0')^* c_1^!  \cK_1  \tw{-i}  \xrightarrow{\di} (c_1')^! d_0^* \cK_1  \tw{-i} \xrightarrow{\dd} (c_1')^! d_1^! \cK_2 \tw{-i-j}
\]

\subsubsection{2-morphisms} The 2-morphisms of $\LZ(S)_!$ and $\LZ(S)^*$ differ: 
\begin{itemize}
\item In $\LZ(S)_!$: Given two morphisms $(c, i, \cc)$ and $(d, j, \dd)$ from $(A_0, \cK_0)$ to $(A_1, \cK_1)$, a \emph{2-morphism} $(c, i, \cc) \rightarrow (d, j, \dd)$ in $\LZ(S)_!$ is a map of correspondences 
\begin{equation}\label{eq: 2-morphism}
\begin{tikzcd} 
A_0  \ar[d, equals] & C \ar[r, "c_1"] \ar[l, "c_0"'] \ar[d, "f"] & A_1 \ar[d, equals] \\
A_0 & D \ar[r, "d_1"] \ar[l, "d_0"'] & A_1
\end{tikzcd}
\end{equation}
in which $f$ is proper (hence the map of correspondences is left pushable by Example \ref{ex: pushable/pullable examples}), and such that $\dd = f_! \cc$ in the sense of \S \ref{ssec: pushforward functoriality for CC} (so that $i=j$).
\item In $\LZ(S)^*$: Given two morphisms $(c, i, \cc)$ and $(d, j, \dd)$ from $(A_0, \cK_0)$ to $(A_1, \cK_1)$, a \emph{2-morphism} $(c,  i, \cc) \rightarrow (d, j, \dd)$ is a map of correspondences \eqref{eq: 2-morphism} in which $f$ is quasi-smooth (hence the map of correspondences is right pullable by Example \ref{ex: pushable/pullable examples}), and such that $\cc = f^* \dd$ (so that $i-j$ equals the defect $\delta_f$).
\end{itemize}
The composition of 2-morphisms is given by the obvious construction. 

\begin{remark}
The category $\LZ(S)_!$ is a graded motivic version of the category $\cC_S$ from \cite[Construction 2.6]{LZ22}. This category is adapted to the purpose of proving relative Lefschetz-Verdier formulas, which concern the compatibility of pushforwards with the Lefschetz-Verdier pairing (cf. \S \ref{ssec: Lefschetz-Verdier}). The category $\LZ(S)^*$ (which to our knowledge has not been previously considered) is adapted to proving compatibility of pullbacks with the Lefschetz-Verdier pairing.
\end{remark}

\begin{example}\label{ex: natural}
Let $f \co A \rightarrow A'$ be a morphism of derived Artin stacks over $S$. Then for $\cK \in \Dmot{A}$, we have a 1-morphism in $\LZ(S)_!$ or $\LZ(S)^*$ 
\[
f_\natural \co (A, \cK) \rightarrow (A', f_! \cK)
\]
given by the correspondence 
\[
\begin{tikzcd} 
A &  \ar[l, equals] A \ar[r, "f"] & A'
\end{tikzcd}
\]
equipped with the cohomological correspondence $\cK = \Id^* \cK \xrightarrow{\mrm{unit}} f^! f_! \cK$. 

For $\cK' \in \Dmot{A'}$, we have a 1-morphism in $\LZ(S)_!$ or $\LZ(S)^*$ 
\[
f^\natural \co (A', \cK') \rightarrow (A, f^* \cK)
\]
given by the correspondence 
\[
\begin{tikzcd}
A' & \ar[l, "f"'] A \ar[r, equals] & A
\end{tikzcd}
\]
equipped with the tautological cohomological correspondence $f^* \cK'  \rightarrow  \Id^! (f^* \cK') $. 
\end{example}

\begin{example}\label{ex: 2-morphism} Let 
\begin{equation}\label{diag: 2-morphism}
\begin{tikzcd}
A_0  \ar[d, "f_0"] & \ar[l, "c_0"'] C \ar[r, "c_1"]  \ar[d, "f"]  & A_1\ar[d, "f_1"]  \\
A_0' & \ar[l, "c_0'"'] C' \ar[r, "c_1'"]  & A_1'
\end{tikzcd}
\end{equation}
be a commutative diagram of derived Artin stacks over $S$. 

Suppose \eqref{diag: 2-morphism} is left pushable as a map of correspondences. Let $(c,i, \cc) \co (A_0, \cK_0) \rightarrow (A_1, \cK_1)$ be a morphism in $\LZ(S)_!$. Then, unraveling the definitions, there is tautologically a unique 2-morphism in $\LZ(S)_!$ fitting into a commutative diagram with the pushforward cohomological correspondence $f_! \cc$ from \S \ref{ssec: pushforward functoriality for CC}: 
\begin{equation}\label{eq: push 2-morphism}
\xymatrix{
(A_0,\cK_0) \ar[d]^{f_{0\natural}}\ar[r]^{(c,i, \cc)} & (A_1,\cK_1)\ar[d]^{f_{1\natural}}\\
(A_0',f_{0!} \cK_0)\ar[r]_{(c', i, f_!\cc)} & (A'_1,f_{1!} \cK_1)\ultwocell\omit{^{}}}
\end{equation}

Now suppose instead that \eqref{diag: 2-morphism} is right pullable as a map of correspondences. Let $(c',i', \cc') \co (A_0', \cK_0') \rightarrow (A_1', \cK_1')$ be a morphism in $\LZ(S)^*$. Then, unraveling the definitions, there is tautologically a unique 2-morphism in $\LZ(S)^*$ fitting into a commutative diagram with the pullback cohomological correspondence $f^* \cc'$ from \S \ref{ssec: pullback functoriality for CC}: 
\begin{equation}\label{eq: pull 2-morphism}
\xymatrix{
(A_0,f_0^* \cK_0) \ar[r]^{(c,i, f^* \cc' )}   & (A_1,f_1^* \cK_1) \dltwocell\omit{^{}} \\
(A_0', \cK_0')\ar[r]_{(c', i', \cc')} \ar[u]^{f_{0}^\natural} & (A'_1, \cK_1') \ar[u]^{f_1^\natural} 
}
\end{equation}
where $i = i' + \delta_f$.

\end{example}

\subsubsection{Symmetric monoidal structure} We construct a symmetric monoidal structure on $\LZ(S)^*$ and $\LZ(S)_!$. In both cases, we define the tensor product of objects as 
\[
(A, \cK_A) \otimes (B, \cK_B) \coloneqq (A \times_S B, \cK_A \boxtimes_S \cK_B)
\]
where we recall that $\cK_A \boxtimes_S \cK_B \coloneqq \pr_A^* \cK_A \otimes \pr_B^* \cK_B$ for the projection maps $A \xleftarrow{\pr_A} A \times_S B \xrightarrow{\pr_B} B$. 

The tensor product of 1-morphisms $(c, i, \cc) \co (A_0, \cK_0) \rightarrow (A_1, \cK_1)$ and $(d, j,  \dd) \co (B_0, \cL_0) \rightarrow (B_1, \cL_1)$ is the product (over $S$) correspondence $c \times_S d$ equipped with the cohomological correspondence
\[
(c_0 \times_S d_0)^* (\cK_0 \boxtimes_S \cL_0) \cong c_0^*(\cK_0) \boxtimes_S d_0^*(\cL_0) \xrightarrow{\cc \boxtimes_S \dd} c_1^!  (\cK_1) \tw{-i} \boxtimes_S d_1^! (\cL_1) \tw{-j} \rightarrow (c_1 \times_S d_1)^!  (\cK_1 \boxtimes \cL_1) \tw{-i-j}
\]
where the last map is adjoint to the K\"{u}nneth formula 
\[
(c_1 \times_S d_1)_! (\cK_1' \boxtimes_S \cL_1' ) \cong  c_{1!} \cK_1' \boxtimes_S d_{1!} \cL_1'
\]
from Lemma \ref{lem: kunneth}.

In both $\LZ(S)^*$ and $\LZ(S)_!$, the tensor product of 2-morphisms is induced by product of morphisms of stacks over $S$. 

\begin{example}
The monoidal unit in both $(\LZ(S)_!, \otimes)$ and $(\LZ(S)^*, \otimes)$ is the object $(S, \Qsh{S})$. 
\end{example}

\subsection{Dualizable objects} Any symmetric monoidal category  $(\msf{C}, \otimes)$ has a notion of \emph{dualizable object}: this means an object $c \in \msf{C}$ such that there exists a \emph{dual} $c^\vee \in \msf{C}$ and evaluation (resp. coevaluation) morphisms $\ev_c \co c^\vee \otimes c \rightarrow 1_{\msf{C}}$ (resp. $\coev_c \co 1_{\msf{C}} \rightarrow c \otimes c^\vee$) such that the composites
\[
c \xrightarrow{\coev_c \otimes \Id_c} c \otimes c^\vee  \otimes c \xrightarrow{\Id_c \otimes \ev_c} c, \hspace{1cm} 
c^\vee \xrightarrow{\Id_{c^\vee} \otimes \coev_c} c^\vee \otimes c \otimes c^\vee \xrightarrow{\ev_c \otimes \Id_{c^\vee}} c^\vee 
\]
are isomorphic to the respective identity morphisms.

\begin{prop}
Let $(A, \cK)$ be a dualizable object of $\LZ(S)_!$ or $\LZ(S)^*$. Then the dual of $(A, \cK)$ is $(A, \DD_{A/S} \cK)$. 
\end{prop}

\begin{proof}
The proof of \cite[Proposition 2.11]{LZ22} works verbatim.
\end{proof}

\begin{cor}Let $(A, \cK)$ be a dualizable object of $\LZ(S)_!$ or $\LZ(S)^*$. Then the canonical map $\cK \rightarrow \DD_{A/S} (\DD_{A/S} \cK)$ is an isomorphism. 
\end{cor}

\begin{proof}
In any symmetric monoidal category, any dualizable object is isomorphic to its double dual. 
\end{proof}



\begin{prop}\label{prop: USLA c is dualizable} Let $\cK \in \Dmotg{A}$ be USLA over $S$. Then $(A,\cK)$ is dualizable in $\LZ(S)_!$ or $\LZ(S)^*$.
\end{prop}

\begin{proof}
We will show that $(A, \DD_{A/S}(\cK))$ is dual to $(A, \cK)$ by explicitly constructing the evaluation and coevaluation morphisms (satisfying the necessary properties). In fact, the construction of the evaluation morphism does not invoke the USLA hypothesis: it is given by the correspondence 
\[
\begin{tikzcd}
& A \ar[dl, "\Delta"'] \ar[dr] \\
A \times_S A& & S
\end{tikzcd}
\]
equipped with the tautological cohomological correspondence 
\[
\Delta^*(\cK \boxtimes_S \DD_{A/S} \cK) \cong \cK \otimes_S \DD_{A/S}(\cK) \rightarrow \DD_{A/S}.
\]

The coevaluation morphism will have underlying correspondence
\[
\begin{tikzcd}
& A \ar[dl] \ar[dr, "\Delta"] \\
S & & A \times_S A
\end{tikzcd}
\]
Using Proposition \ref{prop: ULA hom tensor}, we have the following isomorphisms in $\Dmot{A}$: 
\[
\begin{tikzcd}[column sep = huge]
\Delta^! (\cK \boxtimes_S \DD_{A/S} \cK) \ar[r, "\sim", "\text{Prop \ref{prop: ULA hom tensor}}"'] & \Delta^! \cHom_{A\times_S A}(\pr_1^* \cK, \pr_0^! \cK) \cong \cHom_A( \cK, \cK).
\end{tikzcd}
\]
Thus $\Id_{\cK} \in \Hom_A(\cK, \cK)$ induces a map $\Qsh{A} \rightarrow \Delta^! (\cK \boxtimes_S \DD_{A/S} \cK)$, which defines the coevaluation morphism. 

The composite $(\Id_c \otimes \ev_c) \circ (\coev_c \otimes \Id_c )$ is supported on the outer correspondence in 
\[
\begin{tikzcd}
& & A \ar[dl, "\Delta"'] \ar[dr, "\Delta"]  \\
& A \times_S A \ar[dl, "\pr_0"'] \ar[dr, "\pr_{01} \times \Id"'] & & A \times_S A \ar[dl, "\Id \times \pr_{12}"] \ar[dr, "\pr_1"] \\
A & & A \times_S A \times_S A & &  A
\end{tikzcd}
\]
and unraveling the definitions shows that the resulting cohomological correspondence is isomorphic to the identity morphism. A similar analysis applies to $(\ev_c \otimes \Id_{c^\vee}) \circ (\Id_{c^\vee} \otimes \coev_c)$. 
\end{proof}

\subsection{Categorical traces}\label{ssec: traces}
Recall that any endomorphism $t \in \End(c)$ of a dualizable object $c$ in any symmetric monoidal category $(\msf{C}, \otimes)$ with unit $\un_{\msf{C}}$ has a notion of \emph{trace} $\Tr(t) \in \End(\un_{\msf{C}})$, defined as the composite
\[
\un_{\msf{C}} \xrightarrow{\coev_c} c \otimes c^\vee \xrightarrow{t \otimes \Id} c \otimes c^\vee  \cong c^\vee \otimes c \xrightarrow{\ev_c} \un_{\msf{C}}.
\]
If $\msf{C}$ is a 2-category, then $\End(\un_{\msf{C}})$ forms a 1-category, denoted $\Omega \msf{C}$. 

Specializing this construction, let $S$ be a derived Artin stack locally of finite type over a field.
We obtain two categories $\Omega \LZ(S)_!$ and $\Omega\LZ(S)^*$ with the same objects: in both cases the objects are triples $(F, i, \alpha)$ where $F$ is a derived Artin stack locally of finite type over $S$ (identified with the correspondence $S \leftarrow F \rightarrow S$), $i \in \Z$, and $\alpha  \in  \Corr_F(\Qsh{S}, \Qsh{S}\tw{-i}) = \CH_{i}(F/S)$ is a relative Chow cycle over $S$ of degree $i$.

The morphisms in $\Omega \LZ(S)_!$ and $\Omega\LZ(S)^*$ differ:
\begin{itemize}
\item In $\Omega \LZ(S)_!$, morphisms $(F, i, \alpha) \rightarrow (F',j, \beta)$ are proper maps $f \co F \rightarrow F'$ over $S$ such that $f_* \alpha = \beta \in \CH_{j}(F'/S)$ (so that $i=j$). 
\item In $\Omega \LZ(S)^*$, morphisms $(F,i, \alpha) \rightarrow (F', j, \beta)$ are quasi-smooth maps $f \co F \rightarrow F'$ over $S$ such that $f^! \beta = \alpha  \in \CH_{i}(F/S)$ (so that $i-j = d(f)$ is the relative dimension of $f$). 
\end{itemize}

Given a dualizable object $(A, \cK) \in \LZ(S)_!$ or $\LZ(S)^*$, an endomorphism of $(A, \cK)$ in $\LZ(S)_!$ or $\LZ(S)^*$ consists of a correspondence $(A \xleftarrow{c_0} C \xrightarrow{c_1} A)$, an integer $i \in \Z$, and a cohomological correspondence $\cc \in \Hom_C(c_0^* \cK, c_1^! \cK \tw{-i})$. The \emph{trace} of such an endomorphism is the triple $(\Fix(C) , i, \alpha \in \CH_{i}(\Fix(C)/S)$). We will refer to $\alpha$ as the \emph{trace of $\cc$}, and write 
\begin{equation}
\Tr_C(\cc) \coloneqq \alpha \in \CH_{i}(\Fix(C)/S).
\end{equation}

In other words, for every correspondence $(A \xleftarrow{c_0} C \xrightarrow{c_1} A)$ of derived Artin stacks over $S$ and every $\cK \in \Dmot{A}$, we obtain a canonical linear map
\begin{equation}
\Tr_C \co \Corr_C(\cK, \cK \tw{-i}) \to \CH_i(\Fix(C)/S).
\end{equation}
In particular, if $S=\Spec(\F)$, we obtain a trace map valued in cycle classes on $\Fix(C)$.

\subsection{Lefschetz-Verdier pairings}\label{ssec: Lefschetz-Verdier}
The general theory of pairings in symmetric monoidal 2-categories is documented in \cite[\S 1]{LZ22}. In particular, given a symmetric monoidal 2-category $(\msf{C}, \otimes)$ and morphisms $u \co c \rightarrow d$ and $v \co d \rightarrow c$ in $\msf{C}$, with $c$ dualizable, we have the pairing 
\begin{equation}
\pair{u,v} \coloneqq \Tr(v \circ u) \in \Omega \msf{C}. 
\end{equation}

\begin{example}
If $d = c$ and $v = \Id_c$, we have
\[
\pair{u, \Id_c} = \Tr(u) \in \Omega \msf{C}. 
\]
\end{example}

Suppose we have a diagram in $\msf{C}$
\begin{equation}\label{e.twosquare}
\xymatrix{c\ar[r]^u\ar[d]_f & d\ar[r]^v \ar[d]^g & c\ar[d]^f\\
c'\ar[r]^{u'} & d'\ar[r]^{v'}\ultwocell\omit{^\alpha} & c'\ultwocell\omit{^\beta}}
\end{equation}
with $c$ and $c'$ dualizable. Then by \cite[Construction 1.8]{LZ22} we have a morphism 
\begin{equation}\label{eq: fix pairing morphism}
\langle u,v\rangle\to \langle u',v'\rangle \in \Omega \msf{C}.
\end{equation}

\subsubsection{Compatibility with pushforward}

Let 
\begin{equation}\label{diag: lefschetz-verdier push}
\begin{tikzcd}
A_0 \ar[d, "f_0"] & C \ar[l, "c_0"'] \ar[d, "f"] \ar[r, "c_1"] & A_1  \ar[d, "f_1"] & D \ar[l, "d_1"'] \ar[r, "d_0"] \ar[d, "g"]  & A_0 \ar[d, "f_0"]\\
A'_0 & C' \ar[l, "c_0'"'] \ar[r, "c_1'"] & A_1' & D' \ar[l, "d_1'"'] \ar[r, "d_0"] & A'_0
\end{tikzcd}
\end{equation}
be a commutative diagram of derived Artin stacks over $S$. Write $E$ (resp. $E')$ for the composite correspondence of $C$ and $D$ (resp. $C'$ and $D'$) and $h$ for the induced map $h \co E \rightarrow E'$. 

\begin{thm}\label{thm: proper push pairing}
Assume in \eqref{diag: lefschetz-verdier push} that $f$ is proper, each $f_i$ is separated, and the square with verticles $A_1, A_1', D, D'$ is pushable. Let $\cK_0 \in \Dmotg{A_0}$ be such that $\cK_0$ is USLA over $S$ and $f_{0!} \cK_0 \in \Dmotg{A_0}$ is USLA over $S$.\footnote{This second assumption is automatic if $f_0$ is proper, by Lemma \ref{lem: proper pushforward preserves ULA}, which will be satisfied in all the situations where we will apply Theorem \ref{thm: proper push pairing}.} Let $\cK_1 \in \Dmotg{A_1}$ and $u \co c_0^* \cK_0 \rightarrow c_1^! \cK_1\tw{-i}$ and $v \co d_1^* \cK_1 \rightarrow d_0^! \cK_0\tw{-j}$. View $u$ and $v$ as morphisms in $\Hom_{\LZ(S)_!}((A_0, \cK_0), (A_1, \cK_1))$ and $\Hom_{\LZ(S)_!}((A_1, \cK_1), (A_0, \cK_0))$, respectively. 

Then $\Fix(h) \co \Fix(E) \rightarrow \Fix(E')$ is proper and\footnote{In forming this trace, we are implicitly using that $f_{0!} \cK$ is geometric by Remark \ref{rem: preservation of constructibility} (since $f_0$ is assumed to be representable).}
\[
\Fix(h)_* \pair{u,v} = \pair{f_! u, g_! v} \in \CH_{i+j}(\Fix(E')).
\]
\end{thm}

\begin{proof}
The analogous result for schemes and $\ell$-adic coefficients is \cite[Theorem 2.21]{LZ22}, and the argument is the same in essence. From the assumed left pushability of $g$ regarded as a map of correspondences, applying \eqref{eq: push 2-morphism} gives a commutative diagram 
\begin{equation}\label{eq: push rLV 1}
\xymatrix{
(A_1,\cK_1) \ar[d]^{f_{1\natural}}\ar[r]^{v} & (A_0,\cK_0)\ar[d]^{f_{0\natural}}\\
(A_1',f_{1!} \cK_1)\ar[r]_{g_!v} & (A'_0,f_{0!} \cK_0)\ultwocell\omit{^{}}}
\end{equation}
Decompose the left side of \eqref{diag: lefschetz-verdier push} as 
\begin{equation}\label{eq: push rLV 2}
\xymatrix{A_0 \ar[d]_{f_0} & C\ar[l]_{c_0}\ar@{=}[d]\ar[r]^{c_1} & A_1\ar@{=}[d]\\
A'_0\ar@{=}[d] & C\ar[l]_{f_0c_0}\ar[d]^{f}\ar[r]^{c_1} & A_1\ar[d]^{f_1}\\
A'_0  & C'\ar[l]_{c_0'}\ar[r]^{c_1'} & A_1'}
\end{equation}
The bottom left square is pushable by the assumption that $f$ is proper. The top left square is pushable by the assumption that $f_0$ is separated. Hence we may apply \eqref{eq: push 2-morphism} to obtain a commutative diagram 
\begin{equation}\label{eq: push rLV 3} 
\xymatrix{ 	
 (A_0,\cK_0) \ar[d]^{f_{0\natural}} \ar[r]^{u} & (A_1,\cK_1) \ar@{=}[d]  \\
(A_0',f_{0!} \cK_0) \ar[r]^w \ar@{=}[d] & (A_1,\cK_1) \ar[d]^{f_{1\natural}}  \ultwocell\omit{^{}}  \\
(A_0',f_{0!} \cK_0) \ar[r]_{f_!u}  & (A_1',f_{1!} \cK_1)  \ultwocell\omit{^{}}
}
\end{equation}
Concatenating \eqref{eq: push rLV 3} with \eqref{eq: push rLV 1} gives the commutative diagram 
\[
\xymatrix{(A_0,\cK_0)\ar[r]^{u}_\Downarrow\ar[d]_{f_{0\natural}} & (\cA_1,\cK_1) \ar[d]^{f_{1\natural}} \ar[r]^{v} & (A_0, \cK_0)\ar[d]^{f_{0\natural}}\\
(A'_0,f_{0!}\cK_0)\ar[r]_{f_!u}^\Downarrow\ar[ru]|{w} & (A_1',f_{1!} \cK_1 )\ar[r]_{g_	!v} & (A_0',f_{0!}\cK_0)\ultwocell\omit{^{}}}
\]
Then by \eqref{eq: fix pairing morphism} we have a map $(\Fix(E), \pair{u,v} ) \rightarrow (\Fix(E'), \pair{f_! u, g_! v})$ in $\Omega \LZ(S)_!$, which by definition entails the equality $\Fix(f)_* \pair{u,v} = \pair{f_! u, g_! v}$. 
\end{proof}

\subsubsection{Compatibility with pullback}
Consider the same diagram \eqref{diag: lefschetz-verdier push} and again write $E$ (resp. $E')$ for the composite correspondence of $C$ and $D$ (resp. $C'$ and $D'$) and $h$ for the induced map $h \co E \rightarrow E'$. 

\begin{thm}\label{thm: qsm pull pairing} Assume in \eqref{diag: lefschetz-verdier push} that $f$ is quasi-smooth and each $f_i$ is smooth, and the square with vertices $D,D',A_0,A_0'$ is pullable. Let $\cK_0' \in \Dmotg{A'_0}$ be such that $\cK_0'$ is USLA over $S$ (hence $f_0^* \cK_0' \in \Dmotg{A_0}$ is USLA over $S$ by Lemma \ref{lem: USLA smooth local}(1)). Let $\cK_1' \in \Dmotg{A_1'}$ and $u' \co (c_0')^* \cK_0' \rightarrow (c_1')^! \cK_1' \tw{-i}$ and $v \co (d_0')^* \cK_1' \rightarrow (d_1')^! \cK_0'\tw{-j}$. Identify $u'$ and $v'$ with the corresponding morphisms in $\Hom_{\LZ(S)^*}((A_0', \cK_0'), (A_1', \cK_1'))$ and $\Hom_{\LZ(S)^*}((A_1', \cK_1'), (A_0', \cK_0'))$, respectively. 

Then $\Fix(h) \co \Fix(E) \rightarrow \Fix(E)'$ is quasi-smooth and\footnote{In forming this trace, we are implicitly using that $f_0^* \cK$ is geometric by Remark \ref{rem: preservation of constructibility}.}
\[
\Fix(h)^! \pair{u',v'} = \pair{f^* u', g^* v'} \in \CH_{i+j +  \delta_h}(\Fix(E)).
\]
\end{thm}

\begin{proof}
From the pullability of the right half one has a 2-morphism in $\LZ(S)^*$, applying \eqref{eq: pull 2-morphism} gives a commutative diagram 
\begin{equation}\label{eq: pull rLV 1}
\xymatrix{
(A_1,f_1^*\cK_1')  \ar[r]^{g^* v'} & (A_0, f_0^* \cK_0') \dltwocell\omit{^{}} \\
(A_1',\cK_1')\ar[r]_{v'} \ar[u]^{f_1^\natural
} & (A_0',\cK_0') \ar[u]^{f_0^\natural}
}
\end{equation}
Decompose the left side of \eqref{diag: lefschetz-verdier push} as
\begin{equation}\label{eq: pull rLV 2}
\xymatrix{A_0 \ar@{=}[d] & C\ar[l]_{c_0}\ar@{=}[d]\ar[r]^{c_1} & A_1 \ar[d]^{f_1}\\
A_0 \ar[d]_{f_0} & C\ar[l]_{c_0}\ar[d]^{f}\ar[r]^{f_1 c_1} & A_1'\ar@{=}[d]\\
A_0' & C'\ar[l]_{c_0'}\ar[r]^{c_1'} & A_1'}
\end{equation}
The bottom right square is pullable by the assumption that $f$ is quasi-smooth. The top right square is pullable by the assumption that $f_1$ is smooth. Hence we may apply \eqref{eq: pull 2-morphism} to obtain a commutative diagram 
\begin{equation}\label{eq: pull rLV 3} 
\xymatrix{ 
 (A_0, f_0^* \cK_0')\ar@{=}[d] \ar[r]^{f^* u'} & (A_1, f_1^*\cK_1')  \dltwocell\omit{^{}} \\
(A_0,f_0^*\cK_0')\ar[r]_{w} & (A_1',\cK_1') \ar@{=}[d]   \dltwocell\omit{^{}}  \ar[u]^{f_1^\natural} \\
(A_0', \cK_0') \ar[r]_{u'} \ar[u]^{f_0^\natural}  & (A_1',\cK_1')  
}
\end{equation}
Concatenating \eqref{eq: pull rLV 3} with \eqref{eq: push rLV 1} gives the commutative diagram 
\[
\xymatrix{
(A_0,f_0^*\cK_0')\ar[r]^{f^* u'}_\Downarrow \ar[dr]|{w}  & (A_1,f_1^* \cK_1')   \ar[r]^{g^* v'} & (A_0, f_0^*\cK_0')  \dltwocell\omit{^{}} \\
(A_0',\cK_0') \ar[u]_{f_0^\natural}  \ar[r]_{ u' }^\Downarrow  & (A_1',\cK_1')\ar[r]_{ v'} \ar[u]^{f_1^\natural} & (A'_0,\cK_0') \ar[u]^{f_0^\natural}
}
\]
Then by \eqref{eq: fix pairing morphism} we have a map $(\Fix(E), \pair{f^*u',f^*v'} ) \rightarrow (\Fix(E'), \pair{u', v'}) \in \Omega \LZ(S)^*$, which by definition entails the equality $\Fix(h)^! \pair{u',v'} = \pair{f^* u', g^* v'}$. 
\end{proof}

\section{Motivic sheaf-cycle correspondence}\label{sec: motivic sheaf-cycle} In this section we develop the motivic sheaf-cycle correspondence, which is based on the trace map constructed in \S \ref{ssec: traces}.
This material is a motivic lift of \cite[\S 4]{FYZ3}, much of which carries over verbatim (in those cases we omit the proofs). The only aspect that is substantially different is Proposition~\ref{prop: trace commutes with smooth pullback}, which we prove using the motivic Lu--Zheng categorical trace.

\subsection{The trace}

Let
\begin{equation}
\begin{tikzcd}
A & C \ar[l, "c_0"'] \ar[r, "c_1"]  & A
\end{tikzcd}
\end{equation}
be a correspondence of derived Artin stacks over a field $\F$.
For every $\cK \in \Dmotg{A}$ and $i\in\bZ$, the construction of \S \ref{ssec: traces} gives a trace map
\begin{equation}
\Tr_C \co \Corr_C(\cK, \cK \tw{-i}) \to \CH_i(\Fix(C)).
\end{equation}
Indeed, since $\cK$ is geometric (by assumption) and USLA over $\Spec (k)$ (Example~\ref{ex: USLA over point}), Proposition~\ref{prop: USLA c is dualizable} implies that $(A, \cK)$ is dualizable in $\LZ(\F)^*$ or $\LZ(\F)_!$. 

\subsection{Functoriality of the trace}\label{ssec: trace cc functoriality} We study the compatibility of pushforward and pullback on cohomological correspondences and Chow groups (see \S \ref{ssec: functoriality for cohomological correspondences}) under formation of traces.
We consider a map of correspondences of derived Artin stacks over a field $\F$:
\begin{equation}\label{eq: functoriality for trace}
\begin{tikzcd}
A \ar[d, "f_0"'] & C \ar[l, "c_0"'] \ar[r, "c_1"]  \ar[d, "f"]  & A  \ar[d, "f_1=f_0"] \\
B  & D \ar[l, "d_0"'] \ar[r, "d_1"]  & B  
\end{tikzcd}
\end{equation}

\subsubsection{Proper pushforward} The following result is a motivic lift of \cite[Proposition 4.5.1]{FYZ3}. 

\begin{prop}\label{prop: trace commutes with proper push}
With notation as in \eqref{eq: functoriality for trace}, assume that $f_0, f$ are proper and let $\cc\in \Corr_{C}(\cK,\cK\tw{-i})$ where $\cK \in \Dmotg{A}$. Then $\Fix(f) \co \Fix(C) \rightarrow \Fix(D)$ is proper and we have
\[
\Tr_D(f_{!}  \cc ) =\Fix(f)_* (\Tr_C(\cc)) \in \CH_{i}(\Fix(D)).
\]
\end{prop}

\begin{proof}Apply Theorem \ref{thm: proper push pairing} with $S \coloneqq \Spec \F$, $A_0 = A_1 \coloneqq A$, $\cK_0 = \cK_1 \coloneqq \cK$, $u  \coloneqq \cc$, and $v \coloneqq \Id$. 
\end{proof}

\subsubsection{Quasi-smooth pullback} The following result is a motivic lift of \cite[Proposition 4.5.4]{FYZ3}.

\begin{prop}\label{prop: trace commutes with smooth pullback}
With notation as in \eqref{eq: functoriality for trace}, assume that $f_1$ is smooth and  $f$ is quasi-smooth. Let  $\d=d(f)-d(f_{1})$. Let $\mf{d} \in \Corr_{D}(\cK,\cK\tw{-i})$ where $\cK \in \Dmotg{B}$. Then $\Fix(f) \co \Fix(C) \rightarrow \Fix(D)$ is quasi-smooth and we have
\[
\Tr_C(f^{*} \mf{d}) = \Fix(f)^{!} (\Tr_D(\mf{d})) \in \CH_{i+\d}(\Fix(C)).
\]
\end{prop}

\begin{proof}Apply Theorem \ref{thm: qsm pull pairing} with $S \coloneqq \Spec \F$, $A_0 = A_1 \coloneqq A$, $\cK_0 = \cK_1 \coloneqq \cK$, $u  \coloneqq \dd$, and $v \coloneqq \Id$. 
\end{proof}

\subsection{The fundamental class as a trace}\label{ssec: VFC as trace}

Let $A \xleftarrow{c_0} C \xrightarrow{c_1} A$ be a correspondence of derived Artin stacks over $\F$.
If $c_1$ is quasi-smooth and $A$ is smooth over $\F$, then $\Fix(C)$ is quasi-smooth over $\F$ of relative dimension $d(c_1)$, so there is a fundamental class (\S \ref{ssec:Gys})
\begin{equation}
[\Fix(C)] \in \CH_{d(c_1)}(\Fix(C)).
\end{equation}

On the other hand, regarding the relative fundamental class $[c_1] \in \CH_{d(c_1)}(C/ A)$ as a map $c_1^* \Qsh{A} \rightarrow c_1^! \Qsh{A} \tw{-d(c_1)}$ in $\Dmot{C}$, the composite
\[
c_0^* \Qsh{A} \cong \Qsh{A} \cong c_1^* \Qsh{A} \xrightarrow{[c_1]} c_1^! \Qsh{A} \tw{-d(c_1)}
\]
defines a cohomological correspondence $\cc_{A} \in \Corr_C(\Qsh{A}, \Qsh{A}\tw{-d(c_1)})$.

\begin{prop}\label{prop: smooth derived local terms} If $c_1$ is quasi-smooth and $A$ is smooth, then we have
\[
[\Fix(C)] = \Tr_C(\cc_A) \in \CH_{d(c_1)} (\Fix(C)).
\]
\end{prop}

\begin{proof}
Consider the map of correspondences 
\begin{equation}\label{eq: pull from point}
\begin{tikzcd}
A  \ar[d, "\pi_0"] & \ar[l, "c_0"'] C \ar[r, "c_1"] \ar[d, "\pi"]  & A \ar[d, "\pi_1"] \\
\pt &  \ar[l] \pt \ar[r] & \pt 
\end{tikzcd}
\end{equation}
where $\pt=\Spec(\F)$.
By assumption $\pi_1$ is smooth and $c_1$ is quasi-smooth, so $\pi$ is quasi-smooth and \eqref{eq: pull from point} is right pullable.
Unravelling definitions, we can write $\cc_A$ as the pullback of the trivial cohomological correspondence $\cc_{\pt} \in \Corr_{\pt}(\Q, \Q)$ (cf. \cite[Lemma 9.4.2]{FYZ3}):
\[
\pi^* \cc_{\pt} = \cc_A \in \Corr_C(\Qsh{A}, \Qsh{A}\tw{-d(c_1)}).
\]
By Proposition~\ref{prop: trace commutes with smooth pullback} we have
\[
\Tr_C(\pi^* \cc_{\pt}) = \Fix(\pi)^! (\Tr_{\pt}(\cc_{\pt})) = \Fix(\pi)^![\pt] = [\Fix(C)] \in \CH_{d(c_1)}(\Fix(C)),
\]
whence the claim.
\end{proof}

\subsection{Frobenius-twisted trace} \label{ssec: tr sht}

\subsubsection{Frobenius}

We take $\F = \F_q$ to be a finite field.
Any derived Artin stack over $\F_q$ is equipped with a Frobenius endomorphism $\Frob$, which in terms of the functor of points is the absolute Frobenius $\Frob_q$ on the test scheme.
That is, for any derived Artin stack $X$ over $\F_q$, we denote by $\Frob : X \to X$ the morphism sending an $R$-point $x : \Spec(R) \to X$ to the composite
\[ \Spec(R) \xrightarrow{\Frob_q} \Spec(R) \xrightarrow{x} X \]
for every commutative $\F_q$-algebra $R$.

\subsubsection{Fix vs. Sht}\label{sssec: fix vs sht}

For a correspondence
\[
\begin{tikzcd}
 A & \ar[l, "c_0"'] C \ar[r, "c_1"] & A
\end{tikzcd}
\]
over $\F_q$, we will let $\Sht(C)$ (or sometimes $\Sht_A$) be the derived fibered product 
\begin{equation}\label{eq: fix vs sht 1}
\begin{tikzcd}
\Sht(C) \ar[r] \ar[d] & C \ar[d, "{(c_0, c_1)}"] \\
A \ar[r, "{(\Id, \Frob)}"] & A \times A
\end{tikzcd}
\end{equation}
This derived fibered product can be also be presented by the derived Cartesian square 
\begin{equation}\label{eq: fix vs sht 2}
\begin{tikzcd}
\Sht(C) \ar[r] \ar[d] & C \ar[d, "{(\Frob \circ c_0, c_1)}"] \\
A \ar[r, "\Delta"] & A \times A
\end{tikzcd}
\end{equation}
which is the ``fixed point Cartesian square'' for the correspondence 
\begin{equation}\label{eq: fix vs sht 3}
\begin{tikzcd}
A & \ar[l, "\Frob \circ c_0"'] C^{(1)} \ar[r, "c_1"] & A
\end{tikzcd}
\end{equation}
where $C^{(1)} \coloneqq C$ but with the left map twisted by $\Frob$. In other words, we have a canonical identification
\begin{equation}
\Sht(C)=\Fix(C^{(1)}).
\end{equation}

\subsubsection{Sht-valued trace}\label{sssec: sht val tr}

Given $\cK_0, \cK_1 \in \Dmot{A}$ and a cohomological correspondence $\cc \co c_0^* \cK_0 \rightarrow c_1^! \cK_1$ on $C$, plus the canonical Weil structure $\Frob^* \cK_0 \cong \cK_0$ (because $A$ is defined over $\F_{q}$), we have a cohomological correspondence $\cc^{(1)} \co (\Frob \circ c_0)^* \cK_0 \rightarrow c_1^! \cK_1$. In this way we obtain a linear isomorphism
\begin{equation*}
\Corr_{C}(\cK_{0}, \cK_{1})\isom \Corr_{C^{(1)}}(\cK_{0}, \cK_{1})
\end{equation*}
sending $\cc \mapsto \cc^{(1)}$. If $\cK_0$ is geometric and $\cK_1 = \cK_0 \tw{-i}$, then we define 
\[
\Tr_C^{\Sht}(\cc)\coloneqq\Tr_C(\cc^{(1)}) \in \CH_{i}( \Fix(C^{(1)}))=\CH_{i}( \Sht(C)).
\]
This determines a linear map
\begin{equation}
\Tr_C^{\Sht} : \Corr_C(\cK_0, \cK_0\vb{-i}) \to \CH_i(\Sht(C)).
\end{equation}

\subsubsection{The fundamental class of \texorpdfstring{$\Sht(C)$}{Sht(C)}}

In the situation of \S \ref{sssec: fix vs sht}, Proposition~\ref{prop: smooth derived local terms} yields:

\begin{cor}\label{cor: smooth derived local terms Sht}
If $c_1$ is quasi-smooth and $A$ is smooth over $\F_q$, then we have 
\[
 \Tr_{C}^{\Sht}(\cc_A) = [\Sht(C)] \in \CH_{d(c_1)}(\Sht(C))
\]
where $\Qsh{A}$ is equipped with its natural Weil structure.
\end{cor}

\begin{remark}
It is interesting to ask in what generality Corollary~\ref{cor: smooth derived local terms Sht} holds without the smoothness of $A$.
Indeed, it is shown in \cite[Lemma 4.2.1]{FYZ3} that $\Sht(C)$ is quasi-smooth (and hence admits a fundamental class) as long as $c_1$ is quasi-smooth.
\end{remark}

\subsection{Shift and twist}\label{ssec: trace variants}

\subsubsection{}

Let $A \xleftarrow{c_0} C \xrightarrow{c_1} A$ be a correspondence over a field $\Spec \F$. Given $\cK \in \Dmot{A}$ and $\cc\in \Corr_{C}(\cK_{0},\cK_{1})$, the map $\cc: c_{0}^{*}\cK_{0}\to c_{1}^{!}\cK_{1}$ induces for every $m,n\in\Z$ a map
\begin{equation*}
c_{0}^{*}\cK_{0}[m](n)\to c_{1}^{!}\cK_{1}[m](n)
\end{equation*}
which we denote by $\TT_{[m](n)}\cc$.
The assignment $\cc\mapsto \TT_{[m](n)}\cc$ defines an isomorphism
\begin{equation*}
\TT_{[m](n)} \co \Corr_{C}(\cK_{0},\cK_{1})\isom\Corr_{C}(\cK_{0}[m](n),\cK_{1}[m](n)).
\end{equation*}

\subsubsection{}

The trace map $\Tr_C \co \Corr_{C}(\cK,\cK\tw{-i}) 
\rightarrow \CH_{i}(\Fix(C))$ satisfies the identity
\begin{equation*}
\Tr_C(\TT_{[m](n)}\cc)=(-1)^{m}\cdot \Tr_C(\cc)\in \CH_{i}(\Fix(C)).
\end{equation*}

\subsubsection{Sht-valued trace}

Suppose $\F=\F_q$ and consider the map $\Tr^{\Sht}_C \co \Corr_{C}(\cK,\cK\tw{-i}) \rightarrow \CH_{i}(\Sht(C))$ (\S\ref{sssec: sht val tr}).
We have the identity
\begin{equation}\label{eq: Tr Sht sh tw}
\Tr_C^{\Sht}(\TT_{[m](n)}\cc)=(-1)^{m}q^{-n} \cdot \Tr_C^{\Sht}(\cc)\in \CH_{i}(\Sht(C)).
\end{equation}

\section{Specialization and motivic local terms}\label{sec: local terms}

The main result of this section is Theorem \ref{thm: contracting fixed terms}, which says that for a correspondence $c = (Y \leftarrow C \rightarrow Y)$ of derived Artin stacks over a field $\F$, the trace of a cohomological correspondence supported on $c$ can be calculated after restriction to a closed substack $Z \inj Y$, provided that $c$ is ``contracting near $Z$''. We refer to Definition~\ref{def: contracting correspondence} for the meaning of the latter condition; for now we just mention that it is a condition on classical truncations. In fact, the entirety of this section deals only with properties and constructions of underlying classical truncations. Hence \textbf{for this section alone, we change our default conventions so that all constructions (fibered products, etc.) occur within classical algebraic geometry}.

The results in this section have previously appeared for schemes and $\ell$-adic coefficients in \cite{Var07}, and then for schemes and motivic coefficients in \cite{Jin23}. Our only contribution is of a technical nature: we generalize their arguments to (higher) Artin stacks and motivic coefficients. This is needed in applications, in the present paper as well as in in other work-in-progress.

\subsection{Ayoub's nearby cycles functor} We work over a field $\F$. Let $i \co s \inj \A^1_{\F}$ be the origin and $j \co \eta \inj \A^1_{\F}$ its complement. Suppose we have a morphism of Artin stacks $f \co Y \rightarrow \A^1_{\F}$. For $? \in \{s, \eta\}$ the subscript $?$ will denote base change to $?$. Thus we have a commutative diagram 

\begin{equation}\label{eq: specialization diagram}
\begin{tikzcd}
Y_s \ar[d, "f_s"] \ar[r, hook, "i_Y"] & Y \ar[d, "f"] & Y_\eta \ar[d, "f_\eta"] \ar[l, hook', "j_Y"']  \\
s \ar[r, "i", hook] & \A^1_\F & \eta \ar[l, hook', "j"'] 
\end{tikzcd}
\end{equation}
where the squares are Cartesian. 

Ayoub constructed and analyzed the \emph{motivic nearby cycles functor} on schemes, in \cite{Ay07a, Ay07b, Ay14}. In \cite[\S A.2]{HPL22}, Ayoub's construction of the \emph{tame nearby cycles functor} (part of the total motivic nearby cycles) is extended to Artin stacks for $\Dmot{-}$, essentially by repeating Ayoub's construction verbatim. We denote this functor
\[
\Psi^{\mrm{t}}_Y \co \Dmot{Y_\eta} \rightarrow \Dmot{Y_s}.
\]
It satisfies the following properties:
\begin{enumerate}
\item $\Psi^{\mrm{t}}$ is lax-monoidal, so in particular there are binatural transformations 
\begin{equation}\label{eq: Psi lax monoidal}
\Psi^{\mrm{t}}_{Y}(\cK) \otimes  \Psi^{\mrm{t}}_{Y}(\cK') \rightarrow \Psi^{\mrm{t}}_{Y} (\cK \otimes \cK') \in \Dmot{Y_s}
\end{equation}
for all $\cK,\cK'  \in \Dmot{Y_\eta}$. If $Y_0, Y_1$ are Artin stacks over $\A^1_\F$, then as a special case of \eqref{eq: Psi lax monoidal} we have natural maps\footnote{If we worked with the total nearby cycles functor instead of the tame part, then these maps would be isomorphisms.} 
\begin{equation}\label{eq: Psi external tensor}
\Psi^{\mrm{t}}_{Y_0}(\cK_0) \boxtimes_s \Psi^{\mrm{t}}_{Y_1}(\cK_1) \rightarrow \Psi^{\mrm{t}}_{Y_0 \times_{\A^1_\F} Y_1} (\cK_0 \boxtimes_\eta \cK_1) \in \Dmot{Y_{0s} \times_s Y_{1s}}
\end{equation}
for all $\cK_i \in \Dmot{Y_i}$. 

\item For any morphism $g \co Y' \rightarrow Y$, there are natural transformations
\begin{equation}\label{eq: pull Psi}
g_s^* \circ \Psi^{\mrm{t}}_Y \rightarrow \Psi^{\mrm{t}}_{Y'} \circ g_\eta^* \co \Dmot{Y_\eta} \rightarrow \Dmot{Y'_s}
\end{equation}
and 
\begin{equation}\label{eq: Psi shriek pull}
\Psi^{\mrm{t}}_{Y'} \circ  g_\eta^! \rightarrow g_s^! \circ \Psi^{\mrm{t}}_{Y}  \co \Dmot{Y_\eta} \rightarrow \Dmot{Y'_s},
\end{equation}
which are both isomorphisms if $g$ is smooth. 

\item The are natural transformations 
\begin{equation}\label{eq: Psi push}
\Psi^{\mrm{t}}_Y \circ g_{\eta*} \rightarrow g_{s*}  \circ \Psi^{\mrm{t}}_{Y'} \co \Dmot{Y'_\eta} \rightarrow \Dmot{Y_s}
\end{equation}
and 
\begin{equation}\label{eq: shriek push Psi}
g_{s!} \circ \Psi^{\mrm{t}}_{Y'} \rightarrow \Psi^{\mrm{t}}_Y \circ g_{u!} \co \Dmot{Y'_\eta} \rightarrow \Dmot{Y_s},
\end{equation}
which are both isomorphisms if $g$ is proper.

\item $\Psi^{\mrm{t}}_Y$ commutes with shifts and Tate twists. 
\end{enumerate}

\begin{lemma}\label{lem: Psi geometric}
The functor $\Psi^{\mrm{t}}_Y$ preserves geometricity (in the sense of \S \ref{ssec: geometric motives}).
\end{lemma}

\begin{proof}
By \eqref{eq: pull Psi} and the definition of geometricity, the statement can be checked after base change to a smooth atlas $Y' \rightarrow Y$ where $Y'$ is a derived scheme.
Then the claim is \cite[Lemma 6.1.9(2)]{Jin23} (see also \cite[Theorem~3.5.14]{Ay07b} for the case where $\F$ is of characteristic zero).
\end{proof}

\subsection{Contracting correspondences}

Let $c = (Y \xleftarrow{c_0} C \xrightarrow{c_1} Y)$ be a correspondence of locally noetherian (classical) Artin stacks. The following definitions generalize those of Varshavsky in \cite[Definition 1.5.1 and Definition 2.1.1]{Var07}. 

\begin{defn}[Contracting correspondences]\label{def: contracting correspondence}
Let $Z \inj Y$ be a closed embedding defined by an ideal sheaf $\cI_Z \subset \cO_Y$.
\begin{enumerate}
\item We say that $Z$ is \emph{$c$-invariant} if $c_1^{-1}(Z)$ is set-theoretically contained in $c_0^{-1}(Z)$. 
\item We say that $c$ \emph{stabilizes} $Z$ if $c_0^* (\cI_Z) \subset c_1^*(\cI_Z)$ (i.e., $c_1^{-1}(Z)$ is scheme-theoretically contained in $c_0^{-1}(Z)$).
\item We say that $c$ is \emph{contracting near} $Z$ if $c$ stabilizes $Z$ and there exists $n \in \N$ such that $c_0^* (\cI_Z)^n \subset c_1^* (\cI_Z)^{n+1}$. 
\end{enumerate}
Note that (c) $\implies$ (b) $\implies$ (a).
\end{defn}

\begin{example}[Frobenius contracts]\label{ex: Frobenius twist contracting}
Suppose $Y$ is defined over a finite field $\F=\F_q$ and $c = (c_0, c_1) \co C \rightarrow Y \times Y$ is a correspondence stabilizing $Z \inj Y$. Then the Frobenius-twisted correspondence $c^{(1)} \coloneqq (\Frob \circ c_0, c_1) \co C \rightarrow Y \times Y$ (\S\ref{sssec: fix vs sht}) is contracting near $Z$.
Indeed, working locally on $Y$ we may assume $Y$ is noetherian, in which case this is proven in \cite[Lemma~2.2.3]{Var07}.
\end{example}

\begin{const}[Restricting correspondences to stable substacks]\label{const: restrict correspondence to closed} For any closed substack $Z \inj Y$ such that $c$ stabilizes $Z$, define the base-changed correspondence $c_Z$ by the upper row in the commutative diagram
\begin{equation}\label{eq: sub-correspondence}
\begin{tikzcd}
Z \ar[d, hook] & C_Z \ar[l, "c_{0Z}"'] \ar[r, "c_{1Z}"] \ar[d, hook] & Z \ar[d, hook] \\
Y & \ar[l, "c_0"']  C \ar[r, "c_1"] & Y
\end{tikzcd}
\end{equation}	
where the right square is Cartesian. (The hypothesis that $c$ stabilizes $Z$ is used to ensure that the pullback $C_Z \inj C \xrightarrow{c_0} Y$ factors over $Z$.) 
\end{const}

\begin{lemma}\label{lem: zero-section}
Let 
\[
\begin{tikzcd}
Z' \ar[r, hook] \ar[d] & Y' \ar[d, "f"] \\
Z \ar[r, hook] & Y
\end{tikzcd}
\]
be a commutative diagram of locally noetherian Artin stacks. Then the induced map of normal cones $N_{Z'}(Y') \rightarrow N_Z(Y)$ has set-theoretic image in the zero-section $Z \subset N_Z(Y)$ if and only if there exists $n$ such that $f^* (\cI_Z^n)  \subset \cI_{Z'}^{n+1} $. 
\end{lemma}

\begin{proof}
The same proof as in \cite[Lemma 1.4.3(b)]{Var07} works verbatim. 
\end{proof}

\subsection{Specialization} Next we discuss specialization of Chow groups and cohomological correspondences in families. 

\subsubsection{Specialization on Chow groups}\label{sssec: sp chow}

For a commutative diagram \eqref{eq: specialization diagram} in which all squares are Cartesian, there is constructed in \cite[\S 4.5.6]{DJK21}\footnote{%
Strictly speaking, \cite{DJK21} operated in the schematic context, but \cite{KhanI} generalizes all the ingredients of the construction to derived Artin stacks.
} a specialization map on relative Chow groups, 
\begin{equation}
\fsp_Y \co \CH_*(Y_\eta/\eta) \rightarrow \CH_*(Y_s/s).
\end{equation}

\subsubsection{Specialization on cohomological correspondences}\label{sssec: specialize cc}

Suppose we have a correspondence of Artin stacks $Y_0 \xleftarrow{c_0} C \xrightarrow{c_1} Y_1$ over $\A^1_\F$ and a cohomological correspondence $\cc \in \Corr_{C_\eta}(\cK_0, \cK_1)$ supported on $C_\eta$, where $\cK_i \in \Dmot{Y_{i\eta}}$.\footnote{We emphasize again that the subscripts $? \in \{\eta, s\}$ indicate base change to $?$.}
Then we have a cohomological correspondence
\begin{equation}\label{eq: Psi cc}
\Psi_C^{\mrm{t}}(\cc) \co c_{0s}^* \Psi^{\mrm{t}}_{Y_0}(\cK_0) \rightarrow c_{1s}^! \Psi^{\mrm{t}}_{Y_1}  (\cK_1)
\end{equation}
defined as the composition 
\[
\begin{tikzcd}
c_{0s}^* \Psi^{\mrm{t}}_{Y_0}(\cK_0) \ar[r, "\eqref{eq: pull Psi}"] & \Psi^{\mrm{t}}_{C} (c_{0 \eta}^* \cK_0) \ar[r, "\Psi^{\mrm{t}}_C(\cc)"] & \Psi^{\mrm{t}}_C (c_{1\eta}^! \cK_1 )\ar[r, "\eqref{eq: Psi shriek pull}"] &  c_{1s}^! (\Psi^{\mrm{t}}_{Y_1}  \cK_1).
\end{tikzcd}
\]
The assignment $\cc \mapsto \Psi_C^{\mrm{t}}(\cc)$ defines a map
\begin{equation}\label{eq:Psi Corr}
\Psi^{\mrm{t}} \co \Corr_{C_\eta}(\cK_{0}, \cK_1) \rightarrow \Corr_{C_s}(\Psi^{\mrm{t}}_{Y_0} (\cK_0) , \Psi^{\mrm{t}}_{Y_1} (\cK_1)).
\end{equation}


\subsubsection{Specialization vs. trace}

Under favorable conditions, the specialization of cohomological correspondences is compatible with formation of trace. This was shown by Varshavsky for $\ell$-adic sheaves on schemes in \cite[Proposition 1.3.5]{Var07}, and Jin adapted the argument to motivic sheaves on schemes in \cite[Lemma 6.2.4]{Jin23}. We record the statement in our more general context. 

\begin{prop}\label{prop: specialization compatible with trace}
Let $Y \xleftarrow{c_0} C \xrightarrow{c_1} Y$ be a correspondence of derived Artin stacks over $\A^1_\F$.
If $\cK \in \Dmotg{Y_\eta}$ is USLA over $\eta$, then the following diagram commutes:\footnote{To form the trace, we are implicitly using that $\Psi^{\mrm{t}}_{Y_0} \cK$ is dualizable in $\LZ(S)_!$ and $\LZ(S)^*$. This is because  $\Psi^{\mrm{t}}_{Y_0} \cK$ is USLA, by Example \ref{ex: USLA over point}, and geometric, by Lemma \ref{lem: Psi geometric}, hence dualizable by Proposition \ref{prop: USLA c is dualizable}.}
\begin{equation}
\begin{tikzcd}
\Corr_{C_\eta}(\cK_, \cK \tw{-i}) \ar[r, "\Psi_C^{\mrm{t}}"] \ar[d, "\Tr_{C_\eta}"]  &  \Corr_{C_s}(\Psi^{\mrm{t}}_{Y} \cK  , \Psi^{\mrm{t}}_{Y} \cK \tw{-i}) \ar[d, "\Tr_{C_s}"] \\
\CH_{i}(\Fix(C_\eta)/\eta) \ar[r, "\fsp_C"] & \CH_{i}(\Fix(C_s)/s)
\end{tikzcd}
\end{equation}
\end{prop}

\begin{proof}
By the same formal diagram chase as in the proof of \cite[Lemma 6.2.4]{Jin23}, this is reduced to the lax-monoidality of $\Psi^{\mrm{t}}$. 
\end{proof}

\subsection{Specialization to the normal cone}

Let $\iota \co Z \inj Y$ be a closed immersion of Artin stacks over $\F$. Then there is a \emph{deformation to the normal cone} $D_Z(Y)$, which is a family of Artin stacks over $\A^1_{\F}$ which restricts to the constant family $Y \times \G_m$ over $\eta = \G_m$ and the normal cone $N_Z(Y)$ over $s$.
It may be constructed by forming the blow-up of $Y \times \A^1$ along $Z \times \{0\}$ and taking out the blow-up of $Y \times \{0\}$ along $Z \times \{0\}$.
(We emphasize that we are considering the \emph{classical} deformation to the normal cone rather than the derived version.)

The construction of $D_Z(Y)$ is functorial in $Z$ and $Y$. Following Varshavsky \cite{Var07}, we use the notation $\wt{(-)}$ for constructions induced by deformation to the normal cone, and $\wt{(-)}_\eta$ or $\wt{(-)}_s$ for the base changes to $\eta$ or $s$, respectively. For example, given a commutative diagram 
\begin{equation}\label{diag: deformation normal cone functorial}
\begin{tikzcd}
Z' \ar[r, hook, "\iota'"]  \ar[d, "h"] &  Y' \ar[d, "g"] \\
Z \ar[r, "\iota"] & Y
\end{tikzcd}
\end{equation}
we get a map $ \wt{g} \co D_{Z'} (Y') \rightarrow D_Z(Y)$.

\sssec{Specialization of cycle classes}
\label{sssec: sp norm cyc}

Applying the specialization construction of \S\ref{sssec: sp chow} to $D_Z(Y)$, we get a map
\begin{equation}
\fsp_{D_Z(Y)} \co \CH_*(Y \times \eta/\eta) \to \CH_*(N_Z(Y)).
\end{equation}
Composing with the pullback map $\CH_*(Y) \to \CH_*(Y \times \eta/\eta)$, we get a map\footnote{%
This can be defined more directly as in \cite[Def.~3.2.4]{DJK21} or \cite[Constr.~3.1]{KhanI}, but this alternative description will be convenient for us.}
\begin{equation}
\fsp_{Y,Z} \co \CH_*(Y) \to \CH_*(N_Z(Y)).
\end{equation}

\sssec{Specialization of sheaves}
\label{sssec: sp norm sh}

Formation of nearby cycles with respect to $D_Z(Y)$ gives rise to a functor of specialization to the normal cone, 
\begin{equation}\label{eq: spc}
\spc_{Y,Z}	 \co \Dmot{Y} \rightarrow \Dmot{N_Z(Y)}
\end{equation}
defined by the formula
\begin{equation}
\spc_{Y,Z}(\cK) \coloneqq \Psi^{\mrm{t}}_{D_Z(Y)} (\pr^* \cK)
\end{equation}
where $\pr \co Y \times \G_m \rightarrow Y$ is the projection to the first factor.

Given a commutative square \eqref{diag: deformation normal cone functorial}, \eqref{eq: pull Psi} gives a natural transformation 
\begin{equation}\label{eq: specialization pullback}
\wt{g}^*_s \spc_{Y,Z} \rightarrow \spc_{Y', Z'} g^* \co \Dmot{Y} \rightarrow \Dmot{N_{Z'}(Y')}.
\end{equation}

\begin{example}
If $\iota \co Z \inj Y$ is an isomorphism, then $D_Z(Y)$ is the constant family $Z \times \A^1$, so $\spc_{Y,Z}= \Id$ in that case. Taking $Z' \inj Y'$ to be the identity map $Z = Z$ in \eqref{diag: deformation normal cone functorial}, we get a closed embedding $\wt{\iota} \co Z \times \A^1 \inj D_Z(Y)$, which restricts to the constant embedding $Z \times \G_m \inj Y \times \G_m$ over $\eta$ and the zero section $Z \inj N_Z(Y)$ over $s$. Hence in this case, \eqref{eq: specialization pullback} gives a map
\begin{equation}\label{eq: Verdier map}
\wt{\iota}_s^*  \spc_{Y,Z}(\cK) \rightarrow \iota^* \cK \in \Dmot{Z}.
\end{equation}
\end{example}

\begin{prop}[Verdier, Jin]\label{prop: Verdier}
For all $\cK \in \Dmot{Y}$, the map \eqref{eq: Verdier map} is an isomorphism. 
\end{prop}

\begin{proof}
The statement can be checked smooth locally on $Y$, so it reduces to the case where $Y$ (and hence $Z$) is a scheme. Then it is \cite[Proposition 6.3.14]{Jin23}.\footnote{%
For schemes and $\ell$-adic sheaves, the analogous statement appears in work of Verdier \cite[\S 8]{Ver83} with a sketch of proof. A full proof is given by Varshavsky in \cite[\S 3]{Var07}.}
\end{proof}

\subsubsection{Specialization of cohomological correspondences}
\label{sssec: sp norm corr}

Let $c = (Y \xleftarrow{c_0} C \xrightarrow{c_1} Y)$ be a correspondence of Artin stacks locally of finite type over a field $\F$. Suppose $c$ stabilizes a closed substack $Z \inj Y$. Then $c$ can be restricted to a correspondence $(Z \leftarrow C_Z \rightarrow Z)$ by Construction \ref{const: restrict correspondence to closed}. We consider the correspondence $\wt{c}$
\begin{equation}\label{eq: deformation normal cone of correspondence}
\begin{tikzcd}
D_Z(Y) &  D_{C_Z}(C) \ar[r, "\wt{c}_1"] \ar[l, "\wt{c}_0"'] & D_Z(Y)
\end{tikzcd}
\end{equation}
over $\A^1_F$, defined by deformation to the normal cone with respect to the vertical closed embeddings in \eqref{eq: sub-correspondence}.
Its fibers away from the origin are isomorphic to $c$ and over $s : \Spec(\F) \inj \A^1_\F$ it degenerates to the correspondence of normal cones:
\begin{equation}\label{eq: normal cone correspondence}
\begin{tikzcd}
N_Z(Y) &  N_{C_Z}(C) \ar[r, "N(c_1)"] \ar[l, "N(c_0)"'] & N_Z(Y)
\end{tikzcd}
\end{equation}

Given $\cK_0,\cK_1 \in \Dmot{Y}$, consider the specialization map on cohomological correspondences
\begin{equation}
\spc_{C, C_Z} \co \Corr_C(\cK_0, \cK_1) \rightarrow \Corr_{N_{C_Z}(C)}(\spc_{Y,Z}(\cK_0),\spc_{Y,Z}(\cK_1))
\end{equation}
defined as the composite of the pullback $\pr^* \co \Corr_C(\cK_0, \cK_1) \to \Corr_{C\times\G_m}(\pr^*\cK_0, \pr^*\cK_1)$ and the map $\Psi^{\mrm{t}}_{D_{C_Z}(C)}$ \eqref{eq:Psi Corr}.

\subsubsection{Specialization vs. trace}

Applied to the deformation to the normal cone, Proposition~\ref{prop: specialization compatible with trace} yields the following compatibility between the specialization maps of \S\ref{sssec: sp norm cyc} and \S\ref{sssec: sp norm corr}.

\begin{cor}\label{cor: sp norm compatible with trace}
Let $c = (Y \xleftarrow{c_0} C \xrightarrow{c_1} Y)$ be a correspondence of Artin stacks locally of finite type over a field $\F$ and let $Z \inj Y$ be a closed substack stabilized by $c$. Then for every $\cK \in \Dmotg{Y}$, the following diagram commutes:
\begin{equation}
\begin{tikzcd}
\Corr_{C}(\cK, \cK \tw{-i}) \ar[r, "\spc_{C,C_Z}"] \ar[d, "\Tr_{C}"]  &  \Corr_{N_{C_Z}(C)}(\spc_{C_Z,C} \cK  , \spc_{C_Z,C} \cK \tw{-i}) \ar[d, "\Tr_{N_{C_Z}(C)}"] \\
\CH_{i}(\Fix(C)) \ar[r, "\sp_{C,C_Z}"] & \CH_{i}(\Fix(N_{C_Z}(C))).
\end{tikzcd}
\end{equation}
\end{cor}
\begin{proof}
Since $\cK$ is USLA over $\Spec(\F)$ (Example~\ref{ex: USLA over point}), $\pr^*\cK$ is USLA over $\G_m$ where $\pr : Y \times \G_m \to Y$ is the projection (Lemma~\ref{lem: USLA smooth local}).
Hence we may apply Proposition~\ref{prop: specialization compatible with trace}.
\end{proof}

\subsection{Motivic local terms}

We will now prove the following result, which appears in the case of $\ell$-adic sheaves on schemes in \cite[Theorem 2.1.3]{Var07} and in the case of motivic sheaves on schemes in \cite[Theorem 5.2.14]{Jin23}.

Given a correspondence $(Y \xleftarrow{c_0} C \xrightarrow{c_1} Y)$ of Artin stacks locally of finite type over a field $\F$ and a closed substack $\iota \co Z \inj Y$ stabilized by $c$, we consider again the base-changed correspondence $(Z \xleftarrow{c_{0Z}} C_Z \xrightarrow{c_{1Z}} Z)$ as in \eqref{eq: sub-correspondence}.
Since the right square in \eqref{eq: sub-correspondence} is Cartesian, $\iota$ induces a right topologically pullable map of correspondences in the sense of Remark~\ref{rem:right top pullable}, and there is a pullback operation $\iota^*$ on cohomological correspondences (see \S \ref{ssec: pullback functoriality for CC}).

\begin{thm}\label{thm: contracting fixed terms} Let $c = (c_0, c_1) \co C \rightarrow Y \times Y$ be a correspondence of locally finite type Artin stacks over $\F$. Let $\iota \co Z \inj Y$ be a closed substack such that $C$ is contracting near $Z$. Let $\cK \in \Dmotg{Y}$. Then $\Fix(C_Z) \rightarrow \Fix(C)$ is open-closed, and for any $\cc \in \Corr_C(\cK, \cK\tw{-i})$, we have
\[
\Tr_C(\cc)|_{\Fix(C_Z)} = \Tr_{C_Z}(\iota^* \cc) \in \CH_{d}(\Fix(C_Z)).
\]
\end{thm}

\begin{lemma}\label{lem: contracting deformation zero-section}
The correspondence $c$ is contracting near $Z$ if and only if it stabilizes $Z$ and the set-theoretic image of $\wt{c}_{0s} = N(c_0)$ is contained in the zero-section $Z \inj N_Z(Y)$. 
\end{lemma}

\begin{proof}
Follows from \lemref{lem: zero-section}.
See also \cite[Lemma 6.4.2(1)]{Jin23}.
\end{proof}

\begin{lemma}\label{lem: Cartesian cube}
Continuing to assume that $c$ stabilizes $Z$, the commutative cube 
\[
\begin{tikzcd}
\Fix(C_Z) \ar[dr] \ar[dd]  \ar[rr] & & C_Z \ar[dd] \ar[dr, hook]  \\
& \Fix(C) \ar[rr
] \ar[dd]  & & C \ar[dd, "c"] \\
Z \ar[dr, hook] \ar[rr] & & Z \times Z \ar[dr, hook] \\
& Y  \ar[rr, "\Delta"] &  & Y \times Y 
\end{tikzcd}
\]
has all squares Cartesian. 
\end{lemma}

\begin{proof}
Since $Z \inj Y$ is a closed embedding, we have $Z \cong Z \times_Y Z \cong (Z \times Z) \times_{Y \times Y} Y$, so the bottom square is Cartesian. The back face is Cartesian by definition. Hence the diagonal square with vertices $\Fix(C_Z), C_Z, Y, Y \times Y$ is Cartesian. The front face is Cartesian by definition, hence the top face is also Cartesian. 

The right face is Cartesian by the assumption that $c$ stabilizes $Z$. Hence the diagonal square with vertices $\Fix(C_Z), C, Z, Y \times Y$ is Cartesian. As the front face is Cartesian, the left face is also Cartesian. We have now checked that all squares are Cartesian.  
\end{proof}

\begin{prop}\label{prop: contracting fixed locus}
If $c$ is contracting near $Z$, then the map $\Fix(C_Z) \rightarrow \Fix(C)$ is open-closed on reduced substacks. 
\end{prop}

\begin{proof}
In the schematic case this is \cite[Theorem 2.1.3(a)]{Var07}, and the argument is the same in essence. By passing to an open subset of $C$, we may assume that $\Fix(C)$ is connected and noetherian and $\Fix(C_Z)$ is non-empty.
Since $c$ is contracting near $Z$, the map $\wt{c}_{0s} = N(c_0)$ has set-theoretic image contained in the zero-section $Z \inj N_Z(Y)$ by Lemma \ref{lem: contracting deformation zero-section}. Hence the same holds for the composite map
\[
N_{\Fix(C_Z)} (\Fix(C)) \to N_{C_Z}(C) \xrightarrow{N(c_0)} N_Z(Y).
\]
Under the identification $\Fix(C_Z) \cong c'^{-1}(Z)$ where $c' : \Fix(C) \to Y$ (Lemma~\ref{lem: Cartesian cube}), it follows by Lemma~\ref{lem: zero-section} that $(c')^* \cI_Z^n \subset (\cI_{{c'}^{-1}(Z)})^{n+1}$ for some $n$.
Since $(c')^* \cI_Z^n = (\cI_{{c'}^{-1}(Z)})^n$, we deduce that $(\cI_{{c'}^{-1}(Z)})^{n} = (\cI_{{c'}^{-1}(Z)})^{n+1}$ for some $n$. The noetherianity and connectedness then imply that $(\cI_{{c'}^{-1}(Z)})^n = 0$, so $\Fix(C)_{\red} \cong {c'}^{-1}(Z)_{\red} \cong \Fix(C_Z)_{\red}$.
\end{proof}

\begin{prop}\label{prop: constant fixed locus}
If $c$ is contracting near $Z$, then $\Fix(D_{C_Z}(C))_{\red}$ is isomorphic to the constant family $\Fix(C_Z)_{\red}$ over $\A^1$.
\end{prop}

\begin{proof}
In the schematic case this is \cite[Theorem 2.1.3(b)]{Var07}, and the argument is the same in essence. By passing to an open subset of $C$, we may assume that $\Fix(C)$ is connected and noetherian and $\Fix(C_Z)$ is non-empty. Then thanks to Proposition \ref{prop: contracting fixed locus} we have $\Fix(C)_{\red} \cong \Fix(C_Z)_{\red}$. As $\wt{(c_Z)}_{\red}$ is a constant family of correspondences over $\A^1$, we have $\Fix(C)_{\red} \times \A^1 \cong \Fix(D_{C_Z}(C_Z))_{\red}$, which by Lemma \ref{lem: Cartesian cube} admits a closed embedding into $\Fix(D_{C_Z}(C))_{\red}$, which is moreover an isomorphism over $\eta \subset \A^1$. It then suffices to show that the special fiber $\Fix(N_{C_Z}(C))$ is set-theoretically supported within $\Fix(C)$. 

By Lemma \ref{lem: contracting deformation zero-section}, $\wt{c}_{0s} [\Fix(N_{C_Z}(C))]$ has set-theoretic image contained in the zero-section $Z \inj N_Z(Y)$. Hence the same is true for $\wt{c}_{1s} [\Fix(N_{C_Z}(C))]$, so $\Fix(N_{C_Z}(C))$ has set-theoretic image contained in the zero-section $C_Z \inj N_{C_Z}(C)$ under both $\wt{c}_{0s}$ and $\wt{c}_{1s}$. But the restriction of $\wt{c}_{s}$ to $c_1^{-1}(Z)$ equals the correspondence $c_Z$, so we deduce that $\Fix(N_{C_Z}(C))$ is set-theoretically contained in $\Fix(C_Z) \subset \Fix(C)$, as desired. 
\end{proof}

Recall the specialization to the normal cone map $\spc_{C,C_Z}$ on cohomological correspondences defined in \S\ref{sssec: sp norm corr}.

\begin{prop}\label{prop: LT vanish sp}
Suppose $c$ is contracting near $Z$.
Let $\cK \in \Dmotg{Y}$ and let $\cc \in \Corr_C(\cK, \cK\tw{-i})$ be a cohomological correspondence.
If $\spc_{C,C_Z}(\cc) = 0$, then we have
\[
\Tr_C(\cc)|_{\Fix(C_Z)} = 0 \in \CH_{i}(\Fix(C_Z)).
\]
\end{prop}

\begin{proof}
By Corollary~\ref{cor: sp norm compatible with trace} we have 
\[
\fsp_{C,C_Z}(\Tr_C(\cc)) = \Tr_{N_{C_Z}(C)}(\spc_{C, C_Z}(\cc)) = 0.
\]
Since the map $\fsp_{C,C_Z}$ is an isomorphism in this case by Proposition~\ref{prop: constant fixed locus}, we conclude that $\Tr_C(\cc) = 0$.
\end{proof}

\begin{cor}\label{cor: LT vanish boundary}
Suppose $c$ is contracting near $Z$.
Let $\cK \in \Dmotg{Y}$ and let $\cc \in \Corr_C(\cK, \cK\tw{-i})$ be a cohomological correspondence.
If $\iota^* \cK \cong 0 \in \Dmot{Z}$, then we have
\[
\Tr_C(\cc)|_{C_Z} = 0 \in \CH_{i}(\Fix(C_Z)).
\]
\end{cor}
\begin{proof}
By Proposition~\ref{prop: Verdier} we have $\spc_{Y,Z}(\cK)|_Z \cong \cK|_Z \cong 0$.
By Lemma~\ref{lem: contracting deformation zero-section}, the map $\wt{c}_{0s}$ has set-theoretic image contained in the zero-section $Z \subset N_Z(Y)$, hence also $\wt{c}_{0s}^* \spc_{Y,Z}(\cK) \cong 0$.
In particular, $\spc_{C, C_Z}(\cc) = 0$.
We conclude by Proposition~\ref{prop: LT vanish sp}.
\end{proof}

\begin{proof}[Proof of Theorem~\ref{thm: contracting fixed terms}]
Since the diagram \eqref{eq: sub-correspondence} is left pushable (as the vertical maps are closed embeddings, hence proper, cf. Example \ref{ex: pushable/pullable examples}), $\iota_! \iota^* \cc \in \Corr_C(\iota_! \iota^* \cK_0, \iota_!\iota^* \cK_1 \tw{-i})$ is defined.
By Proposition~\ref{prop: trace commutes with proper push} we have $
\Tr_C(\iota_* \iota^* \cc) = \Fix(\iota)_* (\Tr_{C_Z}(\iota^* \cc))$.

Let $j \co V \inj Y$ be the open embedding complementary to $\iota$. Since $Z$ is $c$-invariant, we have $c_0^{-1}(V) \subset c_1^{-1}(V)$. Let $C_V \coloneqq c_0^{-1}(V)$. Then the map of correspondences 
\begin{equation}\label{eq: open sub-correspondence}
\begin{tikzcd}
V \ar[d, hook] & C_V \ar[l, "c_{0V}"'] \ar[r, "c_{1V}"] \ar[d, hook] & V \ar[d, hook] \\
Y & \ar[l, "c_0"']  C \ar[r, "c_1"] & Y
\end{tikzcd}
\end{equation}
is right pullable (since the vertical maps are open embeddings, hence smooth, cf. Example \ref{ex: pushable/pullable examples}), so $j^* \cc \in \Corr_{C_V}(j^* \cK, j^* \cK\tw{-i})$ is defined. Since the left square of \eqref{eq: open sub-correspondence} is Cartesian, it is also left pushable, so $j_! j^* \cc \in \Corr_{C}(j_!j^* \cK, j_!j^* \cK\tw{-i})$ is defined. 
Since $\iota^* j_! j^* \cK \cong 0$ in $\Dmot{Z}$, we have $\Tr_C(j_! j^* \cc)|_{\Fix(C_Z)} = 0$ by Corollary~\ref{cor: LT vanish boundary}.

We have
\begin{equation}\label{eq: trace additive}
\Tr_C(\cc) = \Tr_C(\iota_* \iota^* \cc) + \Tr_C(j_! j^* \cc) = \Fix(\iota)_* (\Tr_{C_Z}(\iota^* \cc))
\end{equation}
by additivity of the trace (see \cite[Lemma 5.2.7]{Jin23}, whose proof works verbatim for stacks). Then restricting to $\Fix(C_Z)$, and we conclude using that $\Tr_C(j_! j^* \cc)|_{\Fix(C_Z)} = 0$ as found in the preceding paragraph. 
\end{proof}

\section{Derived homogeneous Fourier transform}\label{sec: derived homogeneous FT}

Let $E$ be a vector bundle over a scheme $S$ and $\wh{E}$ the dual vector bundle\footnote{We also use the notation $E^*$ for the dual vector bundle to $E$. However, this leads to an inconvenient notation for the dual of the map, so we avoid it when discussing the Fourier transform. We note that the notation $E^\vee$ is also reserved for the \emph{Serre dual} of $E$.}. Laumon \cite{Laumon} introduced a geometric Fourier transform
\begin{equation*}
  \cD^{\Gm}(E; \overline{\bQ}_\ell)
  \to \cD^{\Gm}(\wh{E}; \overline{\bQ}_\ell)
\end{equation*}
on bounded constructible derived categories of \emph{homogeneous} (i.e., $\Gm$-equivariant) $\ell$-adic sheaves.
It can be regarded as a uniform version of the $\ell$-adic Fourier transform of Deligne-Laumon \cite{Lau87} (for base fields of characteristic $p>0$) and the $D$-module 
Fourier transform of Brylinski--Malgrange--Verdier (for base fields of characteristic $0$).

We will describe an extension of this construction from vector bundles to \emph{derived vector bundles}, i.e., total spaces of perfect complexes (cf. \cite[\S 6.1]{FYZ3}). As explained in \S \ref{ssec: notate dvb}, the total space of a perfect complex $\cE$ exists naturally as a derived Artin stack $\Tot(\cE)$, which can exhibit both derived and stacky behavior depending on the amplitude of the complex.
Thus in this context, the Fourier transform manifests a duality between derived and stacky phenomena.

For this section (only), the scope of the sheaf theory will be expanded significantly. In fact, we will show that the homogeneous Fourier transform is well-behaved in the context of any reasonable six functor formalism. More precisely, we will work in the generality of any \emph{topological weave} in the sense of \cite{Weaves}. This includes classical six functor formalisms such as the derived \inftyCat of sheaves of abelian groups (over $\bC$) or the derived \inftyCat of $\ell$-adic sheaves (over a base on which $\ell$ is invertible), but also various motivic \inftyCats: Voevodsky motives, MGL-modules, or motivic spectra (not even orientability is required). We include this added generality because the arguments are uniform, and also for the sake of forthcoming applications in different sheaf-theoretic contexts. 

Here is an outline of this section. We begin in \S\ref{sec:main} by defining the homogeneous Fourier transform for derived vector bundles and stating our results about it. \S\ref{sec:equivariant} recalls the Contraction Principle and its consequences. \S \ref{sec:j} carries out some technical computations. In \S\ref{sec:j} we redo some straightforward computations from \cite{Laumon} that don't involve derived vector bundles, adapting the arguments from \emph{op. cit.} to the generality we work in here. In \S\ref{sec:easy} we prove the easier properties that are independent of involutivity. The proof of involutivity is achieved in \S\ref{sec:cosuppinvol} -- \S\ref{sec:cosupp}. Finally the remaining properties of the Fourier transform are derived from involutivity in \S \ref{sec:funct2}.

\subsection{Conventions and notation}

  \subsubsection{Sheaves}

    Throughout the section we fix a \emph{topological weave} $\bD$, which is an axiomatization of a sheaf theory with the six functor formalism introduced in \cite{Weaves}.
    Roughly speaking this is a six-functor formalism with the following properties:
    \begin{enumcompress}
      \item\label{item:weave/loc}
      \emph{Localization:}
      The \inftyCat $\bD(\initial)$ is zero.
      For any derived Artin stack $Y$ and any closed-open decomposition $i : Z \to Y$, $j : Y \setminus Z \to Y$, there is a canonical exact triangle of functors
      \begin{equation}\label{eq:loc}
        j_! j^* \to \id \to i_*i^*.
      \end{equation}

      \item\label{item:weave/htp}
      \emph{Homotopy invariance:}
      For any derived Artin stack $Y$ and vector bundle $\pi : E \to Y$, the unit morphism $\id \to \pi_*\pi^*$ is invertible.
    \end{enumcompress}
    Note that localization implies \emph{derived invariance}: for a derived Artin stack $Y$, the inclusion of the classical truncation $Y_\cl \hook Y$ induces an equivalence $\bD(Y) \cong \bD(Y_\cl)$ which commutes with each of the six functors.
    We also have \emph{Poincaré duality}, which for a topological weave gives a canonical isomorphism
    \begin{equation}\label{eq:Poin}
      f^!(-) \cong f^*(-)\vb{T_f}
    \end{equation}
    for $f$ any smooth morphism with relative tangent bundle $T_f$. (Here $\vb{T_f}$ is the \emph{Thom twist} by $T_f$.) 
    If $\bD$ admits an orientation (as in the first five examples below), we may identify $\vb{T_f} \cong \vb{d} \coloneqq (d)[2d]$ where $d$ is the relative dimension of $f$.

    On (derived) schemes, examples of weaves are as follows:
    \begin{enumcompress}
      \item\emph{Betti sheaves (over $\bC$):}
      The assignment $Y \mapsto \bD(Y)$ sending $Y$ to the derived \inftyCat $\on{D}(Y(\bC); \bZ)$ of sheaves of abelian groups on the topological space $Y(\bC)$.

      \item\emph{Étale sheaves (finite coefficients, over $k$ with $n \in k^\times$):}
      The assignment $Y \mapsto \bD(Y)$ sending $Y$ to the derived \inftyCat $\on{D}_{\et}(Y; \bZ/n\bZ)$ of sheaves of $\bZ/n\bZ$-modules on the small étale site of $Y$.

      \item\emph{Étale sheaves ($\ell$-adic coefficients, over $k$ with $\ell \in k^\times$):}
      The assignment $Y \mapsto \bD(Y)$ sending $Y$ to the $\ell$-adic derived \inftyCat $\on{D}_{\et}(X; \bZ_\ell)$ of sheaves on the small étale site of $Y$, i.e., the limit of $\on{D}_{\et}(Y; \bZ/\ell^n\bZ)$ over $n>0$.

      \item\label{item:weave/motives}\emph{Motives:}
      For any commutative ring $\Lambda$, the assignment $Y \mapsto \bD(Y)$ sending $Y$ to the \inftyCat $\cD_{\mrm{mot}}(Y;\Lambda) \coloneqq \on{D_{H\Lambda}}(Y)$ of modules over the $\Lambda$-linear motivic Eilenberg--MacLane spectrum $H\Lambda_Y$ (defined as in \cite{zbMATH07015021}).

      \item\emph{Cobordism motives:}
      The assignment $Y\mapsto \bD(Y)$ sending $Y$ to the \inftyCat $\MGLmod(Y)$ of modules over Voevodsky's algebraic cobordism spectrum $\MGL_Y$.

      \item\emph{Motivic spectra:}
      The assignment $X\mapsto \bD(Y)$ sending $Y$ to the \inftyCat $\mrm{SH}(Y)$ of motivic spectra over $Y$.
    \end{enumcompress}

The categories $\bD(Y)$ are symmetric monoidal, and we denote by $\un_Y$ (or just $\un$ when context is clear) the monoidal units.    
    
    When the weave satisfies étale descent, it can be extended to (derived) Artin stacks following the method of \cite{LiuZheng} (see also \cite[App.~A]{KhanI}).
    This is the case for the first three examples, as well as for the weave of interest in the rest of this paper, $Y \mapsto \Dmot{Y}$.

    In general, $\bD$ only satisfies Nisnevich descent.
    In that case the six-functor formalism extends to $\Nis$-Artin stacks (see \cite[\S 4]{Weaves}, \cite[\S 1]{equilisse}).
    For $\tau\in\{\Nis,\et\}$ we define $(\tau,n)$-Artin and $\tau$-Artin stacks as in \cite[0.2.2]{equilisse}:
    \begin{defnlist}
      \item
      A stack is $(\et,0)$-Artin, resp. $(\Nis,0)$-Artin, if it is an algebraic space (resp. a decent algebraic space).
      Here an algebraic space is \emph{decent} if it is Zariski-locally quasi-separated, or equivalently Nisnevich-locally a scheme.
      
      \item
      For $n>0$, $X$ is \emph{$(\tau,n)$-Artin} if it has $(\tau,n-1)$-representable diagonal and admits a smooth morphism $U \twoheadrightarrow X$ with $\tau$-local sections from some scheme $U$.
      A stack is \emph{$\tau$-Artin} if it is $(\tau,n)$-Artin for some $n$.
    \end{defnlist}
    For $\tau=\et$, these are the usual notions of $n$-Artin stacks and Artin stacks, while e.g. $(\Nis,1)$-Artin stacks are the same as quasi-separated $1$-Artin stacks with separated diagonal by \cite[\S 6.7]{LMB}.

    If our chosen weave $\bD$ does not satisfy étale descent, then we adopt the convention that \emph{Artin} means ``$\Nis$-Artin''.

  \subsubsection{\texorpdfstring{$\bG_m$}{Gm}-Equivariant sheaves}

    If $X$ is a derived Artin stack with $\Gm$-action, we define the \inftyCat of $\Gm$-equivariant sheaves on $X$ as the \inftyCat of sheaves on the quotient stack:
    \begin{equation*}
      \bD^{\Gm}(X) \coloneqq \bD([X/\Gm]).
    \end{equation*}
    There is a forgetful functor
    \begin{equation*}
      \bD^{\Gm}(X) \to \bD(X)
    \end{equation*}
    defined by $*$-pullback along the (smooth) quotient morphism $X \twoheadrightarrow [X/\Gm]$.

    For any $\Gm$-equivariant morphism $f : X \to Y$ the four operations $f^*$, $f_*$, $f_!$, $f^!$ are defined on the $\Gm$-equivariant category using the induced morphism $[X/\Gm] \to [Y/\Gm]$.
    Since each operation commutes with smooth $*$-inverse image, they each commute with forgetting equivariance.

    Note that, when the weave $\bD$ does not satisfy étale descent, we are implicitly using the fact that the group scheme $\Gm$ is special in the sense of Serre (see \cite[\S 4.1]{SCC_1958__3__A1_0}), so that $X \twoheadrightarrow [X/\Gm]$ admits Zariski-local sections and hence the quotients are $\Nis$-Artin whenever $X$ is.

\subsubsection{Derived vector bundles} In addition to \S \ref{ssec: notate dvb}, we will use the following notations and properties of derived vector bundles. 

    \begin{enumerate}
      \item We write $E[n] \coloneqq \V(\cE[n])$ for every integer $n$.
      \item We say $E$ is \emph{of amplitude $\ge 0$} (resp. $\le 0$, $[a,b]$) if $\cE$ is of tor-amplitude\footnote{We remind that we are using the \emph{cohomological} indexing here.} $\ge 0$ (resp. $\le 0$, $[a,b]$). Note that since $\cE$ is perfect, $E$ is of some finite amplitude.
    \end{enumerate}

    \begin{notat}
      Given a derived vector bundle $E$ over a derived Artin stack $S$, we denote by $\pi_E : E \to S$ the projection and $0_E : S \to E$ the zero section.
    \end{notat}

    \begin{notat}\label{notat: quo}
      For every $E \in \DVect(S)$, it will be convenient to denote the quotient by the $\bG_m$-scaling action by:
      \begin{equation}
        \quo{E} \coloneqq [E/\bG_m].
      \end{equation}
      Given a morphism of derived vector bundles $\phi : E' \to E$, we also write $\quo{\phi} : \quo{E'} \to \quo{E}$ for the induced morphism.
      We will also use the same notation for $\bG_m$-invariant subsets of $E$, e.g., $\quo{\bG_{m,S}} = S$.
    \end{notat}

    If the weave $\bD$ does not satisfy étale descent, we need the following:

    \begin{prop}
      Let $S$ be a derived stack and $E \in \DVect(S)$ a derived vector bundle over $S$.
      If $E$ is of amplitude $\ge 0$, then the projection $\pi_E \co E \to S$ is affine.
      If $E$ is of amplitude $\ge -n$, where $n>0$, then $\pi_E$ is $(n,\Nis)$-Artin.
    \end{prop}
    \begin{proof}
      The proof is the same as that of $n$-Artinness which is well-known.
      We recall it anyway for the reader's convenience.
      Suppose $E$ is of amplitude $\ge -n$ where $n=0$ (resp. $n>0$).
      It is enough to show that if $S$ is affine, then $E$ is affine (resp. $(n,\Nis)$-Artin).
      In the $n=0$ case we have $E = \V(\cE) \cong \Spec(\Sym(\cE^*))$ (\S \ref{ssec: notate dvb}).
      Thus assume $n>0$.

      Let us assume the statement known for $n-1$ and argue by induction.
      Since $E[-1]$ is of amplitude $\ge -n+1$, it is $(\Nis,n-1)$-Artin by inductive hypothesis.
      This implies that $E$ has $(\Nis,n-1)$-representable diagonal.

      We now construct a $\Nis$-atlas for $E$.
      Since $S$ is affine we may choose a presentation of $\cE$ as a cochain complex of vector bundles (as in \S \ref{ssec: notate perf}) and let $\cE^{\le 0}$ and $\cE^{>0}$ denote the brutal truncations so that we have an exact triangle $\cE^{>0} \to \cE \to \cE^{\le 0}$.
      Taking total spaces we have a derived Cartesian square
      \begin{equation*}
        \begin{tikzcd}
          E^{\le 0} \ar{r}{i}\ar{d}
          & E \ar{d}
          \\
          S \ar{r}{0}
          & E^{>0}.
        \end{tikzcd}
      \end{equation*}
      We claim that $0 \co S \to E^{>0}$ is smooth and admits Nisnevich-local sections.
      Since $E^{\le 0}$ is affine by the $n=0$ case, $i \co E^{\le 0} \to E$ will then be a $\Nis$-atlas for $E$.

      Let $T$ be an affine derived scheme over $S$ and $f : T \to E^{>0}$ an $S$-morphism.
      Since $E^{>0}$ is of amplitude $<0$ and $T$ is affine, we have $\pi_0 \Maps_S(T, E^{>0}) \cong \{0\}$ by the definition of $\Tot(-)$ in \S \ref{ssec: notate dvb}.
      In other words, $f$ factors through the zero section $0 : S \to E^{>0}$.
      The base change of $0 \co S \to E^{>0}$ along $f$ is therefore identified with the projection $E^{>0}[1]\fibprod_S T \to T$.
      The latter clearly admits a section over $T$ (e.g. the zero section), and is smooth because $E^{>0}[1]$ is of amplitude $[-n+1,0]$.
      The claim follows.
    \end{proof}

    \begin{rem}
      In fact one can easily show: $E$ is of amplitude $\ge 0$ if and only if $\pi_E$ is affine, if and only if $\pi_E$ is representable, if and only if $0_E$ is a closed immersion.
      Similarly, $E$ is of amplitude $\le 0$ if and only if $\pi_E$ is smooth, if and only if $\pi_{\wh{E}}$ is affine.
    \end{rem}

\subsection{Definition and properties of the derived homogeneous Fourier transform}
\label{sec:main}

  \subsubsection{Homogeneous Fourier kernel} The quotient stack $\quo{E}$ classifies pairs $(L, \phi \co L \rightarrow E)$. Consider the homogeneous evaluation morphism
    \begin{equation}
      \ev_E : \quo{\wh{E}} \fibprod_S \quo{E} \to \quo{\A^1_S}
    \end{equation}
    sending a pair
    \begin{equation*}
      \big((L, \phi : L \to \wh{E}), (L', \psi : L' \to E)\big)
    \end{equation*}
    to
    \begin{equation*}
      \big( L \otimes L', L \otimes L' \xrightarrow{\phi \otimes \psi} \wh{E} \otimes_S E \xrightarrow{\ev} \A^1_S \big).
    \end{equation*}
    We define
    \begin{equation}
      \sP_E \coloneqq \ev_E^* ( \quo{j_*}(\un)) \in \bD(\quo{\wh{E}} \fibprod_S \quo{E})
    \end{equation}
    where $j : \bG_{m,S} \to \A^1_S$ is the inclusion and $\quo{j} : S = \quo{\bG_{m,S}} \hook \quo{\A^1_S}$ is the induced map of quotient stacks.

  \subsubsection{Homogeneous Fourier transform}\label{sssec:FT}
  
    Let $\pr_1$ and $\pr_2$ denote the respective projections
    \begin{equation*}
      \begin{tikzcd}
        & {\quo{\wh{E}} \fibprod_S \quo{E}} \ar{rd}{\pr_2}\ar[swap]{ld}{\pr_1} &
        \\
        {\quo{\wh{E}}} & & {\quo{E}}
      \end{tikzcd}
    \end{equation*}

    The \emph{homogeneous Fourier transform} on $E$ is the functor
    \begin{equation*}
      \FT_E : \bD^{\Gm}(E) \to \bD^{\Gm}(\wh{E})
    \end{equation*}
    defined by
    \begin{equation}\label{eq: FT definition}
      \cK \mapsto \pr_{1!}(\pr_2^*(\cK) \otimes \sP_E)[-1].
    \end{equation}
At times it will be convenient to think of $\FT_E$ equivalently as a functor $\bD(\quo{E}) \rightarrow \bD(\quo{\wh{E}})$. 

\subsubsection{Geometricity}\label{sssec: weave gm}
We define the full subcategory $\bD_{\mrm{gm}}(-)$ as in \S \ref{ssec: geometric motives}: 
\begin{enumerate}
\item For $S$ a derived scheme,  $\bD_{\mrm{gm}}(S)$ is the thick subcategory of $\bD(S)$ generated by all Thom twists of all $f_{\sh}(\un_T)$ as $f \co T \rightarrow S$ ranges over smooth schemes over $S$, where $f_{\sh}$ is the left adjoint of $f^*$. 
\item For $S$ a derived Artin stack, $\bD_{\mrm{gm}}$ is the full subcategory of objects that become geometric on any (equivalently all) atlas from a derived scheme.
\end{enumerate}

The arguments in the proof of \cite[Theorem~4.2.29]{CD19} show that:

\begin{thm}\label{thm:dgm}
  Let $\bD$ be a topological weave on derived schemes locally of finite type over a quasi-excellent scheme.
  If $\bD$ is $\bQ$-linear and satisfies étale descent and topological invariance (e.g. $\bD=\cD_{\mrm{mot}}(-; \bQ)$), then the six operations restrict to $\bD_{\mrm{gm}}$.
\end{thm}

Under the assumptions of \thmref{thm:dgm}, consider the lisse extension of $\bD$ to derived Artin stacks (locally of finite type over a quasi-excellent scheme).
It is then immediate that the six operations preserve geometricity, \emph{with the possible exception of} the operations $f_*$ and $f_!$ for $f$ a non-representable morphism, just as in Remark \ref{rem: preservation of constructibility}.
We will see in Corollary \ref{cor: FT preserve local constructibility} that $\FT$ also preserves geometricity in this case.

\subsubsection{Zero bundle}
\label{ssec:main/zero}

  By abuse of notation, we denote the zero bundle $\V_S(0) = S$ by $0_S \in \DVect(S)$.

  \begin{prop}\label{prop:zero}
    Let $o : B\bG_{m,S} \to B\bG_{m,S}$ be the involution $L \mapsto \wh{L}$.
    There are natural isomorphisms $\FT_{0_S} \cong o^* \cong o_* \cong o_! \cong o^!$.
  \end{prop}

  The proof is in \S \ref{ssec:easy/zero}.

\subsubsection{Involutivity}
\label{ssec:main/invol}

  In the most general case, our statement of involutivity is up to twisting with a canonical $\otimes$-invertible object. 

  \begin{lem}\label{lem:L}
    Let $E \in \DVect(S)$.
    The object
    \begin{equation}\label{eq:L}
      \sL^{E} \coloneqq
      \pi_{E!} \FT_{\wh{E}}(\un_{\wh{E}})
      \cong 0_{E}^! \FT_{\wh{E}}(\un_{\wh{E}})
      \in \bD^{\Gm}(S)
    \end{equation}
    is $\otimes$-invertible.
  \end{lem}

  The proof is in \S \ref{ssec:cosupp/L}.

  \begin{thm}\label{thm:invol}
    For every $E \in \DVect(S)$, there is a canonical isomorphism
    \begin{equation}\label{eq:hygrothermal}
      (-) \otimes \pi_{E}^*(\sL^{E}) \to \FT_{\wh{E}} \circ \FT_{E}(-)
    \end{equation}
    of functors $\bD^{\Gm}(E) \to \bD^{\Gm}(E)$.
  \end{thm}

  The proof is assembled at the beginning of \S\ref{sec:cosupp}.

  \begin{notat}
    Let $E \in \DVect(S)$.
    Tensoring with $\sL^E$ defines an auto-equivalence
    \begin{equation*}
      (-)\twbrace{E} \coloneqq (-) \otimes \sL^E
      \quad \text{of}~\bD^{\Gm}(S).
    \end{equation*}
    We denote its inverse by $(-)\twbrace{-E}$.
    We also write $(-)\twbrace{E} \coloneqq (-) \otimes f^*(\sL^E)$ as an endofunctor of $\bD^{\Gm}(X)$ where $f : X \to S$ is any derived Artin stack over $S$ with $\Gm$-action.
    In this notation \eqref{eq:hygrothermal} reads:
    \begin{equation}
      (-)\twbrace{E} \to \FT_{\wh{E}} \FT_{E}(-).
    \end{equation}
  \end{notat}

  \begin{cor}\label{cor:inverse}
    For every $E \in \DVect(S)$, the functor $\FT_{\wh{E}}(\bullet)\twbrace{-E}$ determines a canonical inverse to $\FT_E$.
  \end{cor}

\subsubsection{Base change}
\label{ssec:main/funbase}

  \begin{prop}\label{prop:funbase}
    For every morphism $f : S' \to S$, denote by $f_E : E' \to E$ and $f_{\wh{E}} : \wh{E'} \to \wh{E}$ its base changes.
    Then there are canonical isomorphisms
    \begin{align}
      f_{\wh{E}}^* \circ \FT_{E} &\cong {\FT_{E'}} \circ f_E^*,\tag{BC$^*$}\label{eq:BC^*}\\
      f_{\wh{E}*} \circ {\FT_{E'}} &\cong {\FT_{E}} \circ f_{E*},\tag{BC$_*$}\label{eq:BC_*}\\
      f_{\wh{E}!} \circ \FT_{E'} &\cong {\FT_{E}} \circ f_{E!}\tag{BC$_!$}\label{eq:BC_!}\\
      f_{\wh{E}}^! \circ \FT_{E} &\cong {\FT_{E'}} \circ f_E^!\tag{BC$^!$}\label{eq:BC^!}.
    \end{align}
  \end{prop}

  The proof is in \S \ref{ssec:basechange1} and \S \ref{sec:funct2}.

  \begin{lem}\label{lem:twistbase}
    Let $E \in \DVect(S)$.
    For every morphism $f : S' \to S$, we have a canonical isomorphism
    \begin{equation}
      f^* \sL^E \cong \sL^{E'}.
    \end{equation}
    Moreover, if $f_E : E' \to E$ denotes the base change, there are canonical isomorphisms
    \begin{align}
      f^*_E((-)\twbrace{E}) &\cong f_E^*(-)\twbrace{E'},\\
      f^!_E((-)\twbrace{E}) &\cong f_E^!(-)\twbrace{E'},\\
      f_{E*}((-)\twbrace{E'}) &\cong f_{E*}(-)\twbrace{E},\\
      f_{E!}((-)\twbrace{E'}) &\cong f_{E!}(-)\twbrace{E},
    \end{align}
    and similarly for $f_{\wh{E}} : \wh{E'} \to \wh{E}$.
  \end{lem}

  The proof is in \S\ref{ssec:basechangetwist1} and \S\ref{ssec:basechangetwist2}.

\subsubsection{Functoriality}
\label{ssec:main/fun}

  \begin{prop}\label{prop:fun}
    For every morphism of derived vector bundles $\phi : E' \to E$ over $S$, there are canonical isomorphisms
    \begin{align}
      \Ex^{*,\FT} &: \wh{\phi}^* \circ {\FT_{E'}} \to {\FT_{E}} \circ \phi_!\tag{Fun$^*$}\label{eq:funA}\\
      \Ex^{\FT,!} &: {\FT_{E'}} \circ \phi^! \to \wh{\phi}_* \circ {\FT_{E}}\tag{Fun$_*$}\label{eq:funA'}\\
      \Ex^{!,\FT} &: \wh{\phi}^{!} \circ {\FT_{E'}\twbrace{-E'}} \to \FT_{E}\twbrace{-E} \circ \phi_*\tag{Fun$^!$}\label{eq:funB}\\
      \Ex^{\FT,*} &: {\FT_{E'}\twbrace{-E'}} \circ \phi^* \to \wh{\phi}_! \circ {\FT_{E}\twbrace{-E}}. \tag{Fun$_!$}\label{eq:funB'}
    \end{align}
  \end{prop}

  The proof is in \S \ref{ssec:easy/fun} and \S \ref{sec:funct2}.

  \begin{exam}
    Let $E \in \DVect(S)$.
    Since the projection $\pi_{E} : E \to S$ and zero section $0_{E} : S \to E$ are dual to $0_{\wh{E}} : S \to \wh{E}$ and $\pi_{\wh{E}} : \wh{E} \to S$, respectively, we get canonical isomorphisms
    \begin{align}
      \FT_{0_S} \circ ~\pi_{E!} &\cong 0_{\wh{E}}^* \circ \FT_{E}\\
      \FT_{E} \circ ~0_{E!} &\cong \pi_{\wh{E}}^* \circ \FT_{0_S}.
    \end{align}
  \end{exam}

\subsubsection{Outline of proof: support and cosupport properties}
\label{ssec:main/cosupp}

  We will see that the proof of involutivity (\thmref{thm:invol}) eventually boils down to what we call the ``cosupport property'' for the object $\FT_{\wh{E}}(\un_{\wh{E}}) \in \bD^{\Gm}(E)$.

  When $E$ is of amplitude $\ge 0$, the object $\FT_{\wh{E}}(\un_{\wh{E}})$ is supported on the zero section of $E$ (which is a closed immersion):

  \begin{prop}[Support property]\label{prop:suppinvol}
    Let $E \in \DVect(S)$.
    If $E$ is of amplitude $\ge 0$, then we have:
    \begin{thmlist}
      \item\label{item:Ghent}
      There is a canonical isomorphism
      \begin{equation}
        {0}_E^* \FT_{\wh{E}}(\un)
        \cong \un_S\vb{-\wh{E}}.
      \end{equation}

      \item\label{item:journalize}
      The canonical morphism
      \begin{equation}\label{eq:doctrinally}
        \FT_{\wh{E}}(\un)
        \xrightarrow{\mrm{unit}} {0_{E*}} {0_E^*} \FT_{\wh{E}}(\un)
        \cong {0_{E*}}(\un_S)\vb{-\wh{E}}
      \end{equation}
      is invertible.
    \end{thmlist}
    In particular, there is a canonical isomorphism
    \begin{equation}\label{eq:conntwist}
      \un_S\twbrace{E}
      \cong 0_E^* \FT_{\wh{E}}(\un)
      \cong \un_S \vb{-\wh{E}}.
    \end{equation}
  \end{prop}

  The proof is in \S\ref{sec:supp}. In general, the zero section $0_E$ is not a closed immersion, so that $0_{E!}$ does not agree with $0_{E*}$.
  Nevertheless, the following dual version of \propref{prop:suppinvol} holds for $E$ of arbitrary amplitude:

  \begin{prop}[Cosupport property]\label{prop:main/cosupp}
    For every $E \in \DVect(S)$, the object $\FT_{\wh{E}}(\un_{\wh{E}}) \in \bD^{\Gm}(E)$ lies in the essential image of the fully faithful functor $0_{E!}$.
    More precisely, the canonical morphism
    \begin{equation}\label{eq:urodelan}
      0_{E!} (\un_S\twbrace{E})
      \cong 0_{E!} 0_{E}^!(\FT_{\wh{E}}(\un_{\wh{E}}))
      \xrightarrow{\counit} \FT_{\wh{E}}(\un_{\wh{E}})
    \end{equation}
    is invertible.
  \end{prop}

  The proof is in \subsectionref{ssec:cosupp/proof}.
  Involutivity will then follow from:

  \begin{lem}\label{lem:cosuppinvol}
    Let $E \in \DVect(S)$.
    If $E$ satisfies the cosupport property, then there is a canonical isomorphism
    \begin{equation*}
      (-)\twbrace{E} \to \FT_{\wh{E}} \FT_{E}(-).
    \end{equation*}
  \end{lem}

  The proof is in \S\ref{sssec: proof of involutivity}.

\subsubsection{Identification of the twist}

  We do not know whether the twist $\sL^E \cong \un_S\twbrace{E}$ can be identified with the Thom twist $\un_S\vb{-\wh{E}}$ in general.
  If $E$ admits a global presentation as a cochain complex of vector bundles, one can with some care build such an isomorphism from the vector bundle case.
  To glue together these local isomorphisms (choosing presentations locally on $S$) we would need to show they are compatible up to coherent homotopy.
  Assuming the existence of a suitable t-structure this question is reduced to the heart, where we just need to check a cocycle condition.
  This line of argument leads to a proof of the following statement, by the same argument as in \cite[\S A.3.5]{FYZ3}.\footnote{The gluing is quite subtle. This is due to the fact that if $E$ is a vector bundle, the isomorphisms \eqref{eq:conntwist} for $E$ and $\wh{E}$ only agree up to a sign.}
  
  \begin{prop}\label{prop:gluetwist}
    Suppose that the weave $\bD$ admits an orientation and a t-structure in which the unit is discrete, i.e., $\un_S \in \bD(S)^\heartsuit$ for all derived Artin stacks $S$.
    Then for every $E \in \DVect(S)$ there exists a canonical isomorphism
    \begin{equation}\label{eq:gluetwist}
      \sL^E \cong \un_S\twbrace{E} \cong \un_S\vb{-r}
    \end{equation}
    where $r=\rk(E)$.
  \end{prop}

\begin{remark} 
Proposition \ref{prop:gluetwist} applies to the weave $\bD = \Dmot{-}$, by Lemma \ref{lem:Qplus heart}.
 \end{remark}

\subsection{The contraction principle}
\label{sec:equivariant}

  The following \emph{contraction principle} is well-known in the case of a separated morphism of schemes. In the context of $D$-modules it appears in \cite[Theorem~C.5.3]{DrinfeldGaitsgoryCompact}, and the proof works for an arbitrary topological weave. 

  \begin{prop}[Contraction principle]\label{prop:contract}
    Let $\pr : Y \to S$ be a morphism of derived Artin stacks and $s : S \to Y$ a section.
    Suppose there is an $\A^1$-homotopy $Y \times \A^1 \to Y$ between $\id_Y$ and $s \circ \pr$, so that the diagram
    \[ \begin{tikzcd}[matrix yscale=0.1,matrix xscale=1.3]
      Y \ar{rd}{i_0}\ar[bend left,swap]{rrd}{s\circ \pr}
      &
      &
      \\
      & Y \times \A^1 \ar{r}
      & Y
      \\
      Y \ar[swap]{ru}{i_1}\ar[bend right]{rru}{\id_Y}
      &
      &
    \end{tikzcd} \]
    commutes.
    Then the canonical morphisms
    \begin{equation*}
      \pr_* \xrightarrow{\unit} \pr_*s_*s^* \cong s^*,
      \qquad
      s^! \cong \pr_!s_!s^! \xrightarrow{\counit} \pr_!
    \end{equation*}
    are invertible on $\Gm$-equivariant sheaves.
  \end{prop}

  \begin{cor}\label{cor:contractder}
    For every derived Artin stack $Y$ and every derived vector bundle $E$ over $Y$, the natural transformations
    \begin{align*}
      &\pi_{E*} \xrightarrow{\unit} \pi_{E*}0_{E*}0_E^* \cong 0_E^*\\
      &0_E^! \cong \pi_{E!}0_{E!}0_E^! \xrightarrow{\counit} \pi_{E!}
    \end{align*}
    are invertible on $\Gm$-equivariant sheaves.
    In particular, the functors $\pi_E^*$, $\pi_E^!$, $0_{E*}$, and $0_{E!}$ are all fully faithful on $\Gm$-equivariant sheaves.
  \end{cor}

  \begin{proof}
    The first claim is a special case of \propref{prop:contract}.
    For every $\cK \in \rD^{\Gm}(Y)$ there is a commutative diagram
    \[ \begin{tikzcd}
      \cK \ar{r}{\mrm{unit}_{\pi_E}}\ar[equals]{rd}
      & \pi_{E*}\pi_E^*(\cK) \ar{d}{\mrm{unit}_{0_E}}
      \\
      & 0_E^*\pi_E^*(\cK)
    \end{tikzcd} \]
    where the vertical arrow is invertible by the first claim.
    This shows that $\unit : \id \to \pi_{E*} \pi_E^*$ is invertible on $\Gm$-equivariant sheaves.
    Similarly, on $\Gm$-equivariant sheaves, the counit $\pi_{E!}\pi_E^! \to \id$ is identified with the tautological isomorphism $0^!\pi^! \cong \id$; the counit $0_E^*0_{E*} \to \id$ is identified with $\pi_{E*} 0_{E*} \cong \id$; and the unit $\id \to 0_{E}^! 0_{E!}$ is identified with $\id \cong \pi_{E!} 0_{E!}$.
  \end{proof}

  \begin{cor}\label{cor:chlorine}
    Let $E \in \DVect(S)$.
    For any Cartesian square
    \begin{equation*}
      \begin{tikzcd}
        Y_0 \ar{r}{i}\ar{d}{f_0}
        & Y \ar{d}{f}
        \\
        S \ar{r}{0_E}
        & E
      \end{tikzcd}
    \end{equation*}
    where $f$ is smooth, the unit $\un_Y \to i^!i_!(\un_Y)$ is invertible.
  \end{cor}
  \begin{proof}
    Apply $f_0^*$ on the left to the isomorphism $\unit : \id \to 0_E^! 0_{E!}$ (\corref{cor:contractder}).
    Under the isomorphisms $\Ex^*_!$ and $\Ex^{*!}$ (the latter since $f$ is smooth), the result is identified with $\unit : f_0^* \to i^! i_! f_0^*$.
  \end{proof}

\subsection{Computations on \texorpdfstring{$\quo{\A^1}$}{A1/Gm} and \texorpdfstring{$\quo{\A^1 \times \quo{\A^1}}$}{A1/Gm x A1/Gm}}
\label{sec:j}

We lift some simple computations from \cite{Laumon} (namely, Lemmas~1.4, 3.2, 3.3, and 3.4 of \emph{op. cit.}) to the generality of topological weaves.

\subsubsection{The sheaf \texorpdfstring{$\quo{j}_*(\un)$}{j\_*(1)}}
  Let the notation be as follows:
  \begin{equation}\label{diagram: homogeneous A^1}
    \begin{tikzcd}
      S \ar{r}{i}\ar[equals]{rd}
      & \A^1_S \ar[leftarrow]{r}{j}\ar{d}{\pr}
      & \bG_{m,S} \ar{ld}{q}
      \\
      & S &
    \end{tikzcd}
  \end{equation}
  where $i$ is the zero section.

  We record some basic observations about the sheaf $j_*(\un) \in \bD^{\Gm}(\A^1_S)$, or more precisely
  \begin{equation*}
    \quo{j}_*(\un) \in \bD(\quo{\A^1_S})
  \end{equation*}
  where $\quo{j} : S = \quo{\bG_{m,S}} \hook \quo{\A^1_S}$.

  \begin{prop}\label{prop:j_*(1)}\leavevmode
    \begin{thmlist}
      \item
      \emph{Geometricity.}
      The sheaf $\quo{j}_*(\un)$ is geometric.

      \item
      \emph{Base change.}
      For any morphism $f : S' \to S$, let $j' : S' = \quo{\bG_{m,S'}} \hook \quo{\A^1_{S'}}$ denote the base change of $j$ along $f' : \quo{\A^1_{S'}} \to \quo{\A^1_S}$.
      Then the canonical morphism
      \begin{equation*}
        \Ex^*_* : f'^* \quo{j}_*(\un) \to j'_*(\un)
      \end{equation*}
      is invertible.

      \item
      \emph{Projection formula.}
      For every $\cK \in \bD(\quo{\A^1_S})$, the canonical morphism
      \begin{equation*}
        \mrm{Pr}^*_* : \quo{j}_*(\un) \otimes \cK \to \quo{j}_*\quo{j}^*(\cK)
      \end{equation*}
      is invertible.

      \item\label{item:retolerate}
      We have $\quo{\pr}_!\,\quo{j}_* \cong 0$ in $\bD(\quo{S})$.

      \item\label{item:binh}
      There is a canonical isomorphism
      \[ \quo{j}_*(\un) \cong u_! j_{1*}(\un)[1] \]
      in $\bD(\quo{\A^1})$, where $j_1 : \A^1 \setminus \{1\} \to \A^1$ is the complement of the unit section and $u : \A^1_S \twoheadrightarrow \quo{\A^1_S}$ is the quotient morphism.
    \end{thmlist}
  \end{prop}

  \begin{proof}
  We consider the $\Gm$-scaling quotient of diagram \eqref{diagram: homogeneous A^1}, writing the resulting morphisms as $\quo{i} : \quo{S} \to \quo{\A^1_S}$, etc.
  We have the localization triangles
  \begin{equation}\label{eq:harmonica}
    \quo{j}_!
    \cong \quo{j}_! (\quo{j}^*) ( \quo{j}_*) 
    \xrightarrow{\counit} \quo{j}_*
    \xrightarrow{\unit} (\quo{i}_*) ( \quo{i}^* ) \quo{j}_*.
  \end{equation}
  Applying $\quo{\pr}_!$ yields
  \begin{equation*}
    \quo{q}_!
    \to \quo{\pr}_! \ \quo{j}_*
    \xrightarrow{\unit} \quo{\pr}_! (\quo{i}_*) ( \quo{i}^* ) \quo{j}_*
    \cong \quo{i}^*  \ \quo{j}_*.
  \end{equation*}
  We have $\quo{\pr}_! \ \quo{j}_* \cong \quo{i}^! \ \quo{j}_* \cong 0$ \itemref{item:retolerate} by the Contraction Principle (\propref{prop:contract}) and base change formula.
  We deduce a canonical isomorphism
  \begin{equation*}
    \quo{i}^* \ \quo{j}_* 
    \cong \quo{q}_![1].
  \end{equation*}
  In particular, we can rewrite \eqref{eq:harmonica} as an exact triangle
  \begin{equation*}
    \quo{j}_!
    \to \quo{j}_*
    \to \quo{i}_*  \ \quo{q}_![1].
  \end{equation*}

  Since $\quo{j}_!$, $\quo{i}_* = \quo{i}_!$, and $\quo{q}_!$ preserve geometric objects, it follows that $\quo{j}_*$ preserves geometric objects.
  Similarly, the base change and projection formulas for $\quo{j}_!$, $\quo{i}_!$, and $\quo{q}_!$ yield the claimed base change and projection formulas for $\quo{j}_*(\un)$.\footnote{We omit verification of commutativity of some diagrams, expressing e.g. the compatibility of $\Ex^*_* : f'^* \ \quo{j_*}(\un) \to j'_*(\un)$ and $\Ex^*_! : j'_!(\un) \to f'^* \  \quo{j_!}(\un)$.}

  For the final claim~\itemref{item:binh}, we begin by observing the canonical isomorphism
  \[ (\quo{j}^* )u_! j_{1*} (\un) \cong \un[-1] \]
  using the base change formula $(\quo{j}^* ) u_! \cong q_! j^*$ for the Cartesian square
  \[ \begin{tikzcd}
    \bG_{m,S} \ar{r}{j}\ar{d}{q}
    & \A^1_S \ar{d}{u}
    \\
    S = \quo{\bG_{m,S}} \ar{r}{\quo{j}}
    & \quo{\A^1_S}
  \end{tikzcd} \]
  and the observation that
  \[ q_! j^* j_{1*}(\un) \cong \un[-1] \]
  which is a straightforward computation using the base change formula and localization.

  It will then suffice to show that the unit morphism
  \[ u_! j_{1*} (\un) \to \quo{j}_*\quo{j}^*(u_! j_{1*} (\un)) \cong \quo{j}_*(\un) \]
  is invertible.
  By localization, this is equivalent to showing that $\quo{i}^! (u_! j_{1*} (\un)) \cong 0$.
  By the Contraction Principle (\propref{prop:contract}), we have
  \[ \quo{i}^! (u_! j_{1*} (\un)) \cong \quo{\pr}_! (u_! j_{1*} (\un))  \cong \quo{q}_! (\pr_! j_{1*}(\un)) = 0\]
  since $\pr_! j_{1*}(\un) = 0 \in \bD(S)$ by a straightforward localization argument.
  \end{proof}

  \begin{cor}[Preservation of geometricity]\label{cor: FT preserve local constructibility}
    Assume the weave $\bD$ is as in \thmref{thm:dgm} (e.g. $\bD=\cD_{\mrm{mot}}(-; \bQ)$).
    Then for every derived vector bundle $E \rightarrow S$ over a derived Artin stack $S$, the functor $\FT_E$ preserves geometricity.
  \end{cor}
  
  \begin{proof}
    With notation as in \eqref{eq: FT definition}, the functor $\FT_E$ is the composite of the functors $\pr_2^*(-)$, $(-) \otimes \sP_E[-1]$, and $\pr_{1!}(-)$. As discussed in \S \ref{sssec: weave gm}, geometricity is preserved by $*$-pullbacks, $!$-pullbacks, and tensor product with geometric objects. Hence $\pr_2^*(-)$ preserves geometricity on general grounds, $(-) \otimes \sP_E[-1]$ preserves geometricity because Proposition \ref{prop:j_*(1)}(v) shows that $\sP_E[-1]$ is geometric, and $\pr_{1!}(-)$ preserves geometricity because it can be identified with $!$-pullback to the zero section by Proposition \ref{prop:contract}.
  \end{proof}

  \subsubsection{The square of \texorpdfstring{$\quo{j}_*(\un)$}{j\_*(1)}} We establish more technical properties of $\quo{j}_*(\un)$. 
  
    \begin{lem}\label{lem:kunn}
      There is a canonical isomorphism
      \[ \quo{j}_*(\un) \boxtimes_S \quo{j}_*(\un) \cong (\quo{j} \fibprod_S \quo{j})_*(\un)  \in \bD(\quo{\A^1_S} \fibprod_S \quo{\A^1_S}).\]
    \end{lem}
    
    \begin{proof}
      By definition,
      \[ \quo{j}_*(\un) \boxtimes_S \quo{j}_*(\un) \coloneqq \pr_1^*(\quo{j}_*(\un)) \otimes \pr_2^*(\quo{j}_*(\un)) \]
      where $\pr_1$ and $\pr_2$ are the projections $\quo{\A^1_S} \fibprod_S \quo{\A^1_S} \to \quo{\A^1_S}$.
      By smooth base change, we have $\pr_1^* \quo{j}_*(\un) \cong j_{1*}(\un)$, where $j_1 = \quo{j} \fibprod_S \id : S \fibprod_S \quo{\A^1_S} \to \quo{\A^1_S} \fibprod_S \quo{\A^1_S}$ and similarly for the second term.
      By the projection formula (\propref{prop:j_*(1)}),
      \begin{equation*}
        j_{1*}(\un) \otimes j_{2*}(\un)
        \cong j_{1*}j_1^*j_{2*}(\un).
      \end{equation*}
      By smooth base change for the Cartesian square
      \[ \begin{tikzcd}
        S \fibprod_S S \ar{r}\ar{d}
        & \quo{\A^1_S} \fibprod_S S \ar{d}{j_2}
        \\
        S \fibprod_S \quo{\A^1_S} \ar{r}{j_1}
        & \quo{\A^1_S} \fibprod_S \quo{\A^1_S},
      \end{tikzcd} \]
      where the diagonal composite is $\quo{j} \times \quo{j}$, we have
      \begin{equation*}
        j_{1*}j_1^*j_{2*}(\un) \cong (\quo{j} \times \quo{j})_*(\un),
      \end{equation*}
      whence the claim.
    \end{proof}

  Consider the morphism
  \[ c = (\quo{\pr_1}, \quo{\pr_2}) : \quo{\A^2_S} \to \quo{\A^1_S} \fibprod_S \quo{\A^1_S} \]
  induced by the projections $\A^2 \to \A^1$ (which are $\Gm$-equivariant), which exhibits $\quo{\A^1_S} \fibprod_S \quo{\A^1_S}$ as the quotient of $\quo{\A^2_S}$ by the action $\lambda\cdot(x,y) = (x,\lambda\cdot y)$.
  We have a commutative diagram
  \[ \begin{tikzcd}
    \bG_{m,S} \fibprod_S \bG_{m,S} \ar{r}\ar[twoheadrightarrow]{d}
    & \A^2_S \setminus \{0\}_S \ar{r}{j^2}\ar[twoheadrightarrow]{d}
    & \A^2_S \ar[twoheadrightarrow]{d}\ar[leftarrow]{r}{i^2}
    & S\ar[twoheadrightarrow]{d}
    \\
    \{1\}_S \fibprod_S \bG_{m,S} \ar{r}\ar[twoheadrightarrow]{d}
    & \quo{(\A^2_S \setminus \{0\}_S)} \ar{r}{\quo{j^2}}\ar[twoheadrightarrow]{d}
    & \quo{\A^2_S} \ar[twoheadrightarrow]{d}{c}\ar[leftarrow]{r}{\quo{i^2}}
    & \quo{S}\ar[twoheadrightarrow]{d}{d}
    \\
    S \ar{r}{j_a}
    & U \ar{r}{j_b}
    & \quo{\A^1_S} \fibprod_S \quo{\A^1_S}\ar[leftarrow]{r}{i_b}
    & \quo{S} \fibprod_S \quo{S}
  \end{tikzcd} \]
  where the bottom row is the quotient of the middle one by the $\Gm$-action which scales the second coordinate (with weight $1$).
  In the two left-hand columns, the horizontal rows are factorizations of $j \times j$, $\quo{j \times j}$, and $\quo{j} \times \quo{j}$, respectively.
  In the two right-hand columns, the horizontal rows are complementary open/closed immersions.

  Let $\Delta \subset \A^2_S$ denote the diagonal, $i_\Delta : \Delta \hook \A^2_S$ the inclusion, and $j_\Delta$ the open complement of $i_\Delta$.
  
  \begin{lem}\label{lem:pnpqnq}
    There is a canonical isomorphism
    \[ (\quo{j} \times \quo{j})_*(\un) \cong c_! (\quo{j}_{\Delta*})(\un)[1] \in \bD(\quo{\A^1_S} \times_S \quo{\A^1_S}). \]
  \end{lem}

  \begin{rem}
    Let $e : \A^2_S \to \A^1_S$ denote the ``difference'' morphism, given informally by $(x,y) \mapsto x-y$.
    By smooth base change for the square
    \[ \begin{tikzcd}
      \quo{(\A^2_S\setminus\Delta)} \ar{r}{\quo{j_\Delta}}\ar{d}{\quo{e}}
      & \quo{\A^2_S} \ar{d}{\quo{e}}	
      \\
      \quo{\bG_{m,S}} \ar{r}{\quo{j}}
      & \quo{\A^1_S}
    \end{tikzcd} \]
    we can write $\quo{e}^* \quo{j}_* (\un) \cong \quo{j}_{\Delta,*}(\un)$ and hence also
    \begin{equation}
      (\quo{j} \times \quo{j})_*(\un) \cong c_! \quo{e}^* \quo{j}_* (\un)[1].
    \end{equation}
  \end{rem}

\begin{proof}[Proof of Lemma \ref{lem:pnpqnq}]   
Set
    $$\cK \coloneqq c_!\quo{j_{\Delta *}}(\un) \in \bD(\quo{\A^1_S}\fibprod_S\quo{\A^1_S}).$$
    We claim there are isomorphisms:
    \begin{enumerate}
      \item\label{item:aouuhp/i} $i_b^!(\cK) \cong 0$,
      \item\label{item:aouuhp/j} $j_b^*(\cK) \cong {j}_{a*}(\un)[-1]$.
    \end{enumerate}
    By localization it will follow that the unit
    \[ \mrm{unit} : \cK \to j_{b*}j_b^*(\cK) \cong j_{b*}j_{a*}(\un)[-1] \cong (\quo{j} \times \quo{j})_*(\un)[-1] \]
    is invertible, as claimed.

    \emph{Proof of \itemref{item:aouuhp/i}}.
    Since $\quo{i^2}$ is the zero section of the vector bundle $\quo{p^2} : \quo{\A^2_S} \to \quo{S}$ (where $p^2 : \A^2_S \to S$ is the projection), the Contraction Principle (\propref{prop:contract}) yields natural isomorphisms isomorphism $ (\quo{i^{2}})^! \cong (\quo{p_{0}^2})_!$ and $i_b^! \cong (\quo{p} \times \quo{p})_!$. Hence we have 
    \begin{align*}
      i_b^! c_!\, \quo{j_{\Delta *}}
      &\cong (\quo{p} \times \quo{p})_! c_! \quo{j_{\Delta*}} \cong d_!\, \quo{p_{0!}^2}\, \quo{j_{\Delta*}} \cong d_!\, \quo{i^{2!}}\, \quo{j_{\Delta*}}
    \end{align*}
    where $d$ is the diagonal of $\quo{S}$ as in the diagram above.
    But $(\quo{i^{2}})_! (\quo{j_{\Delta}})_* \cong 0$ by base change (as $0 \in \A^2$ is contained in $\Delta$).

    \emph{Proof of \itemref{item:aouuhp/j}}.
    Consider the diagram of Cartesian squares
    
    \[ \begin{tikzcd}[matrix xscale=0.3]
      \Gm\setminus\{1\}\ar[equals]{r}\ar[hookrightarrow]{d}{j'_1}
      & \quo{(\Gm \times \Gm\setminus\Delta)} \ar{r}\ar[hookrightarrow]{d}
      & \quo{(\A^1 \times \Gm\setminus\Delta)} \ar[equals]{r}\ar[hookrightarrow]{d}
      & \A^1 \setminus \{1\} \ar{r}\ar[hookrightarrow]{d}{j_1}
      & \quo{(\A^2\setminus\Delta)} \ar[hookrightarrow]{d}{\quo{j_\Delta}}
      \\
      \Gm \ar[equals]{r}\ar{d}{q}
      & \quo{(\Gm \times \Gm)} \ar{r}\ar[twoheadrightarrow]{d}
      & \quo{(\A^1 \times \Gm)} \ar[equals]{r}\ar[twoheadrightarrow]{d}
      & \A^1 \ar{r}\ar[twoheadrightarrow]{d}{f}
      & \quo{\A^2} \ar[twoheadrightarrow]{d}{c}
      \\
      S \ar[equals]{r}
      & \quo{\Gm} \ar[hookrightarrow]{r}
      & \quo{\A^1} \ar[equals]{r}
      & \quo{\A^1} \ar[hookrightarrow]{r}
      & \quo{\A^1} \times \quo{\A^1}
    \end{tikzcd} \]
    where we omit the subscripts $_S$ for sanity of notation.
    By base change we get
    \[
      (\quo{j} \times \quo{j})^* c_! \,\quo{j_{\Delta*}} (\un)
      \cong q_{!} j'_{1*} (\un)
      \cong \un[-1]
    \]
    where the second isomorphism follows easily from localization.

    Since $\quo{j} \times \quo{j} = j_b \circ j_a$ this isomorphism gives by adjunction a morphism
    \[ j_b^*(\cK) \to j_{a*}(\un)[-1] \]
    in $\bD(U)$ which we claim is invertible.
    Write $U$ as the union of the two opens $[\quo{(\A^1\times\Gm)}/\Gm]$ and $[\quo{(\Gm\times\A^1)}/\Gm]$, where $[-/\Gm]$ is the quotient by the scaling action on the second coordinate.
    Over either open this morphism restricts to the isomorphism
    \[ u_! j_{1*}(\un) \cong \quo{j}_*(\un)[-1] \]
    constructed in \propref{prop:j_*(1)}(v). 
\end{proof}

  \subsection{Easy proofs}
\label{sec:easy} Here we spell out, for completeness, proofs which carry over essentially verbatim from the case of classical vector bundles. 

\subsubsection{Zero bundle}
\label{ssec:easy/zero}

We prove \propref{prop:zero}. Recall that $i \co 0_S \inj \A^1_S$ is the zero section and $j \co \G_{m,S} \inj \A^1_S$ is the complementary open embedding. 

  The evaluation map $\ev_{0_S} : [\wh{0}_S/\Gm] \times [0_S/\Gm] \to [\A^1_S/\Gm]$ factors as the composite
  \[ B\bG_{m,S} \times B\bG_{m,S} \xrightarrow{m} B\bG_{m,S} \xrightarrow{\quo{i}} [\A^1_S/\bG_{m,S}] \]
  where the first map $m$ is $(L, L') \mapsto L \otimes L'$.
  By localization, one computes $\quo{i^*} \quo{j_{*}}(\un) \cong {q_{!}}(\un)[1]$, where $q : S \to B\bG_{m,S}$.
  Note that we have a Cartesian square
  \[ \begin{tikzcd}
    B\bG_{m,S} \ar{r}\ar{d}{(\id,o)}
    & S \ar{d}{q}
    \\
    B\bG_{m,S} \times B\bG_{m,S} \ar{r}{m}
    & B\bG_{m,S}
  \end{tikzcd} \]
  where the left-hand vertical map is $L \mapsto (L, \wh{L})$.
  Thus the kernel of $\FT_0$ is
  \[
    \sP_0 = \ev_{0_S}^* \quo{j}_*(\un)
    \cong m^* q_!(\un)[1]
    \cong (\id,o)_!(\un)[1]
  \]
  and $\FT_0$ itself is given by
  \[
    \FT_0(\cK) \cong \pr_{1!}(\pr_2^*(\cK) \otimes (\id,o)_!(\un))[1][-1]\\
    \cong \pr_{1!}(\id,o)_!(\id,o)^*\pr_2^*(\cK)\\
    \cong o^*(\cK)
  \]
  by the projection formula.

  For the remaining isomorphisms, note that $o$ is finite étale so that $o_* \cong o_!$ and $o^* \cong o^!$.
  Since $o$ is an involution, we also have $o^*o^* \cong \id$, hence $o^* \cong o_*$ and $o_! \cong o^!$. \qed 

\subsubsection{Base change, part 1}
\label{ssec:basechange1}

  Let the notation be as in \propref{prop:funbase}.

  The base change property for $\quo{j}_*(\un)$ (\propref{prop:j_*(1)}) implies that there is a canonical isomorphism
  \begin{equation}
    f_{\wh{E} \fibprod_S E}^* (\sP_E) \cong \sP_{E'}
  \end{equation}
  where $f_{\wh{E} \fibprod_S E} : \wh{E'} \fibprod_{S'} E' \to \wh{E} \fibprod_S E$ is the base change of $f$.
  The isomorphisms \eqref{eq:BC^*} and \eqref{eq:BC_!} follow immediately using the base change and projection formulas.\qed 

  The proofs of the remaining two isomorphisms will be done after proving involutivity.

\subsubsection{Functoriality, part 1}
\label{ssec:easy/fun}

  We prove the isomorphism \eqref{eq:funA} of \propref{prop:fun}. We omit the base $S$ from notation; all products are over $S$. 

  Let $\phi : E' \to E$ be a morphism of derived vector bundles.
  We have the commutative square
  \[ \begin{tikzcd}
    \quo{\wh{E}} \times \quo{E'} \ar{r}{\id \times {\phi}}\ar{d}{{\wh{\phi}} \times \id}
    & \quo{\wh{E}} \times \quo{E} \ar{d}{\ev_E}
    \\
    \quo{\wh{E'}} \times \quo{E'} \ar{r}{\ev_{E'}}
    & \quo{\A^1}
  \end{tikzcd} \]
  whence a canonical isomorphism
  \[ (\id \times {\phi})^*\sP_{E}^* \cong ({\wh{\phi}} \times \id)^*(\sP_{E'}). \]

  We use the following commutative diagram, where all squares are Cartesian.
  \[ \begin{tikzcd}[matrix xscale=0.2]
    & & \quo{E'} \times \quo{E'} \ar[leftarrow,swap]{ld}{\wh{\phi}\times\id}\ar{rd}{\id\times f} \ar[bend right=50,swap]{lldd}{\pr_2}\ar[bend left=50]{rrdd}{\pr_1} & &
    \\
    & \quo{\wh{E}} \times \quo{E'} \ar[swap]{ld}{\pr_2}\ar{rd}{\id\times f} & & \quo{\wh{E'}} \times \quo{E} \ar[leftarrow,swap]{ld}{\wh{\phi} \times \id}\ar{rd}{\pr_1} &
    \\
    \quo{E'} \ar{rd}{\phi} & & \quo{\wh{E}} \times \quo{E} \ar[swap]{ld}{\pr_2}\ar{rd}{\pr_1} & & \quo{\wh{E'}} \ar[leftarrow,swap]{ld}{\wh{\phi}}
    \\
    & \quo{E} & & \quo{\wh{E}} & &
  \end{tikzcd} \]

  By the base change and projection formulas we have:
  \begin{align*}
    \FT_{E} {\phi_!}(\cK)[1]
    &\cong \pr_{1!} (\pr_2^*({\phi_!}\cK) \otimes \sP_{E}) \cong \pr_{1!} ((\id \times {\phi})_!(\pr_2^*\cK) \otimes \sP_{E})\\
    &\cong \pr_{1!}(\id \times {\phi})_! (\pr_2^*\cK \otimes (\id \times {\phi})^*(\sP_{E})) \cong \pr_{1!} (\pr_2^*\cK \otimes ({\wh{\phi}} \times \id)^*(\sP_{E'})).
  \end{align*}
  Similarly we have:
  \begin{align*}
    {\wh{\phi}^*} \FT_{E'}(\cK)[1]
    &\cong {\wh{\phi}^*} \pr_{1!} (\pr_2^*(\cK) \otimes \sP_{E'}) \cong \pr_{1!} (\id \times {\phi})_! ({\wh{\phi}} \times \id)^* (\pr_2^*(\cK) \otimes \sP_{E'})\\
    &\cong \pr_{1!} ({\wh{\phi}} \times \id)^* (\pr_2^*(\cK) \otimes \sP_{E'}) \cong \pr_{1!} (\pr_2^*(\cK) \otimes ({\wh{\phi}} \times \id)^* \sP_{E'}).
  \end{align*}
  Comparing these gives the desired isomorphism. \qed 
  
\subsubsection{Base change for the twist, part 1}
\label{ssec:basechangetwist1}

  Given a morphism $f : S' \to S$ and $E' \coloneqq E \fibprod_S S' \in \DVect(S')$, we show the isomorphisms
  \begin{equation*}
    f^* \sL^E \cong \sL^{E'}
  \end{equation*}
  and
  \begin{equation*}
    f^*_E((-)\twbrace{E}) \cong f_E^*(-)\twbrace{E'},
    \qquad
    f_{E!}(-)\twbrace{E} \cong f_{E!}((-)\twbrace{E'})
  \end{equation*}
  of \lemref{lem:twistbase}.
  The remaining parts of these statements will be proven in \subsectionref{ssec:basechangetwist2}.

  By \eqref{eq:BC^*} we have:
  \begin{align*}
    f^*\sL^E
    &= f^* \pi_{E!} \FT_{\wh{E}}(\un) \cong \pi_{E'!} f_{\wh{E}}^* \FT_{\wh{E}}(\un) \cong \pi_{E'!} \FT_{\wh{E'}}(\un) = \sL^{E'}.
  \end{align*}

  From $f^* \sL^E \cong \sL^{E'}$ we now deduce
  \begin{align*}
    f_E^*((-)\twbrace{E})
    &= f_E^*(-\otimes\pi_E^*(\sL^E)) \cong f_E^*(-)\otimes f_E^*\pi_E^*(\sL^E) \cong f_E^*(-)\otimes \pi_{E'}^*f^*(\sL^E)\\
    &\cong f_E^*(-)\otimes \pi_{E'}^*(\sL^{E'}) = f_E^*(-)\twbrace{E'}.
  \end{align*}

  Similarly, using the projection formula we have:
  \begin{align*}
    f_{E!}(-)\twbrace{E}
    &= f_{E!}(-) \otimes \pi_E^*(\sL^E) \cong f_{E!}(- \otimes f_E^*\pi_E^*(\sL^E)) \cong f_{E!}(- \otimes \pi_{E'}^*f^*(\sL^E))\\
    &\cong f_{E!}(- \otimes \pi_{E'}^*(\sL^{E'})) = f_{E!}((-)\twbrace{E'}).
  \end{align*}
\qed   

\subsection{Proof of involutivity assuming cosupport}
\label{sec:cosuppinvol}

In this section we prove \lemref{lem:cosuppinvol}.

\subsubsection{Kernel of the square}

  Consider the following commutative diagram:
  \[\begin{tikzcd}[matrix xscale=0.1, matrix yscale=1.3]
    && \quo{E} \fibprod_S \quo{E} \\
    && \quo{E} \fibprod_S \quo{\wh{E}} \fibprod_S \quo{E}  \\
    & \quo{E} \fibprod_S \quo{\wh{E}} && \quo{\wh{E}} \fibprod_S \quo{E} \\
    \quo{E} && \quo{\wh{E}} && \quo{E}
    \arrow["\pr_1", swap, from=3-2, to=4-1]
    \arrow["\pr_2", from=3-2, to=4-3]
    \arrow["\pr_1", swap, from=3-4, to=4-3]
    \arrow["\pr_2", from=3-4, to=4-5]
    \arrow["\pr_{12}", swap, from=2-3, to=3-2]
    \arrow["\pr_{23}", from=2-3, to=3-4]
    \arrow["\pr_{13}", from=2-3, to=1-3]
    \arrow["\pr_1", swap, bend right, from=1-3, to=4-1]
    \arrow["\pr_2", bend left, from=1-3, to=4-5]
  \end{tikzcd}\]
  A straightforward application of base change and projection formulas yields an identification
  \begin{equation}\label{eq:FT2}
    {\FT_{\wh{E}}} \circ {\FT_{E}}(-)
    \cong \pr_{1!}(\pr_2^*(-) \otimes \sP'')[-2]
  \end{equation}
  where
  \[ \sP'' \coloneqq \pr_{13!}(\pr_{12}^*(\sP_{\wh{E}}) \otimes \pr_{23}^*(\sP_E)) \in \bD(\quo{E} \fibprod_S \quo{E}). \]

  Consider the $\Gm$-action scaling both coordinates of $E \fibprod_S E$.
  The two projections $\pr_1, \pr_2 : E \fibprod_S E \to E$ are $\Gm$-equivariant.
  We let $c : \quo{(E \fibprod_S E)} \to \quo{E} \fibprod_S \quo{E}$ denote the induced morphism $(\quo{\pr_1}, \quo{\pr_2})$.
  We denote by $e : E \fibprod_S E \to E$ the ``difference'' morphism, given informally by $(x,y) \mapsto x-y$; this is also $\Gm$-equivariant.

  \begin{lem}\label{lem:P''}
    For any $E \in \DVect(S)$, there is a canonical isomorphism
    \begin{equation*}
      \sP''
      \cong c_! (\quo{e})^* \FT_{\wh{E}}(\un)[2].
    \end{equation*}
  \end{lem}

\begin{proof}
  Consider the morphism
  \[ \ev'' : \quo{E} \fibprod_S \quo{\wh{E}} \fibprod_S \quo{E} \to \quo{\A^1_S} \fibprod_S \quo{\A^1_S} \]
  given on points by
  \begin{multline*}
    \big((L, x : L \to E), (L', \phi : L' \to \wh{E}), (L'', y : L'' \to E)\big)\\
    \,\mapsto\,
    \big(
    ( L' \otimes L, L' \otimes L \xrightarrow{\phi \otimes x} \wh{E} \fibprod_S E \xrightarrow{\ev} \A^1_S ),
    ( L' \otimes L'', L' \otimes L'' \xrightarrow{\phi \otimes y} \wh{E} \otimes E \xrightarrow{\ev} \A^1_S )
    \big).
  \end{multline*}

  We have commutative squares
  \[ \begin{tikzcd}
    \quo{E} \fibprod_S \quo{\wh{E}} \fibprod_S \quo{E} \ar{r}{\pr_{23}}\ar{d}{\ev''}
    & \quo{E} \fibprod_S \quo{\wh{E}} \ar{d}{\ev}
    \\
    \quo{\A^1_S} \fibprod_S \quo{\A^1_S} \ar{r}{\pr_1}
    & \quo{\A^1_S},
  \end{tikzcd}
  \quad
  \begin{tikzcd}
    \quo{E} \fibprod_S \quo{\wh{E}} \fibprod_S \quo{E} \ar{r}{\pr_{12}}\ar{d}{\ev''}
    & \quo{E} \fibprod_S \quo{\wh{E}} \ar{d}{\ev'}
    \\
    \quo{\A^1_S} \fibprod_S \quo{\A^1_S} \ar{r}{\pr_2}
    & \quo{\A^1_S}.
  \end{tikzcd} \]
  This yields
  \begin{equation}\label{eq: kernel eq 1}
    \sP''
    \cong \pr_{13!} \ev''^*(\quo{j}_*(\un) \boxtimes_S \quo{j}_*(\un))
    \cong \pr_{13!} \ev''^* c_! (\quo{j}_{\Delta})_* (\un)[1]
  \end{equation}
  where the second isomorphism comes from Lemmas~\ref{lem:kunn} and \ref{lem:pnpqnq}.

  Next observe that we have a commutative diagram
  \begin{equation}\label{eq:lanthana}
    \begin{tikzcd}[matrix xscale=0.5]
      \quo{S} \ar{r}{\quo{0}}\ar[phantom]{rd}{\di}
      & \quo{E}\ar[phantom]{rd}{\di}
      & \quo{\wh{E}} \fibprod_S \quo{E} \ar[swap]{l}{\pr_2}\ar{r}{\ev}
      & \quo{\A^1_S} \ar[leftarrow]{r}{\quo{j}}\ar[phantom]{rd}{\di}
      & \quo{\Gm}\ar[leftarrow]{d}{\quo{e}}
      \\
      \quo{E} \ar{r}{\quo{\Delta}}\ar[swap]{u}{\quo{\pi_E}}\ar[equals]{d}
      & \quo{(E\fibprod_S E)} \ar[swap]{u}{\quo{e}}\ar{d}{c}\ar[phantom]{rd}{\di}
      & \quo{\wh{E}} \fibprod_S \quo{(E\fibprod_S E)} \ar[swap]{l}{\pr_2}\ar{r}\ar[swap]{u}{\id\times\quo{e}}\ar{d}\ar[phantom]{rd}{\di}
      & \quo{\A^2_S}\ar[swap]{u}{\quo{e}}\ar{d}{c}\ar[leftarrow]{r}{\quo{j_\Delta}}
      & \quo{(\A^2\setminus\Delta)}
      \\
      \quo{E} \ar{r}{\Delta}
      & \quo{E}\fibprod_S \quo{E}
      & \quo{E}\fibprod_S\quo{\wh{E}}\fibprod_S\quo{E} \ar[swap]{l}{\pr_{13}}\ar{r}{\ev''}
      & \quo{\A^1_S}\fibprod_S\quo{\A^1_S}
      &
    \end{tikzcd}
  \end{equation}
  where the squares labeled by $\di$ are derived Cartesian.
  The notation is as follows:
  \begin{itemize}
    \item
    The two projections $\pr_1, \pr_2 : E \fibprod_S E \to E$ are $\Gm$-equivariant.
    We let $c : \quo{(E \fibprod_S E)} \to \quo{E} \fibprod_S \quo{E}$ denote the induced morphism $(\quo{\pr_1}, \quo{\pr_2})$.
    Similarly for $c : \quo{\A^2_S} \to \quo{\A^1_S} \fibprod_S \quo{\A^1_S}$.

    \item 
    $\Delta$ is the diagonal of $\quo{E}$ and $\quo{\Delta}$ is the quotient of the diagonal of $E$.

    \item
    $e : E \fibprod_S E \to E$ is the ``difference'' morphism, given informally by $(x,y) \mapsto x-y$.
    Similarly for $e : \A^2 \to \A^1$.
  \end{itemize}

From \eqref{eq: kernel eq 1}, applying proper and smooth base change isomorphisms to the Cartesian squares in \eqref{eq:lanthana} gives 
  \begin{equation}\label{eq: kernel eq 2}
    \sP''
    \cong \pr_{13!} \ev''^* c_! \quo{j}_{\Delta*} \quo{e}^*(\un)[1] \cong c_! \quo{e}^* \pr_{2!} \ev_E^* \quo{j}_*(\un)[1].
  \end{equation}

  Under the automorphism of $\wh{\quo{E}} \fibprod_S \quo{E}$ which swaps the factors, the morphism $\ev_E : \wh{\quo{E}} \fibprod_S \quo{E} \to \quo{\A^1_S}$ is identified with $\ev_{\wh{E}}$ and the projection $\pr_2 : \wh{\quo{E}} \fibprod_S \quo{E} \to \quo{E}$ is identified with $\pr_1 : \quo{E} \fibprod_S \wh{\quo{E}} \to \quo{E}$.
  Thus by definition we have an isomorphism
  \begin{equation}\label{eq: FT 1} 
  \FT_{\wh{E}}(\un)
    \cong \pr_{2!} \ev_E^* \quo{j}_*(\un)[-1]. \end{equation}
  Combining \eqref{eq: kernel eq 2} and \eqref{eq: FT 1}, we obtain an isomorphism
  \begin{equation*}
    \sP'' \cong c_! \quo{e}^* \FT_{\wh{E}}(\un)[2]
  \end{equation*}
  as desired.

\end{proof}

\subsubsection{Proof of Lemma~\ref{lem:cosuppinvol}}\label{sssec: proof of involutivity}

  Since $E$ satisfies the cosupport property, we have the canonical isomorphism
  \begin{equation*}
    0_{E!}(\cL^{E})
    \coloneqq 0_{E!}0_E^! \FT_{\wh{E}}(\un)
    \to \FT_{\wh{E}}(\un).
  \end{equation*}
  Applying $c_! \quo{e}^*$ yields a canonical isomorphism
  \begin{equation*}
    c_! \quo{e}^* 0_{E!}(\cL^{E})[2]
    \to c_! \quo{e}^* \FT_{\wh{E}}(\un)[2]
    \cong \sP''.
  \end{equation*}
  By base change, using the diagram \eqref{eq:lanthana} again, we obtain a canonical isomorphism
  \begin{equation*}
    \Delta_! {\pi_E^*}(\cL^{E})[2]
    \to \sP''.
  \end{equation*}

  Finally, plugging this into \eqref{eq:FT2} yields
  \begin{align*}
    \FT_{\wh{E}}\FT_E(\cK)
    &\cong \pr_{1!}(\pr_2^*(\cK) \otimes \sP'')[-2] \cong \pr_{1!}(\pr_2^*(\cK) \otimes \Delta_! {\pi_E^*}(\cL^{E}))\\
\text{(Projection formula) } \implies    &\cong \pr_{1!}(\Delta_!(\Delta^*\pr_2^*(\cK) \otimes {\pi_E^*}(\cL^{E}))) \cong \cK \otimes {\pi_E^*}(\cL^{E}). \qed 
  \end{align*}

\subsection{Co/support and involutivity for \texorpdfstring{$E \ge 0$}{E<=0}}
\label{sec:supp}

In this section we fix $E \in \DVect(S)$ of amplitude $\le 0$.
We prove the support property \propref{prop:suppinvol} for $E$ (which implies the cosupport property too) and deduce the optimal form of involutivity in this case.

  \subsubsection{Restriction to zero}

    We first compute the inverse image along $\quo{0}_E : \quo{S} \to \quo{E}$:
    \[
      \quo{0}_E^* \FT_{\wh{E}}(\un)
      \cong \FT_{0_S} (\quo{\pi}_{\wh{E}!} (\un))
      \cong \FT_{0_S} (\un\vb{-\wh{E}})
      \cong \un\vb{-\wh{E}},
    \]
    using functoriality as well as the isomorphisms $\quo{\pi}_{\wh{E}}^! \cong \quo{\pi}_{\wh{E}}^*\vb{\wh{E}}$ (Poincaré duality) and $(\quo{\pi}_{\wh{E}})_! (\quo{\pi}_{\wh{E}})^! \cong \id$ (homotopy invariance) for the morphism $\quo{\pi}_{\wh{E}} : \quo{\wh{E}} \to \quo{S}$ (smooth because $\wh{E}$ is of amplitude $\le 0$).

  \subsubsection{Support and cosupport properties}

    Since $E$ is of amplitude $\ge 0$, the zero-section $0_E$ is a closed immersion and $0_{E*} = 0_{E!}$.
    By localization, invertibility of either \eqref{eq:doctrinally} or \eqref{eq:urodelan} is equivalent to the assertion that $\FT_{\wh{E}}(\un)$ is supported on the image of $\quo{0}_E : \quo{S} \to \quo{E}$.

    It will suffice to show that for every residue field $v : \Spec(\kappa) \to \quo{E}$ that factors through the complement of $0_E$, $v^* \FT_{\wh{E}}(\un)$ vanishes.
    Without loss of generality, we may replace $S$ by $\Spec(\kappa)$ and show that for every nowhere zero section $s$ of $E \to S$, the inverse image along $s' : S \xrightarrow{s} E \twoheadrightarrow \quo{E}$ vanishes.

    We have the following commutative diagram:
    \[ \begin{tikzcd}
      S\ar[leftarrow]{r}{a_{\wh{E}}}\ar{d}{s'}
      & \quo{\wh{E}} \ar{r}{\quo{\ev_s}}\ar{d}{\id \times s'}
      & \quo{\A^1_S} \ar[equals]{d}
      \\
      \quo{E}\ar[leftarrow]{r}{\pr_2}
      & \quo{\wh{E}}\times \quo{E} \ar{r}{\ev}
      & \quo{\A^1_S}
    \end{tikzcd} \]
    where the left-hand square is Cartesian.
    Here $\quo{\ev_s}$ is the $\Gm$-quotient of the ``evaluation at $s$'' morphism $\ev_s : \wh{E} \to \A^1_S$, and $a_{\wh{E}}$ is the projection.
    Hence we have
    \begin{equation}\label{eq:quadrans}
      s'^* \pr_{2!} \ev^*
      \cong a_{\wh{E}!} (\id\times s')^* \ev^*
      \cong a_{\wh{E}!} \quo{\ev_s^*}.
    \end{equation}
    Since $s$ is nowhere zero, $\ev_s$ is surjective and its fibre $F = \Fib(\ev_s : \wh{E} \to \A^1_S)$ is of amplitude $\le 0$.
    Since $S$ is affine, it follows that $\ev_s$ admits a section, hence can be identified with the projection $F \fibprod_S \A^1_S \to \A^1_S$.
    Using Poincaré duality and homotopy invariance for the latter we have
    \begin{equation}\label{eq: connective (co)support 1}
      a_{\wh{E}!} \,\quo{\ev_s^*} \cong a_{!} \,\quo{\ev_{s!}} \quo{\ev_s^*} \cong a_{!} \,\quo{\ev_{s!}} \quo{\ev_s^!} \vb{-F} \cong a_{!} \vb{-F}
    \end{equation}
    where $a : \quo{\A^1_S} \to S$ is the projection.
    Combining with \eqref{eq:quadrans} we deduce
    \begin{equation}\label{eq: connective (co)support 2}
      s'^* \FT_{\wh{E}}(\un)
      \coloneqq s'^* \pr_{2!} \ev^*(\quo{j}_*(\un))
      \cong a_{!}\,\quo{j}_*(\un) \vb{-F}.
    \end{equation}
    Since $a$ factors through $\quo{p} : \quo{\A^1_S} \to \quo{S}$, and $\quo{p}_!\,\quo{j}_* \cong 0$ (\propref{prop:j_*(1)}), \eqref{eq: connective (co)support 2} vanishes, as desired. \qed

\subsubsection{Proof of involutivity}

  At this point, combining \propref{prop:suppinvol} with \lemref{lem:cosuppinvol} yields the following form of involutivity:

  \begin{cor}[Involutivity]\label{cor:conninvol}
    Let $E \in \DVect(S)$ be of amplitude $\ge 0$.
    Then there is a canonical isomorphism
    \begin{equation*}
      (-)\vb{-\wh{E}} \to \FT_{\wh{E}} \FT_{E}(-).
    \end{equation*}
  \end{cor}
  \begin{proof}
    By \lemref{lem:cosuppinvol}, the cosupport property for $\wh{E}$ yields the canonical isomorphism
    \begin{equation*}
      (-) \otimes \pi_{E}^*(\sL^{E}) \to \FT_{\wh{E}} \FT_{E}(-).
    \end{equation*}
    By \propref{prop:suppinvol}\itemref{item:journalize}, we have the canonical isomorphism
    \begin{equation}
      \sL^{E}
      \cong \pi_{E!} 0_{E*}(\un)\vb{-\wh{E}}
      \cong \pi_{E!} 0_{E!}(\un)\vb{-\wh{E}}
      \cong \un_S\vb{-\wh{E}}.
    \end{equation}
  \end{proof}

\subsection{Cosupport and involutivity for general \texorpdfstring{$E$}{E}}
\label{sec:cosupp}

Let $E \in \DVect(S)$ be an arbitrary derived vector bundle.
We will now prove \propref{prop:main/cosupp}, which implies involutivity (\thmref{thm:invol}) by \lemref{lem:cosuppinvol}.
We will also show that the twist $\sL^E = \un\twbrace{E}$ is $\otimes$-invertible in general (\lemref{lem:L}).

\subsubsection{Proof of Proposition~\ref{prop:main/cosupp}}
\label{ssec:cosupp/proof}

  Let $E \in \DVect(S)$ and let us show that the canonical morphism
  \begin{equation}\label{eq:ideoplastia}
    \counit : 0_{\wh{E}!} 0_{\wh{E}}^! \FT_E(\un_E)
    \to \FT_E(\un_E)
  \end{equation}
  is invertible.
  Equivalently, $\FT_E(\un_E)$ lies in the essential image of the fully faithful functor $0_{\wh{E}!}$ (\corref{cor:chlorine}).

  Let $f : S' \to S$ be a smooth surjection and adopt the notation of \propref{prop:funbase}.
  Applying $f_{\wh{E}}^*$ on the left to the morphism \eqref{eq:ideoplastia} yields, by the base change formula $\Ex^*_! : f_{\wh{E}}^* 0_{\wh{E}!} \cong 0_{\wh{E'}!} f^*$, the exchange isomorphism $\Ex^{*!} : f^* 0_{\wh{E}}^! \cong 0_{\wh{E'}}^! f_{\wh{E}}^*$, and \eqref{eq:BC^*}, a canonical morphism
  \begin{equation}
    0_{\wh{E'}!} 0_{\wh{E'}}^! \FT_{E'/S'}(\un_{E'})
    \to \FT_{E'/S'}(\un_{E'})
  \end{equation}
  which one easily checks is identified with the counit.
  Thus the claim is local on $S$ and we may assume that $S$ is affine.
  In particular, by choosing a global presentation $\cE$ as in \S \ref{ssec: notate dvb}, we may write $E$ as the fibre of a morphism $d : E^- \to E^+$ where $E^-$ is of amplitude $\le 0$ and $E^+$ is of amplitude $\ge 0$.\footnote{namely, if $(\cdots \to \cE^{-1} \to \cE^0 \to \cE^{1} \to \cdots)$ is a global presentation for $\cE$, take $E^- = \Tot_S(\cE^{\leq 0})$ and $E^+ = \Tot_S(\cE^{\geq 1}[1])$ to be the brutal truncations}
  We have the Cartesian squares
  \begin{equation}\label{eq:nieceship}
    \begin{tikzcd}
      E \ar{r}{i}\ar{d}{\pi_E}
      & E^- \ar{d}{d}
      \\
      S \ar{r}{0_{E^+}}
      & E^+,
    \end{tikzcd}
    \qquad
    \begin{tikzcd}
      \wh{E^+} \ar{r}{\wh{d}}\ar{d}{\pi_{\wh{E^+}}}
      & \wh{E^-} \ar{d}{\wh{i}}
      \\
      S \ar{r}{0_{\wh{E}}}
      & \wh{E}
    \end{tikzcd}
  \end{equation}
  where $i$ is a closed immersion (since $E^+$ is of amplitude $\ge 0$) and its dual $\wh{i}$ is a smooth surjection (it is a torsor under $\wh{E^+}$).
  We may thus check \eqref{eq:ideoplastia} is invertible after applying $\wh{i}^*$ on the left; by base change and invertibility of $\Ex^{*!}$ this reduces to show that the morphism
  \begin{equation}\label{eq:Pellaea}
    \counit : \wh{d}_! \wh{d}^! \FT_{E^-} (i_! \un_E)
    \to \FT_{E^-} (i_! \un_E)
  \end{equation}
  is invertible.
  We claim that $\FT_{E^-} (i_! \un_E)$ is in the essential image of $\wh{d}_!$.
  Since $\unit : \un \to \wh{d}^! \wh{d}_!(\un)$ is invertible (\corref{cor:chlorine}), it will follow from adjunction identities that \eqref{eq:Pellaea} is invertible.

  Using $\Ex^{*,\FT}$, $\Ex^*_!$, and the computation of $\FT_{E^+}(\un)$ from \propref{prop:suppinvol} (which applies because $E^+$ is of amplitude $\ge 0$), we first compute:
  \begin{align}
    \FT_{\wh{E^-}} (\wh{d}_! \un_{\wh{E^+}})
    &\cong d^* \FT_{E^+} (\un_{E^+}) \cong d^* 0_{\wh{E^+}!}(\un_{\wh{E^+}}) \vb{-\wh{E^+}} \cong i_! (\un_E)\vb{-\wh{E^+}}.
  \end{align}
  Using involutivity for $\wh{E^-}$ (\corref{cor:conninvol}), we deduce the (canonical) isomorphism
  \begin{equation}\label{eq:hackneyer}
    \FT_{E^-}(i_! \un_E)
    \cong \FT_{E^-}\FT_{\wh{E^-}} (\wh{d}_! \un)\vb{\wh{E^+}}\\
    \cong \wh{d}_!(\un)\vb{\wh{E^+}}\vb{-E^-}.
  \end{equation}
  The claim follows.\qed 

\subsubsection{Proof of Lemma~\ref{lem:L}}
\label{ssec:cosupp/L}

  We prove that $\sL^E \coloneqq \pi_{E!} \FT_{\wh{E}}(\un)$ is $\otimes$-invertible.
  This is equivalent to the assertion that the evaluation morphism
  \begin{equation}\label{eq:shiny}
    \ev : \sL^E \otimes \uHom(\sL^E, \un_S) \to \un_S
  \end{equation}
  is invertible.
  If $f : S' \to S$ is a smooth morphism, then we have the canonical isomorphism
  \begin{equation*}
    f^*\uHom(\sL^E, \un_S)
    \cong \uHom(f^*\sL^E, f^*\un_{S'})
  \end{equation*}
  under which the $*$-pullback of \eqref{eq:shiny} is identified with
  \begin{equation*}
    f^*\sL^E \otimes \uHom(f^*\sL^E, f^*\un_{S'}) \to \un_{S'}.
  \end{equation*}
  Let $E' \coloneqq E \fibprod_S S'$ and let $f_{\wh{E}} : \wh{E'} \to \wh{E}$ be the base change of $f$.
  By \subsectionref{ssec:basechangetwist1} we have $f^*\sL^E \cong \sL^{E'}$.
  This shows that the question of invertibility of \eqref{eq:shiny} is local on $S$.
  In particular, we may assume that $E$ admits a presentation so that there are Cartesian squares
  \begin{equation}
    \begin{tikzcd}
      E \ar{r}{i}\ar{d}{\pi_E}
      & E^- \ar{d}{d}
      \\
      S \ar{r}{0_{E^+}}
      & E^+,
    \end{tikzcd}
    \qquad
    \begin{tikzcd}
      \wh{E^+} \ar{r}{\wh{d}}\ar{d}{\pi_{\wh{E^+}}}
      & \wh{E^-} \ar{d}{\wh{i}}
      \\
      S \ar{r}{0_{\wh{E}}}
      & \wh{E}
    \end{tikzcd}
  \end{equation}
  where $E^-$ is of amplitude $\leq 0$ and $E^+$ is of amplitude $\geq 0$, as above in \eqref{eq:nieceship}.
  By symmetry, we may as well show the claim for $\wh{E}$ instead, i.e., that $\sL^{\wh{E}} = \pi_{\wh{E}!} \FT_{E}(\un)$ is $\otimes$-invertible.
  After $*$-pullback along the smooth surjection $\pi_{\wh{E^+}} : \wh{E^+} \to S$, we have:
  \begin{align*}
    \pi_{\wh{E^+}}^* \sL^{\wh{E}}
    &\cong \pi_{\wh{E^+}}^* 0_{\wh{E}}^! \FT_{E}(\un)\tag{\corref{cor:contractder}}\\
    &\cong \wh{d}^! \wh{i}^* \FT_{E}(\un)\tag{$\Ex^{*!}$}\\
    &\cong \wh{d}^! \FT_{E^-}(i_!\un)\tag{\ref{eq:funA}}\\
    &\cong \wh{d}^! \wh{d}_! (\un) \vb{\wh{E^+}}\vb{-E^-}\tag{$\star$}\\
    &\cong \un\vb{\wh{E^+}}\vb{-E^-}\tag{\corref{cor:chlorine}},
  \end{align*}
   which is evidently $\otimes$-invertible. Here $(\star)$ is from the isomorphism $\FT_{E^-}(i_!\un) \cong \wh{d}_!(\un)\vb{\wh{E^+}}\vb{-E^-}$ \eqref{eq:hackneyer}.

\subsection{Base change and functoriality, part 2}
\label{sec:funct2}

Using involutivity, we conclude the proofs of Propositions~\ref{prop:funbase} and \ref{prop:fun}.
Specifically, we will use the fact that the functor
\begin{equation*}
  \FT_{\wh{E}}(\bullet)\twbrace{-E}
\end{equation*}
is inverse to $\FT_E$ (\corref{cor:inverse}).

\subsubsection{Proof of~\texorpdfstring{\eqref{eq:BC_*}}{BC\_*}}

  Let the notation be as in \propref{prop:funbase}.
  Recall the isomorphism
  \begin{equation*}
    f^*_E((-)\twbrace{E}) \cong f_E^*(-)\twbrace{E'}
  \end{equation*}
  of \lemref{lem:twistbase}, proven in \subsectionref{ssec:basechangetwist1}.
  Passing to right adjoints, we have:
  \begin{equation}\label{eq:horsepond}
    f_{E*}(-)\twbrace{-E} \cong f_{E*}((-)\twbrace{-E'})
  \end{equation}

  Recall also the $*$-base change natural isomorphism \eqref{eq:BC^*}
  \begin{equation*}
    f_{\wh{E}}^* \FT_E \cong \FT_{E'} f_E^*
  \end{equation*}
  proven in \S\ref{ssec:basechange1}.
  Passing to right adjoints yields
  \begin{equation*}
    f_{E*} (\FT_{\wh{E'}}(-)\twbrace{-E'})
    \cong \FT_{\wh{E}}(f_{\wh{E}*}(-))\twbrace{-E}.
  \end{equation*}
  Pulling out the twist using \eqref{eq:horsepond}, we deduce
  \begin{equation*}
    f_{E*} (\FT_{\wh{E'}}(-))
    \cong \FT_{\wh{E}}(f_{\wh{E}*}(-)).
  \end{equation*}
  Equivalently, replacing $E$ by $\wh{E}$ gives the natural isomorphism
  \begin{equation}
    f_{\wh{E}*} (\FT_{E'}(-))
    \cong \FT_{E}(f_{E*}(-)).
  \end{equation}

\subsubsection{Proof of \texorpdfstring{\eqref{eq:BC^!}}{(BC\^!)}}

  This follows from \eqref{eq:BC_!} exactly as above.

\subsubsection{Proof of~\texorpdfstring{\eqref{eq:funB}}{(Fun!)}}

  Passing to right adjoints from \eqref{eq:funA}
  \begin{equation*}
    \Ex^{*,\FT} : \wh{\phi}^* \circ {\FT_{E'}} \xrightarrow{\sim} {\FT_{E}} \circ \phi_!
  \end{equation*}
  yields the canonical isomorphism
  \begin{equation}\label{eq:lanced}
    \phi^! \circ {\FT_{\wh{E}}\twbrace{-\wh{E}}} \xrightarrow{\sim} {\FT_{\wh{E'}}\twbrace{-\wh{E'}}} \circ \wh{\phi}_*.
  \end{equation}
  Equivalently, applying this to $\wh{\phi}$ in place of $\phi$ we get the canonical isomorphism
  \begin{equation}
    \Ex^{!,\FT} : \wh{\phi}^{!} \circ {\FT_{E'}\twbrace{-E'}} \xrightarrow{\sim} {\FT_{E}\twbrace{-E}} \circ \phi_*.
  \end{equation}
  \qed 

\subsubsection{Proof of~\texorpdfstring{\eqref{eq:funA'}}{(Fun*)}}

  Begin with the isomorphism \eqref{eq:lanced} above.
  Applying $\FT_{E'}$ on the left and $\FT_{E}$ on the right, then applying the natural isomorphism of Theorem \ref{thm:invol} to $\FT_{E'} \circ \FT_{\wh{E'}}$ and $(-1)^{\rank E}$ times the natural isomorphism of Theorem \ref{thm:invol} to $\FT_{\wh{E}} \circ \FT_E$ (cf. \cite[\S A.2.6]{FYZ3}), and then untwisting, we obtain the natural isomorphism
  \begin{equation*}
    \Ex^{\FT,!} : {\FT_{E'}} \circ \phi^!
    \xrightarrow{\sim} \wh{\phi}_* \circ {\FT_{E}}.
  \end{equation*}
  
  \qed 

\subsubsection{Proof of~\texorpdfstring{\eqref{eq:funB'}}{(Fun!)}}

  Begin with the natural isomorphism \eqref{eq:funA}:
  \begin{equation*}
    \Ex^{*,\FT} : \wh{\phi}^* \circ {\FT_{E'}} \xrightarrow{\sim} {\FT_{E}} \circ \phi_!.
  \end{equation*}
  Applying $\FT_{E'}\twbrace{-E}$ on the left and $\FT_{E}\twbrace{-E}$ on the right then applying the natural isomorphism of Theorem \ref{thm:invol} to $\FT_{E'} \circ \FT_{\wh{E'}}$ and $(-1)^{\rank E}$ times the natural isomorphism of Theorem \ref{thm:invol} to $\FT_{\wh{E}} \circ \FT_E$, and then untwisting, we obtain the natural isomorphism:
  \begin{equation*}
    \Ex^{\FT,*} : {\FT_{E'}\twbrace{-E'}} \circ \phi^*
    \xrightarrow{\sim} \wh{\phi}_! \circ {\FT_{E}\twbrace{-E}}.
  \end{equation*}
  \qed 

\subsubsection{Proof of Lemma~\ref{lem:twistbase}}
\label{ssec:basechangetwist2}

  Let $f : S' \to S$ be a morphism and adopt the notation of \propref{prop:funbase}.
  We prove the remaining isomorphisms of \lemref{lem:twistbase}.
  We already have the isomorphisms
  \begin{equation}\label{eq:cankerwort}
    f^*_E((-)\twbrace{E}) \cong f_E^*(-)\twbrace{E'},
    \qquad
    f_{E!}(-)\twbrace{E} \cong f_{E!}((-)\twbrace{E'})
  \end{equation}
  by \subsectionref{ssec:basechangetwist1}.
  
  \subsubsection{$*$-Push}
  We wish to produce a natural isomorphism
  \begin{equation*}
    f_{E*}(-)\twbrace{E} \cong f_{E*}((-)\twbrace{E'})
  \end{equation*}
  which is equivalent by passage to left adjoints to a natural isomoprhism
  \begin{equation*}
    f_E^*((-)\twbrace{-E}) \cong f_E^*(-)\twbrace{-E'}.
  \end{equation*}
  Twisting by $\twbrace{E'}$ on both sides, this is equivalent to a natural isomorphism 
  \begin{equation*}
f_E^*((-)\twbrace{-E})\twbrace{E'}    \cong f_E^*(-)  ,
  \end{equation*}
  which we have from \eqref{eq:cankerwort}.\qed

  \subsubsection{$!$-Pull}
  This follows from the $!$-push version by the same argument as above. \qed

\section{Fourier analysis for motives}\label{sec: Fourier for motives}

We now specialize the discussion again to the weave $\Dmot{-}$.
We will refer to the homogeneous Fourier transform in this context as the ``motivic homogeneous Fourier transform''. In this section we analyze the interaction between the motivic Fourier transform and cohomological correspondences, and the motivic sheaf-cycle correspondence. In \S \ref{ssec: coco} we recall the notion of cohomological co-correspondences. \S \ref{ssec: renormalized FT} and \S \ref{ssec: global presentation} play a small technical role. In \S \ref{sssec: FT of cc} and \S\ref{sssec: FT functoriality of cc} we study the motivic Fourier transform of cohomological correspondences and their compatibility with pushforwards and pullbacks. Finally in \S \ref{sec: arithmetic FT} we define a homogeneous version of the \emph{arithmetic Fourier transform} from \cite[\S 9]{FYZ3}, and establish its compatibility with the motivic Fourier transform under the sheaf-cycle correspondence. 

\subsection{The renormalized homogeneous Fourier transform}\label{ssec: renormalized FT}

\sssec{}

Let $S$ be a derived Artin stack locally of finite type over a field.
For a derived vector bundle $E \rightarrow S$, the \emph{homogeneous motivic Fourier transform} is the functor
\begin{equation*}
  \FT_E : \Dmot{\quo{E}} \to \Dmot{\quo{\wh{E}}}
\end{equation*}
given by the construction of \S \ref{sssec:FT} in the case of the weave $\Dmot{-}$.

\sssec{}

The \emph{renormalized} homogeneous motivic Fourier transform 
\[
\rFT_E \co \Dmot{\quo{E}} \rightarrow \Dmot{\quo{\wh{E}}}
\]
is defined as
\begin{equation*}
  \rFT_E(-) \coloneqq \FT_E(-)[\rk(E)].
\end{equation*}
As all sheaf-theoretic operations are compatible with shifts, the results of \S \ref{sec: derived homogeneous FT} all have straightforward reformulations in terms of $\rFT$. 

\begin{rem}
The renormalization is designed to match the conventions for the $\ell$-adic Fourier transform in \cite{Lau87, FYZ3}.
Indeed, in the $\ell$-adic context the ``average'' of the Artin--Schreier sheaf $\cL_\psi$ is isomorphic to $\quo{j}_* (\Q)[-1]$ where $\quo{j} \co \quo{\G_m} \rightarrow \quo{\A^1}$ by \cite[Lemme 2.3]{Laumon} (notation as in \ref{notat: quo}).
Hence $\rFT_E$ is compatible with the derived $\ell$-adic Fourier transform of \cite{FYZ3} under $\ell$-adic realization.
On the other hand, one can show that $\FT_E$ (non-renormalized) is compatible with the Fourier--Sato transform of \cite{KashiwaraSchapira,dimredcoha} under Betti realization (see \cite{fouriercomp}).
\end{rem}

\begin{rem}
The renormalization is arguably less natural from the perspective of weights, but more natural from the perspective of perverse sheaves.
\end{rem}

\subsection{Further functoriality properties}\label{ssec: global presentation} 

We formulate some functoriality results that we can prove under globally presented assumptions (which are probably unnecessary for the statements to be true).
We recall the notion of global presentations of derived vector bundles from \S \ref{ssec: notate dvb}.

\sssec{Base change}

Consider a Cartesian square of globally presented derived vector bundles over $S$, along with the dual Cartesian square:
\[
\begin{tikzcd}
& B \ar[dl, "g'"'] \ar[dr, "f'"] \\
A  \ar[dr, "f"'] & & D  \ar[dl, "g"] \\
& C
\end{tikzcd}\hspace{1cm} \text{and} \hspace{1cm}
\begin{tikzcd}
& \wh{C} \ar[dl, "\wh{f}"'] \ar[dr, "\wh{g}"] \\
\wh{A}  \ar[dr, "\wh{g}'"'] & & \wh{D}  \ar[dl, "\wh{f}'"] \\
& \wh{B}
\end{tikzcd}
\]
Then the base change formula gives natural isomorphisms 
\begin{equation}\label{eq: app BC}
g^* f_! \cong (f')_! (g')^*  \quad \text{and} \quad \wh{g}_! \wh{f}^* \cong (\wh{f}')^* \wh{g}'_!
\end{equation}
Let $d \coloneqq d(f)$, $\delta \coloneqq d(g)$. According to \S \ref{ssec:main/fun}, there are natural isomorphisms 
\begin{equation}\label{eq: app FT}
\wh{g}_! \wh{f}^*\rFT_A \cong  \rFT_D g^* f_! [d+\delta](\delta)  \quad \text{and} \quad  (\wh{f}')^* \wh{g}'_! \rFT_A \cong \rFT_D f'_! (g')^* [d+\delta](\delta).
\end{equation}

\begin{prop}\label{prop: ft pbc} 
Assume that $f$ and $g$ are globally presented (in particular, $A,C,D$ are globally presented). Then there is a commutative diagram of functors $\Dmot{A} \rightarrow \Dmot{\wh{D}}$
\begin{equation}\label{eq: FT BC compatibility}
\begin{tikzcd}
\wh{g}_! \wh{f}^*\rFT_A  \ar[r, "\sim"]   \ar[d, "\sim"]  &   \rFT_D g^* f_! [d+\delta](\delta)  \ar[d, "\sim"]  \\
(\wh{f}')^* \wh{g}'_! \rFT_A \ar[r, "\sim"]   &  \rFT_D f'_! (g')^* [d+\delta](\delta) 
\end{tikzcd}
\end{equation}
where the identifications are as in \eqref{eq: app BC} and \eqref{eq: app FT}.
\end{prop}

\begin{proof}
The proof of \cite[Proposition 6.6.3]{FYZ3} works verbatim. 
\end{proof}

\sssec{Gysin vs. forgetting supports}

Recall from \cite[\S 6.4]{FYZ3} that a map $f \co E' \rightarrow E$ of derived vector bundles over $S$ is quasi-smooth if and only if the dual map $\wh{f} \co \wh{E} \rightarrow \wh{E'}$ is separated.
In this case, we have a Gysin natural transformation (see \S \ref{ssec:Gys})
\[ \mrm{gys}_f \co f^* \rightarrow f^! \tw{-d(f)} \]
and a ``forget supports'' natural transformation
\[ \mrm{fsupp}_{\wh{f}} \co \wh{f}_! \rightarrow \wh{f}_*. \]

\begin{prop}\label{prop: fourier of Gysin}
Let $f \co E' \rightarrow E$ be a globally presented quasi-smooth map of derived vector bundles and let $\wh{f} \co \wh{E} \rightarrow \wh{E}'$ be the dual map to $f \co E' \rightarrow E$. Then there is a commutative diagram of functors $\Dmot{\quo{E}} \rightarrow \Dmot{\quo{\wh{E}'}}$
\begin{equation}\label{eq: fourier of gysin}
\begin{tikzcd}
\wh{f}_! \rFT_{E} \ar[d, "\sim"] \ar[r, "\on{fsupp}_{\wh{f}}"] & \wh{f}_* \rFT_{E} \ar[d, "\sim"] \\
\rFT_{E'} f^* [d(f)](d(f)) \ar[r, "{\mrm{gys}_f}"] & \rFT_{E'} f^! [-d(f)]
\end{tikzcd}
\end{equation}
\end{prop}

\begin{proof}
The proof of \cite[Proposition 6.4.2]{FYZ3} works verbatim. 
\end{proof}

\subsection{Cohomological co-correspondences}\label{ssec: coco}
A \emph{co-correspondence} between derived Artin stacks $A_{0}$ and $A_{1}$ is a diagram
\begin{equation}
\begin{tikzcd}
A_0 \ar[r, "c'_{1}"] & C' & A_{1} \ar[l, "c'_{0}"']
\end{tikzcd}
\end{equation}
We define a \emph{cohomological co-correspondence} from $\cK_0 \in \Dmot{A_0}$ to $\cK_1 \in \Dmot{A_1}$ to be an element of $\Hom_{C'} (c'_{1!} \cK_0, c'_{0*} \cK_1)$. Let
\begin{equation}
\CoCorr_{C'}(\cK_{0}, \cK_{1})\coloneqq\Hom_{C'} (c'_{1!} \cK_0, c'_{0*} \cK_1).
\end{equation}

\subsubsection{Correspondences versus co-correspondences}\label{sssec: corr vs cocorr} To see the relation between cohomological correspondences and co-correspondences, suppose we have a Cartesian square
\begin{equation}\label{eq: corr and co-corr diagram}
\begin{tikzcd}
& C^{\flat} \ar[dl, "c_0"'] \ar[dr, "c_1"] \\
A_0 \ar[dr, "c_1'"']  & & A_1 \ar[dl, "c_0'"] \\
& C^{\sharp}
\end{tikzcd}
\end{equation}
Then for $\cK_0 \in \Dmot{A_0}$ and $\cK_1 \in \Dmot{A_1}$, there is a canonical isomorphism of vector spaces
\begin{equation}\label{eq: coh corr to coh co-corr}
\g_{C}: \Corr_{C^{\flat}}(\cK_{0}, \cK_{1})\isom \CoCorr_{C^{\sharp}}(\cK_{0}, \cK_{1})
\end{equation}
given by the composition below, where the isomorphisms come from adjunctions and proper base change:
\begin{equation*}
\Hom_{C^{\flat}}(c_0^* \cK_0 , c_1^! \cK_1) \cong \Hom_{A_1}(c_{1!} c_0^* \cK_0 , \cK_1) \cong  \Hom_{A_1}((c_0')^* c'_{1!}  \cK_0 , \cK_1)  \cong  \Hom_{C^{\sharp}} (c'_{1!} \cK_0, c'_{0*} \cK_1).
\end{equation*}
Note that a correspondence of \emph{derived vector bundles} (over some base $S$) can always be completed to a Cartesian square of the form \eqref{eq: corr and co-corr diagram} by taking $C^\sharp$ to be the pushout in the \inftyCat of derived vector bundles over $B$.

\subsection{Fourier transform of cohomological correspondences}\label{ssec: FT of cc}

In \S \ref{ssec: coco} we defined the notion of cohomological co-correspondence. These arise naturally as the Fourier transforms of cohomological correspondences, as we now explain. 
 
\sssec{}\label{sssec: FT of cc}

Suppose we are given a map of correspondences of derived Artin stacks
\begin{equation}
\xymatrix{
E_{0}\ar[d] & C^{\flat}\ar[d] \ar[l]_{p_{0}}\ar[r]^{p_{1}} & E_{1}\ar[d]\\
S_{0} & C_{S}\ar[l]_{h_{0}}\ar[r]^{h_{1}} & S_{1}} 
\end{equation}
where $E_0, C^{\flat}$ and $E_1$ are derived vector bundles on $S_0, C_{S}$ and $S_1$ respectively. Assume the maps $p_{0}$ and $p_{1}$ are linear. 

Let $\wt{E}_0$ and  $\wt{E}_1$ be the pullbacks of $E_{0}$ and $E_{1}$ to $C_S$ via $h_{i}$. We can canonically extend the correspondence $E_{0}\xleftarrow{p_{0}}C^{\flat}\xr{p_{1}} E_{1}$ to a commutative diagram
 \begin{equation}\label{eq: FT cc 1}
 \begin{tikzcd}
& &  C^{\flat}  \ar[dl, "\wt{p}_0"]   \ar[dr, "\wt{p}_1"']  \ar[ddll, bend right, "p_0"'] \ar[ddrr, bend left, "p_1"] \\ 
& \wt{E}_0 \ar[dr, "\wt p'_1"'] \ar[dl, "h_0^E"] && \wt{E}_1 \ar[dl, "\wt p'_0"]  \ar[dr, "h_1^E"'] \\ 
 E_0 & & C^{\sharp}  & & E_1
 \end{tikzcd}
 \end{equation}
by defining $C^{\sh}$ to be the pushout of the diagram $\wt E_{0}\xleftarrow{\wt p_{0}}C^{\flat}\xr{\wt p_{1}} \wt E_{1}$, taken in the \inftyCat of derived vector bundles over $C_{S}$, so that the inner diamond is (derived) Cartesian.
 
Dualizing \eqref{eq: FT cc 1}, we get a commutative diagram
\begin{equation}
 \begin{tikzcd}
& &  \wh {C^{\sh}} \ar[dl, "\wh{\wt p'_1}"]   \ar[dr, "\wh{\wt p'_0}"']  \ar[ddll, bend right, "\wh{p'_1}"'] \ar[ddrr, bend left, "\wh{p'_0}"]\\ 
& \wh{\wt{E}}_0 \ar[dr, "\wh{\wt{p}}_0"'] \ar[dl, "h_0^{\wh E}"] && \wh{\wt{E}}_1 \ar[dl, "\wh{\wt{p}}_1"]  \ar[dr, "h_1^{\wh E}"'] \\ 
 \wh{E}_0 & & \wh{C^{\flat}}  & & \wh{E}_1
 \end{tikzcd}
\end{equation}
where the inner diamond is Cartesian.

Given $\cK_{i}\in \Dmot{\quo{E_{i}}}$ for each $i \in \{0,1\}$, we define an isomorphism of vector spaces
\begin{equation}\label{eq: FT of corr}
\rFT_{C^{\flat}}: \Corr_{\quo{C^{\flat}}}(\cK_{0}, \cK_{1})\isom \Corr_{\quo{\wh{C}^{\sharp}}}(\rFT_{E_{0}}(\cK_{0}), \rFT_{E_{1}}(\cK_{1})\sm{[d(\wt p_0) + d(\wt p_1)](d(\wt p_0))})
\end{equation}
as the composite of the isomorphisms
\begin{eqnarray*}
\Corr_{\quo{C^{\flat}}}(\cK_{0}, \cK_{1})&=&\Corr_{\quo{ C^{\flat}}}(( h_0^E)^* \cK_0,(h_1^E)^! \cK_1 )\\
\text{(\S\ref{ssec:main/fun}) } \implies &\xr{\rFT_{C^{\flat}}}& \CoCorr_{\quo{\wh{C^{\flat}}}}(\rFT_{\wt E_{0}}(( h_0^E)^* \cK_0),\rFT_{\wt E_{1}}((h_1^E)^! \cK_1 )\sm{[d(\wt p_0) + d(\wt p_1)](d(\wt p_0))})\\
\text{(\S\ref{sssec: corr vs cocorr}) } \implies &\cong& \Corr_{\quo{\wh{C^{\sh}}}}(\rFT_{\wt E_{0}}(( h_0^E)^* \cK_0),\rFT_{\wt E_{1}}((h_1^E)^! \cK_1 )\sm{[d(\wt p_0) + d(\wt p_1)](d(\wt p_0))})\\
\text{(\S\ref{ssec:main/funbase})} \implies &\cong & \Corr_{\quo{\wh{C^{\sh}}}}((h_0^{\wh E})^* \rFT_{E_0}(\cK_0), (h_1^{\wh E})^! \rFT_{E_1}(\cK_1)\sm{[d(\wt p_0) + d(\wt p_1)](d(\wt p_0))})\\
&=& \Corr_{\quo{\wh{C^{\sharp}}}}(\rFT_{E_{0}}(\cK_{0}), \rFT_{E_{1}}(\cK_{1})\sm{[d(\wt p_0) + d(\wt p_1)](d(\wt p_0))}).
\end{eqnarray*}

\subsubsection{Functoriality}\label{sssec: FT functoriality of cc}
We state and prove functorial properties of the Fourier transform of cohomological correspondences \eqref{eq: FT of corr}. Consider a diagram of maps of correspondences of derived Artin stacks
\begin{equation}\label{linear corr varying base}
\xymatrix{ 
E_{0}\ar[d]_{f_{0}}  & C^{\flat}\ar[l]_{p_{0}}\ar[r]^{p_{1}}\ar[d]^{f^{\flat}}  & E_{1}\ar[d]^{f_{1}}\\
F_{0}\ar[d]  & D^{\flat}\ar[l]_{q_{0}}\ar[r]^{q_{1}}\ar[d]^{} & F_{1}\ar[d]\\
S_{0} & C_{S}\ar[l]_{h_{0}}\ar[r]^{h_{1}} & S_{1} 
}
\end{equation}
where $E_{i}$ and $F_{i}$ are derived vector bundles over $S_{i}$ (for $i=0,1$),  and $C^{\flat}$ and $D^{\flat}$ are derived vector bundles over $C_{S}$. All maps between derived vector bundles are assumed to be linear.

Let $\wt{E}_i \rightarrow C_S$, $\wt{F}_i \rightarrow C_S$ and $\wt f_{i}: \wt{E}_i \rightarrow \wt{F}_{i}$ be the base changes of $E_{i}, F_{i}$ and $f_{i}$ along $h_i \co C_S \rightarrow S_i$. Using the discussion in \S\ref{sssec: FT of cc}, we can canonically extend the upper part of the diagram \eqref{linear corr varying base} to a commutative diagram 
\begin{equation}\label{eq: FT cc functoriality diagram varying base}
 \begin{tikzcd}
& &  C^{\flat}  \ar[dl, "\wt{p}_0"] \ar[dd, phantom, "\di"]  \ar[dr, "\wt{p}_1"']  \ar[ddll, bend right, "p_0"']   \ar[ddrr, bend left, "p_1"] \ar[ddd, "f^{\flat}"', bend left] \\ 
& \ar[dl, "h_0^E"]  \wt{E}_0 \ar[dr, "p'_{1}"'] \ar[ddd, "\wt{f}_0"] && \wt{E}_1 \ar[dl, "p'_0"]   \ar[ddd, "\wt{f}_1"]  \ar[dr, "h_1^E"'] \\ 
E_0 \ar[ddr, phantom, "\di"]\ar[ddd, "f_0"]  &  & \wt C^{\sharp}  \ar[ddd, "\wt{f}^{\sh}"', bend right] & & E_1 \ar[ddd, "f_1"]\ar[ddl, phantom, "\di"] \\
& &  D^{\flat}  \ar[dd, phantom, "\di"]\ar[dl, "\wt q_0"]   \ar[dr, "\wt q_{1}"']  \ar[ddll, bend right, "q_0"']  \ar[ddrr, bend left, "q_1"] \\ 
&\ar[dl, "h_0^{F}"]   \wt{F}_0 \ar[dr, "\wt q'_1"'] && \wt{F}_1 \ar[dl, "\wt q'_0"] \ar[dr, "h_1^{F}"']    \\ 
F_0  & &  \wt{D}^{\sharp} & & F_1
 \end{tikzcd}
\end{equation}
where the squares labeled by $\di$ are derived Cartesian.

Since the leftmost parallelogram is derived Cartesian, the square $(C^{\flat}, E_{0}, D^{\flat}, F_{0})$ is pushable if and only if the square $( C^{\flat}, \wt E_{0}, D^{\flat}, \wt F_{0})$ is pushable. When any of these equivalent conditions holds, we have a pushforward map (for $\cK_i \in \Dmot{\quo{E_i}}$),
\begin{equation}
f^{\flat}_{!}: \Corr_{\quo{C^{\flat}}}(\cK_{0}, \cK_{1})\to \Corr_{\quo{D^{\flat}}}(f_{0!}\cK_{0}, f_{1!}\cK_{1}).
\end{equation}

 The dual diagram to \eqref{eq: FT cc functoriality diagram varying base} is: 
\begin{equation}\label{eq: FT cc functoriality diagram varying base dual}
\begin{tikzcd}
 & & \wh{D^{\sharp}}  \ar[dl, "\wh{\wt q'_1}"]   \ar[dr, "\wh{\wt q'_0}"']    \ar[ddd, "\wh{f}^{\sh}"', bend left] \ar[ddll, bend right, "\wh{q'_1}"']   \ar[ddrr, bend left, "\wh{q'_0}"]\\ 
&\ar[dl, "h_0^{\wh F}"]  \wh{\wt{F}}_0 \ar[dr, "\wh{\wt{q}_0}"'] \ar[ddd, "\wh{\wt{f}}_0"]  && \wh{\wt{F}}_1 \ar[dl, "\wh{\wt{q}}_1"]   \ar[ddd, "\wh{\wt{f}}_1"]  \ar[dr, "h_1^{\wh F}"']  \\ 
\wh{F}_0  \ar[ddd, "\wh{f}_0"] &  & \wh{\wt{D}^{\flat}}  \ar[ddd, "\wh{\wt f}^{\flat}"', bend right]    & & \wh{F}_1  \ar[ddd, "\wh{f}_1"] \\
 & & \wh{C^{\sh}}  \ar[dl, "\wh{\wt p'_1}"]   \ar[dr, "\wh{\wt p'_0}"'] \ar[ddll, bend right, "\wh{p'_1}"']   \ar[ddrr, bend left, "\wh{p'_0}"]\\ 
& \wh{\wt{E}}_0 \ar[dr, "\wh{\wt{p}}_0"']  \ar[dl, "h_0^{\wh E}"]  && \wh{\wt{E}}_1 \ar[dr, "h_1^{\wh E}"']   \ar[dl, "\wh{\wt{p}}_1"]  \\ 
\wh{E}_0 &  & \wh{\wt{C}^{\flat} }  && \wh{E}_1
 \end{tikzcd}
\end{equation}
Since the rightmost parallelogram is derived Cartesian, the square $(\wh{D^{\sh}}, \wh F_{1}, \wh{C^{\sh}}, \wh E_{1})$ is pullable if and only if the square $(\wh{ D^{\sh}}, \wh{\wt F_{1}}, \wh{ C^{\sh}}, \wh{\wt E_{1}})$ is pullable. When any of these equivalent conditions holds, we have a pullback map (for $\cK_i \in \Dmot{\quo{\wh{E_i}}}$),
\begin{equation}
(\wh{f^{\sh}})^{*}: \Corr_{\quo{\wh{C^{\sh}}}}(\cK_{0}, \cK_{1})\to \Corr_{\quo{\wh{D^{\sh}}}}(\wh{f}_{0}^{*}\cK_{0}, \wh{f}_{1}^{*}\cK_{1}\tw{-\d_{\wh{ f^{\sh}}}}).
\end{equation}
Moreover, by \cite[Lemma 7.2.1]{FYZ3}, $f^{\flat}$ is left pushable if and only if $\wh{f^{\sh}}$ is right pullable.

\begin{prop}\label{prop: cc push/pull functoriality}  Assume the diagram \eqref{eq: FT cc functoriality diagram varying base dual} is globally presented. 

(1) Suppose the map of correspondences $f^{\flat}: C^{\flat}\to D^{\flat}$ is left pushable. Let $\cK_i \in \Dmot{\quo{E_i}}$ for $i=0,1$. Then the following diagram commutes:
\begin{equation}\label{FT cc push pull diagram}
\xymatrix{\Corr_{\quo{C^{\flat}}}(\cK_{0}, \cK_{1}) \ar[rr]^-{\rFT_{C^{\flat}}}\ar[d]^{f^{\flat}_{!}}& &  \Corr_{\quo{\wh{C^{\sh}}}}(\rFT_{E_{0}}(\cK_0), \rFT_{E_{1}}(\cK_1)\sm{[d(\wt p_0)+d(\wt p_1)](d(\wt p_0))})\ar[d]^{(\wh{f}^{\sh})^{*}}\\
\Corr_{\quo{D^{\flat}}}(f_{0!}\cK_{0}, f_{1!}\cK_{1})\ar[rr]^-{\TT_{[d(f_0)]}\rFT_{D^{\flat}}} & & \Corr_{\quo{\wh{D^{\sh}}}}(\wh{f}_{0}^{*}\rFT_{E_{0}}(\cK_0), \wh{f}_{1}^{*}\rFT_{E_{1}}(\cK_1)\sm{[d(\wt q_0)+d(\wt q_1)+d(f_{0})-d(f_{1})](d(\wt q_0))})
}
\end{equation}
Here we use \cite[Lemma 7.2.2]{FYZ3} to match the differences of the twists that appear in the right vertical map.

(2) Suppose the map of correspondences $f^{\flat}: C^{\flat}\to D^{\flat}$ is right pullable. Let $\cK_i \in \Dmot{\quo{F_i}}$ for $i=0,1$. Then the following diagram commutes
\[
\begin{tikzcd}[column sep = huge]
\Corr_{\quo{D^{\flat}}}(\cK_{0}, \cK_{1}) \ar[rr, "\rFT_{D^{\flat}}"] \ar[d, "(f^{\flat})^*"] & & {\Corr_{\quo{\wh{D^{\sh}}}}(\rFT_{F_{0}}(\cK_0), \rFT_{F_{1}}(\cK_1)\sm{[d(\wt{q}_0)+d(\wt{q}_1)](d(\wt{q}_0))})} \ar[d, "(\wh{f}^{\sh})_{!}"] \\
\Corr_{\quo{C^{\flat}}}(f_{0}^*\cK_{0}, f_{1}^*\cK_{1})   \ar[rr, "\TT_{[d(f_0)](d(f_0))} \rFT_{C^{\flat}}"] & &  {\Corr_{\quo{\wh{C^{\sh}}}}(\wh{f}_{0!}\rFT_{F_{0}}(\cK_0), \wh{f}_{1!}\rFT_{F_{1}}(\cK_1)\sm{[d(\wt{q}_0)+d(\wt{q}_1)](d(\wt{q}_0))})} 
\end{tikzcd}
\]
\end{prop}

\begin{proof} The proof of \cite[Proposition 7.2.4]{FYZ3} works verbatim. 
\end{proof}

\subsection{Homogeneous arithmetic Fourier transform}\label{sec: arithmetic FT} In this section we lift (a homogeneous variant of) the \emph{arithmetic Fourier transform} of \cite[\S 8]{FYZ3} from $\ell$-adic Borel--Moore homology to Chow groups.

\subsubsection{\'{E}tale $\F_q$-vector space bundles}\label{sssec: AFT setup}

Let $T$ be a derived Artin stack locally of finite type over a field.
Let $V \rightarrow T$ be an \'{e}tale locally free $\F_q$-vector space bundle of rank $d$ (thus, the datum of $V$ is equivalent to that of an \'{e}tale $\GL_d(\F_q^\times)$-torsor). 

Define $\quo{V}$ to be the stack quotient $[V/\ul{\F}_q^\times]$, where $\ul{\F}_q^\times$ is the discrete group scheme over $T$ with value $\F_q^\times$. 
Let $\wh{V} \rightarrow T$ be the dual $\F_q$-vector space, i.e., at the level of \'{e}tale sheaves over $T$ we have 
\[
\wh{V} \coloneqq \cHom_{\F_q}(V, \ul{\F}_q). 
\]
Note that $\wh{\wh{V}} \cong V$. We have an ``evaluation'' map 
\[
\begin{tikzcd}
\ev \co \quo{V} \times_T \quo{\wh{V}} \rightarrow [\F_q/ \F_q^\times].
\end{tikzcd}
\]

\sssec{Homogeneous arithmetic Fourier transform}

Let $\xi$ be the function on $\F_q$ defined as 
\[
\xi (x) = \begin{cases} q-1 & x = 0, \\ -1 & x \neq 0. \end{cases}
\]
As $\xi$ is invariant under the scaling action of $\F_q^\times$, $\xi$ descends to a function on the quotient stack $[\F_q/\F_q^\times]$. Its significance is the following: if $\psi$ is any nontrivial additive character of $\ul{\F}_q$, then the function on $[\F_q/\F_q^\times]$ obtained by averaging $\psi$ over $\F_q^\times$ is $\xi$. 

Now consider the diagram 
\[
\begin{tikzcd}
& \quo{V} \times_T \quo{\wh{V}}   \ar[dl, "\pr_0"'] \ar[dr, "\pr_1"]  \ar[r, "\ev"] & {[\F_q/\F_q^\times]}\\
\quo{V} \ar[dr, "\pi"']  & & \quo{\wh{V}} \ar[dl, "\wh{\pi}"] \\
& T
\end{tikzcd}
\]

We define the \emph{homogeneous arithmetic Fourier transform} to be the map 
\[
\aFT_V \co \CH_*(\quo{V}) \rightarrow \CH_{*} (\quo{\wh{V}})
\]
given by
\[
\alpha \mapsto   (-1)^d \pr_{1*}(\pr_0^!(\alpha) \cap \ev^*\xi),
\]
where $d$ is the rank of $V$ as an $\F_q$-vector space over $T$.
Here, we regard the locally constant function $\xi$ as an element of $\CH^0([\F_q/\F_q^\times])$.
We also used the fact that the projections $\pr_i$ are finite étale so that the pushforward and pullback maps are defined.

Similarly, we consider a variant in Chow cohomology (\S \ref{ssec: Chow}):
\[
\aFT \co \CH^*(\quo{V}; \Q) \rightarrow \CH^* (\quo{\wh{V}}; \Q)
\]
given by 
\[
\alpha \mapsto   (-1)^d\pr_{1!}(\pr_0^*(\alpha) \cup \ev^*\xi).
\]

\subsubsection{Basic properties} The following properties of the homogeneous arithmetic Fourier transform follow by the same arguments as in the non-homogeneous case (which are found in \cite[\S 8.2]{FYZ3}).

\begin{notat}\label{notat: pairing} For $\alpha \in \CH_{i}(\quo{V})$ and $\beta \in \CH^{j}(\quo{V})$, we write 
\[
\langle \alpha, \beta \rangle \coloneqq \pi_* (\alpha \cap \beta) \in \CH_{i-j}(T)
\]
where $\pi \co \quo{V} \rightarrow T$ is the projection. 
\end{notat}

\begin{lemma}[Plancherel property]\label{lem: Plancherel for cycles}
Let $\alpha_1 \in \CH_{i}(\quo{V})$ and $\beta_2 \in \CH^{j}(\quo{\wh{V}})$. Then
\[
\langle \alpha_1 , \aFT(\beta_2)\rangle = \langle \aFT(\alpha_1) , \beta_2 \rangle \in \CH_{i-j}(T).
\]
\end{lemma}

\begin{lemma}[Involutivity]\label{lem: AFT involutivity}
We have $\aFT_{\wh{V}} \circ  \aFT_{V} = q^d$ where $d$ is the rank of $V$. 
\end{lemma}

\ssec{Compatibility of motivic and arithmetic Fourier transforms}\label{sssec: AFT and FT} We establish the compatibility of the motivic Fourier transform with arithmetic Fourier transform under the motivic sheaf-cycle correspondence.

\sssec{Setup}

Let $Y$ be a derived Artin stack locally of finite type over a field.
Let $p \co E\to Y$ be a vector bundle\footnote{i.e., a derived vector bundle of amplitude $[0,0]$.}. Suppose $c=(c_{0},c_{1}) \co C\to Y\times Y$ is a correspondence of derived Artin stacks and we are given an isomorphism of vector bundles over $C$
\begin{equation}
\io \co c_{0}^{*}E\cong c_{1}^{*}E.
\end{equation}
Let $C_{E}$ be the total space of $c_{0}^{*}E \stackrel{\iota}\cong c_{1}^{*}E$. For $i \in \{0,1\}$ we let $e_{i} \co C_{E}\cong c_{i}^{*}E\to E$ be the corresponding projection map. Then we have a map of correspondences
\begin{equation}\label{corr vb}
\xymatrix{E\ar[d]^{\pr} &  C_{E}\ar[d]^{\pr_{C}}\ar[l]_-{e_{0}}\ar[r]^{e_{1}} & E\ar[d]^{\pr}\\
Y & C\ar[l]_-{c_{0}}\ar[r]^{c_{1}} & Y }
\end{equation}
such that both squares are Cartesian.

The above data induces a correspondence $\wh e \co C_{\wh E}\to \wh E\times \wh E$ by passing to the dual vector bundles.

Consider the Frobenius twisted correspondence map $c^{(1)}=(\Frob \c c_{0}, c_{1}) \co C^{(1)} \rightarrow Y\times Y$. Similarly, we define $C^{(1)}_{E}$ (a self-correspondence of $E$) and $C^{(1)}_{\wh E}$ (a self-correspondence of $\wh E$). Recall notation (\S \ref{sssec: fix vs sht})
\begin{equation}
\Sht(C)\coloneqq\Fix(C^{(1)}), \quad 
\Sht(C_{E})\coloneqq\Fix(C^{(1)}_{E}), \quad \Sht(C_{\wh E})\coloneqq\Fix(C^{(1)}_{\wh E}).
\end{equation}
The projections 
\begin{equation}
\pi: \Sht(C_{E})\to \Sht(C), \quad  \wh \pi: \Sht(C_{\wh E})\to \Sht(C)
\end{equation}
are \'{e}tale $\F_{q}$-vector space bundles over $\Sht(C)$ that are dual to each other. 

\sssec{}

Let $\cK\in \Dmotg{\quo{E}}$ and $\frc \in \Corr_{\quo{C_E}}(\cK, \cK{\tw{-i}})$.

On one hand, we can form the Frobenius-twisted trace (\S \ref{sssec: sht val tr})
\begin{equation}\label{eq: tr ft 1}
\Tr^{\Sht}(\frc) \coloneqq \Tr(\frc^{(1)})\in \CH_i(\quo{\Sht(C_E)})
\end{equation}
where $\frc^{(1)}$ is the cohomological correspondence
\begin{equation}
\frc^{(1)} \co e_{0}^{*}\Frob_{E}^{*}\cK\cong e_{0}^{*}\cK\xr{\frc}e_{1}^{!}\cK\tw{-i}
\end{equation}
supported on $\quo{C_{E}^{(1)}}$ (see \S \ref{sssec: fix vs sht}).

On the other hand, we can first apply the homogeneous motivic Fourier transform to get a cohomological correspondence
\[\FT_{C_E}(\frc) \in \Corr_{\quo{C_{\wh{E}}}}(\FT_E(\cK), \FT_E(\cK))\]
given by the composite
\begin{equation}
\wh e_{0}^{*}\FT_{E}(\cK)\cong \FT_{C_{E}}(e_{0}^{*}\cK)\xr{\FT_{C_{E}}(\frc)}
\FT_{C_{E}}(e_{1}^{!}\cK\tw{-i})\cong \wh e_{1}^{!}\FT_{E}(\cK){\tw{-i}},
\end{equation}
where we used the commutativity of the Fourier transform with base change (\S \ref{ssec:main/funbase}).
We can then (using that $\FT(\cK)$ is geometric, thanks to Corollary~\ref{cor: FT preserve local constructibility}) form the Frobenius-twisted trace
\begin{equation}\label{eq: tr ft 2}
\Tr^{\Sht}(\FT_{C_{E}}(\frc))\coloneqq\Tr(\FT_{C_{E}}(\frc)^{(1)})\in \CH_i(\quo{\Sht(C_{\wh E})}).
\end{equation}

The two constructions \eqref{eq: tr ft 1} and \eqref{eq: tr ft 2} agree up to the arithmetic Fourier transform:

\begin{thm}\label{thm: trace compatible with FT} In the above situation, we have
\begin{equation}
\Tr^{\Sht}(\FT_{C_{E}}(\frc))=\FT^{\arith}_{\Sht(C_{E})}(\Tr^{\Sht}(\frc))\in \CH_i(\quo{\Sht(C_{\wh E})}).
\end{equation}
\end{thm}

\begin{proof} The version for $\ell$-adic coefficients is \cite[Theorem 8.3.2]{FYZ3}, and the proof here is similar with a few modifications. The renormalized homogeneous Fourier transform  (resp. arithmetic homogeneous Fourier transform) is the composition of three steps: 
\begin{enumerate}
\item pullback along $\quo{E} \leftarrow \quo{E} \times_Y \quo{\wh{E}}$ (resp. pullback along $\quo{\Sht(E)} \leftarrow \quo{\Sht(E)} \times_{\Sht(C)}  \quo{\Sht(\wh{E})}$),
\item tensor with $\sP_E[r-1] $  (resp. multiply by $(-1)^r \ev^* \xi$) where $r=\rk(E)$,
\item pushforward along $\quo{E} \times_Y \quo{\wh{E}} \rightarrow \quo{\wh{E}}$ (resp. pushforward along $ \quo{\Sht(E)} \times_{\Sht(C)}  \quo{\Sht(\wh{E})} \rightarrow \quo{\Sht(\wh{E})}$).
\end{enumerate}
It suffices to show that each of these steps is compatible with the formation of the trace. The first two are easy:
\begin{itemize}
\item Since the pullback in (a) is smooth, the compatibility there follows from Proposition \ref{prop: trace commutes with smooth pullback}. 
\item The formation of trace takes the operation of tensoring with a self-correspondence of a local system to the operation of multiplication by the trace. Hence (b) follows from the computation that the Frobenius trace function of $\sP_E$ is $-\ev^* \xi$. 
\end{itemize}

For (c), we need to show that pushforward through the projection $\pr_1 \co \quo{E} \times_Y \quo{\wh{E}} \rightarrow \quo{\wh{E}}$ commutes with formation of trace. From the map of correspondences 
\[
\begin{tikzcd}
E \ar[d, "f"] & C_E \ar[l, "e_0"'] \ar[r, "e_1"] \ar[d, "f"] & E \ar[d, "f"] \\
\quo{E} & \quo{C_E}  \ar[l, "\quo{e}_0"'] \ar[r, "\quo{e}_1"] & \quo{E}
\end{tikzcd}
\]
we have maps of correspondences
\[\Corr_{\quo{C_E}}(\cK, \cK\tw{-i}) \xrightarrow{f^*} \Corr_{C_E}(f^* \cK, f^* \cK\tw{-i}) \]
and
\[\Corr_{C_E}(\cK, \cK\tw{-i}) \xrightarrow{f_!} \Corr_{\quo{C_E}}(f_! \cK,  f_!\cK\tw{-i}).\]
The endofunctor $f_! f^* \co \Dmot{\quo{E}} \rightarrow  \Dmot{\quo{E}}$ is given by tensoring with $f_! \Qsh{E}$. Since $f$ is a $\G_m$-torsor, this implies that $\Tr(f_! f^* \cc) = (q-1) \Tr(\cc)$, which agrees with $\Fix(f)_* \Fix(f)^! (\Tr(\cc))$ since $\Fix(f)$ is a $\G_m(\F_q)$-torsor. Replacing $\cc$ by $f^* \cc$, it therefore suffices to show that the projection $\pr_1 \co E \times_Y \wh{E} \rightarrow \wh{E}$ commutes with formation of trace. This follows from Lemma \ref{lem: vb push commutes with trace}.
\end{proof}

\begin{lemma}\label{lem: vb push commutes with trace}
In the situation of \eqref{corr vb}, let $\cK\in \Dmotg{E}$ and $\frc \in \Corr_{C_E}(\cK, \cK\tw{-i})$. Let $\Sht(\pr) \co \Sht(C_{E})\to \Sht(C)$ be the induced map on fixed points of $C_{E}^{(1)}$ and $C^{(1)}$, which is an \'etale $\F_q$-vector space bundle (in particular a finite morphism). Then
\begin{equation}
\Tr(\pr_{C!}(\frc))=\Sht(\pr)_{*}(\Tr(\frc^{(1)}))\in \CH_i(\Sht(C)).
\end{equation}
\end{lemma}

\begin{proof}Note that this does not follow from Proposition \ref{prop: trace commutes with proper push} since $\pr$ is far from proper. The analogous result for $\ell$-adic sheaves is \cite[Lemma 8.3.3]{FYZ3}, and the proof here is essentially the same: compactify the map of correspondences, use the compatibility of trace and pushforward on the compactification, and then show that the boundary contribution vanishes. As some references need to be replaced, we will spell out the argument. 

We compactify the map of correspondences \eqref{corr vb} to 
\begin{equation}\label{eq: compactified correspondence}
\xymatrix{\ov E\ar[d]^{\ov \pr} &  C_{\ov E}\ar[d]^{\ov \pr_{C}}\ar[l]_-{\ov e_{0}}\ar[r]^{\ov e_{1}} & \ov E\ar[d]^{\ov \pr}\\
Y & C\ar[l]_-{c_{0}}\ar[r]^{c_{1}} & Y
}
\end{equation}
where $\ov E \coloneqq\PP(E\oplus \cO)\to Y$ and $C_{\ov E} \coloneqq \PP(c_{0}^{*}E\oplus \cO) \stackrel{\iota}\cong \PP(c_{1}^{*}E\oplus \cO)\to C$ is the pullback projective bundle over $C$. Then $C_{\ov E}$ is a self-correspondence of $\ov E$ with a proper map to $C$. Let $E_{\infty}\coloneqq\ov E-E$ be the divisor at infinity, which is isomorphic to $\PP(E)$. Similarly define $C_{E_{\infty}} \coloneqq C_{\ov E}-C_{E}$, which is a self-correspondence of $E_{\infty}$. Let $C^{(1)}_{\ov E}$ and $C^{(1)}_{E_{\infty}}$ be the twists by Frobenius as in \S \ref{sssec: fix vs sht}. Since the vertical maps in \eqref{eq: compactified correspondence} are proper, Proposition \ref{prop: trace commutes with proper push} implies that $\ol{\pr}_{C!}$ is compatible with formation of traces. 

Let $j:E\incl \ov E$ and $j_{C}: C_{E}\incl C_{\ov E}$ be the open inclusions. The map of correspondences 
\[
\begin{tikzcd}
E \ar[d, hook, "j"] & C_E \ar[l, "e_0"'] \ar[r, "e_1"] \ar[d, hook, "j_C"] & E \ar[d, hook, "j"] \\
\ol{E} & C_{\ol{E}}  \ar[l, "\ol{e}_0"'] \ar[r, "\ol{e}_1"] & \ol{E}
\end{tikzcd}
\]
has both squares Cartesian, so it is left pushable. Therefore the pushforward cohomological correspondence
\begin{equation}
\ov\frc\coloneqq j_{C!}(\frc) \in \Corr_{C_{\ov E}}(j_! \cK, j_! \cK \tw{-i})
\end{equation}
is defined. It remains to show that $\Sht(\pr)_* \Tr(\ol{c}) = \Sht(\pr)_* \Tr(\frc)$, which amounts to the vanishing of the contribution from the boundary correspondence. The rest of the argument is exactly the same as the corresponding step in the proof of \cite[Lemma 8.3.3]{FYZ3}, except using Theorem~\ref{thm: contracting fixed terms} instead of \cite{Var07} to see that the contribution from the boundary correspondence vanishes, because $C^{(1)}_{\ov E}$ is contracting near $E_{\infty}$ by Example~\ref{ex: Frobenius twist contracting}.
\end{proof}

\section{Generic modularity of higher theta series}\label{sec: generic modularity}
In this section we will assemble the preceding theory to prove the main result, Theorem \ref{thm: main}.

\subsection{Notation} We fix the following notation throughout the section.

We let $\nu \co X' \rightarrow X$ be an \'{e}tale double cover of smooth projective curves over a finite field $\F_q$ of characteristic $p>2$, and $\sigma \co X' \rightarrow X'$ be the non-trivial automorphism over $X$. We let $\Frob$ denote the $q$-power Frobenius.

For a torsion coherent sheaf $Q$ on a curve $X'$ we let $D_Q$ be its scheme-theoretic support, viewed as a divisor on $X'$, and $|Q| \subset X'$ its set-theoretic support.

Let $n \in \Z_{\geq 1}$. The (smooth, classical, $1$-Artin) stack $\Bun_{GU(n)}$ parametrizes triples $(\cF,\frL, h)$, where $\cF$ is a vector bundle on $X'$ of rank $n$, $\frL$ is a line bundle on $X$, and $h: \cF\isom \s^{*}\cF^{\vee}\ot\nu^{*}\frL$ is an $\frL$-twisted Hermitian structure (i.e., $\s^{*}h^{\vee}=h$). The corresponding moduli space of shtukas is denoted $\Sht_{GU(n)}$.

\subsection{Higher theta series}\label{ssec: higher theta series}

We briefly summarize the construction of higher theta series on $\Sht_{GU(n)}$, which can be found in \cite[\S 4]{FYZ2}. Let $m \in \Z_{\geq 1}$. The stack $\Bun_{GU^{-}(2m)}$ parametrizes triples $(\cG,\frM, h)$, where $\cG$ is a vector bundle on $X'$ of rank $2m$, $\frM$ is a line bundle on $X$, and $h: \cG\isom \s^{*}\cG^{\vee}\ot\nu^{*}\frM$ is an $\frM$-twisted {\em skew-Hermitian} structure (i.e., $\s^{*}h^{\vee}=-h$). Alternatively, we can think of $h$ as an $\cO_{X'}$-bilinear perfect pairing
\begin{equation}\label{h to pairing}
(\cdot,\cdot)_{h}: \cG\times \s^{*}\cG\to \nu^{*}(\frM\ot \om_{X})
\end{equation}
satisfying $(\s^{*}\b,\s^{*}\a)_{h}=-\s^{*}(\a,\b)_{h}$  for local sections $\a$ and $\b$ of $\cG$ respectively.

Let $\Bun_{\wt P_{m}}$ be the moduli stack of quadruples $(\cG,\frM,h,\cE)$ where $(\cG,\frM, h)\in \Bun_{GU^{-}(2m)}$, and $\cE \subset \cG$ is a Lagrangian sub-bundle (i.e., $\cE$ has rank $m$ and the composition $\cE\subset \cG\xr{h}\s^{*}\cG^{\vee}\ot\nu^{*}\frM\to \s^{*}\cE^{\vee}\ot\nu^{*}\frM$ is zero). Thus $\wt{P}_m$ corresponds to the Siegel parabolic of $GU^-(2m)$. 

Assume now that $1 \leq m \leq n$. In \cite[\S 4.6]{FYZ2}, we defined for each $r \geq 0$ and $m \leq n$ a \emph{higher theta series}, which is a function
\[
\wt Z^{r}_{m}  \co \Bun_{\wt P_{m}}(k) \to \CH_{r(n-m)}(\Sht^{r}_{GU(n)}).
\]
We repeat the brief recap of the definition of $\wt Z^{r}_{m}$ from \cite[\S 2]{FYZ3}.
Let $(\cG,\frM,h,\cE) \in \Bun_{\wt P_m}(k)$ and set $\frL \coloneqq\om_{X}\ot \frM$. Let $\Sht^{r}_{U(n),\frL} \subset \Sht^{r}_{GU(n)}$ be the moduli stack of rank $n$ Hermitian shtukas 
\[
\cF_{\bullet}=((x_{i}), (\cF_{i}), (f_{i}), \ph: \cF_{r}\isom {}^{\tau}\cF_{0})
\]
on $X'$ with $r$ legs and similitude line bundle $\frL$. For a vector bundle $\cE$ on $X'$ of rank $m$, the special cycle $\cZ^{r}_{\cE, \frL}$ parametrizes a point $\cF_{\bu}$ of $\Sht^{r}_{U(n),\frL}$, and maps $t_{i}: \cE\to \cF_{i}$ for each $0\le i\le r$, compatible with the shtuka structure on $\cF_{\bullet}$. For a Hermitian map $a: \cE\to \s^{*}\cE^{\vee}\ot\nu^{*}\frL$, let $\cZ^{r}_{\cE, \frL}(a)$ be the open-closed substack of $\cZ^{r}_{\cE, \frL}$ consisting of $(\cF_{\bullet}, t_{\bullet})$ such that the Hermitian form on $\cF_{\bullet}$ induces the Hermitian map $a$ on $\cE$ via $t_{\bullet}$. Let $\z: \cZ^{r}_{\cE,\frL}(a)\to \Sht^{r}_{U(n),\frL}\subset \Sht^{r}_{GU(n)}$ be the  map forgetting $t_{\bu}$, which is known to be finite \cite[Proposition 7.5]{FYZ} and unramified.

In \cite[Definition 4.8]{FYZ2} there is constructed a {\em virtual fundamental class} $[\cZ^{r}_{\cE,\frL}(a)]\in \CH_{r(n-m)}(\cZ^{r}_{\cE,\frL}(a))$, although the interpretation as a derived fundamental class in \cite[\S 5,6]{FYZ2} will be more useful in the proofs. Pushing forward along $\z$, we get Chow classes
\begin{equation*}
\z_{*}[\cZ^{r}_{\cE,\frL}(a)]\in \CH_{r(n-m)}(\Sht^{r}_{U(n),\frL}).
\end{equation*}
The value of $\wt Z^{r}_{m}$ on $(\cG,\frM,h,\cE)$ (recall $\frM=\om_{X}^{-1}\ot\frL$), which we henceforth abbreviate as $(\cG,\cE)$,  is defined as
\begin{equation}
\wt Z^{r}_{m}(\cG,\cE)=\chi(\det\cE)q^{n(\deg \cE-\deg\frL-\deg\om_{X})/2}\sum_{a\in \cA_{\cE,\frL}(k)}\psi(\j{e_{\cG,\cE},a})\z_{*}[\cZ^{r}_{\cE,\frL}(a)],
\end{equation}
where:
\begin{itemize}
\item $\chi: \Pic_{X'}(k)\to \ol{\Q}^{\times}$ is a character whose restriction to $\Pic_{X}(k)$ is the $n$th power of the quadratic character $\Pic_{X}(k)\to \{\pm1\}$ corresponding to the double cover $X'/X$ by class field theory. 
\item $\psi:\F_{q}\to \ol{\Q}^{\times}$ is a nontrivial additive character.
\item The sum is indexed over $\cA_{\cE,\frL}(k)$, the set of Hermitian maps $a: \cE\to \s^{*}\cE^{\vee}\ot\nu^{*}\frL$.
\item Let $\cE'=\cG/\cE$. The pairing $(\cdot,\cdot)_{h}$ in \eqref{h to pairing} induces a perfect pairing $\cE\times \s^{*}\cE'\to \nu^{*}\frL$, which identifies $\cE'$ with $\s^{*}\cE^{*}\ot \nu^{*}\frL$. We thus have a short exact sequence
\begin{equation*}
\xymatrix{0\ar[r] & \cE\ar[r] & \cG\ar[r] & \s^{*}\cE^{*}\ot\nu^{*}\frL\ar[r] & 0}
\end{equation*}
and $e_{\cG,\cE}\in \Ext^{1}(\s^{*}\cE^{*}\ot\nu^{*}\frL, \cE)$ is its extension class. 
\item The pairing $\j{-,-}$ is the Serre duality pairing between $\Ext^{1}(\s^{*}\cE^{*}\ot\nu^{*}\frL, \cE)$ and $\Hom(\cE, \s^{*}\cE^{\vee}\ot\nu^{*}\frL)$.
\end{itemize}

\subsection{The generic modularity theorem}

The Modularity Conjecture of \cite{FYZ2} predicts that $\wt{Z}_m^r$ descends to a function on $\Bun_{GU^-(2m)}(k)$. As explained just after \cite[Conjecture 4.15]{FYZ2}, this can be reformulated as the assertion that $\wt Z^{r}_{m}(\cG, \cE)$ is actually independent of the choice of Lagrangian sub-bundle $\cE \subset \cG$. We will prove this statement after restriction to the generic locus (cf. \S\ref{sssec: generic locus}): 

\begin{thm}\label{thm: generic modularity}
For any $\cG \in \Bun_{GU^-(2m)}(k)$ and any Lagrangian sub-bundles $\wt{\cE}_1, \wt{\cE}_2 \subset \cG$, we have 
\begin{equation}\label{eq: generic modularity}
\wt{Z}_m^r(\cG, \wt{\cE}_1) = \wt{Z}_m^r(\cG, \wt{\cE}_2) \in \CH_{r(n-m)}(\Sht^r_{GU(n)} \times_{(X')^r} \eta^r).
\end{equation}
where $\eta^r$ is as in \S\ref{sssec: generic locus}. 
\end{thm}

\subsubsection{Reduction to transverse Lagrangians}\label{sssec: transverse reduction} As explained in \cite[\S 2.2]{FYZ3}, Theorem \ref{thm: generic modularity} is reduced to the case where the sub-bundles $\wt{\cE}_1, \wt{\cE}_2$ are \emph{transverse} in the sense that their intersection within $\cG$ is the zero section. Henceforth we assume this to be the case.

\subsubsection{Reduction to Harder--Narasimhan truncations}\label{sssec: H-N truncation} Given a Harder--Narasimhan polygon $\mu$ for $GU(n)$, we have a Harder--Narasimhan truncation $\Sht^{r, \leq \mu}_{GU(n)} \inj \Sht^{r}_{GU(n)}$. Because
\[
\CH_*(\Sht^r_{GU(n)} \times_{(X')^r} \eta^r) \cong \limit_{\mu} \CH_*(\Sht^{r, \leq \mu}_{GU(n)} \times_{(X')^r} \eta^r),
\]
it suffices to show \eqref{eq: generic modularity} after restriction to this truncation. Henceforth we fix a Harder--Narasimhan polygon $\mu$ and write $S=\Bun_{GU(n)}^{\le \mu}$ for the corresponding open substack of $\Bun_{GU(n)}$. As in \cite[\S 9.1.2]{FYZ3}, we write $\Hk_S^r \coloneqq h_0^{-1}(S) \cap \ldots \cap h_r^{-1}(S)$ where $h_i \co \Hk_{GU(n)}^r \rightarrow \Bun_{GU(n)}$ are the ``leg maps''.  

For a space over $\Bun_{GU(n)}$ such as $\Sht_{GU(n)}^r$ or the special cycles on it, we write $(-)^{\leq \mu}$ for the pullback to $S$.

\subsection{Transverse Lagrangians ansatz}\label{ssec: Lagrangians setup}
Let $\cG \in \Bun_{GU^{-}(2m)}(k)$. For notational simplicity, we assume the similitude line bundle of $\cG$ is trivial, therefore the skew-Hermitian form reads $h_{\cG}: \cG\isom \s^{*}\cG^{*}$. The general case has the same content. 

Thanks to the transversality assumption from \S \ref{sssec: transverse reduction}, we have a commutative diagram (cf. \cite[\S 2.3.3]{FYZ3})
\begin{equation}\label{eq: Q diagram for saturation}
\begin{tikzcd}
\wt \cE_1 \ar[rr, "b_{12}"] \ar[dr, hook] & & \sigma^* \wt \cE_2^* \ar[rr]\ar[dr, hook, "i_{1}"] & &  Q_{2}  \ar[dr, "\io_{2}"] \\
& \cG \ar[ur] \ar[rr] \ar[dr]  & &  \cG^{\sh}\ar[rr] && Q\\
\wt \cE_2 \ar[rr, "b_{21}"] \ar[ur, hook] & & \sigma^* \wt \cE_1^* \ar[rr] \ar[ur, hook, "i_{2}"] & & Q_{1}\ar[ur, "\io_{1}"] 
\end{tikzcd}
\end{equation}
where the horizontal rows and diagonal sequences are short exact sequences of coherent sheaves on $X'$. Here $b_{12}$ is the composition
\[
\wt{\cE_1} \rightarrow \cG\isom \s^{*}\cG^{*} \rightarrow \sigma^*  \wt{\cE}_2^*. 
\]
The map $b_{21}$ is defined similarly. The maps $\io_{1}$ and $\io_{2}$ are isomorphisms of torsion sheaves on $X'$. As in \cite[\S 2.3.3]{FYZ3}, the duality between $Q_{1}$ and $Q_{2}$ equips $Q$ with two Hermitian structures $h_{12} \co Q \xrightarrow{\sim} \sigma^* Q^*$ and $h_{21} \co  Q \xrightarrow{\sim} \sigma^* Q^*$, related by $h_{12} = -h_{21}$.

For each $i \in \{1,2\}$ let $\cE_i \inj \wt{\cE_i}$ be a sub-sheaf with cokernel a torsion coherent sheaf $\cT_i$ on $X'$. Let $\cT^{*}_{i}=\RHom(\cT_{i}, \cO_{X'})[1]$ be its dual torsion sheaf on $X'$. Therefore, $\wt{\cE_i}$ is the saturation of $\cE_i$ in $\cG$, and we have the diagram below 
\begin{equation}\label{eq: TL exact sequence}
\begin{tikzcd}
\cE_1  \ar[r, hook, "\cT_1"] &   \wt{\cE_1} \ar[r, hook] \ar[rr, hook, bend right, "Q"'] & \cG \ar[r, twoheadrightarrow] & \s^{*}\wt\cE_{2}^* \ar[r, hook, "\sigma^* \cT_2^*"] &  \s^{*}\cE_2^*
\end{tikzcd}
\end{equation}
where the arrows are labeled by their cokernels. 

\subsubsection{Assumptions on $\cT_{1}$ and $\cT_{2}$}\label{sss: assumptions Qi} We make the following assumptions:
 \begin{enumerate}
\item The supports $|Q|, |\cT_1|, |\cT_2|$ are disjoint after mapping to $X$.
\item For all $\cF \in S(\ov k)=\Bun_{U(n)}^{\leq \mu}(\ol{k})$ we have for $i=1,2$
\begin{equation}\label{eq: vanishing assumption 1}
\Ext^1_{X_{\ol{k}}}(\cF^*, \cE_i^*) = 0.
\end{equation}
\end{enumerate}
These conditions can always be arranged, as discussed in \cite[Remark 10.1.1]{FYZ3}. Note that by the dualities in \cite[(10.1.3)]{FYZ3}, \eqref{eq: vanishing assumption 1} is equivalent to
\begin{equation}\label{eq: vanishing assumption 2}
\Hom_{X_{\ol{k}}}(\cF^*, \sigma^* \cE_i) = 0
\end{equation}
for all $\cF \in \Bun_{U(n)}^{\leq \mu}(\ol{k})$ and $i=1,2$. 

Let
\begin{eqnarray}
\wt Q_{1}\coloneqq Q^{*} \oplus \cT_1^* \oplus \s^{*}\cT_2,\\
\wt Q_{2}\coloneqq \sigma^* Q \oplus \s^{*}\cT_1 \oplus \cT_2^{*}.
\end{eqnarray}
From \eqref{eq: TL exact sequence} and the disjointness assumption in \S\ref{sss: assumptions Qi}, we have short exact sequences
\begin{align}
0 \rightarrow  \s^{*}\cE_2 \rightarrow \cE_1^* \rightarrow \wt Q_{1} \rightarrow 0,\label{eq: TL SES 1}\\
0 \rightarrow \s^{*}\cE_1 \rightarrow \cE_2^* \rightarrow \wt Q_{2}\rightarrow 0.\label{eq: TL SES 2}
\end{align}
  
\subsection{Moduli spaces}\label{ssec: moduli spaces}

We will use the same moduli spaces as in \cite[\S 9.1, \S 9.2]{FYZ3}:
\begin{enumerate}
\item For $i \in \{0,r\}$: $U_i,V_i, W_i$, which are derived vector bundles over $S$; and their respective dual derived vector bundles $W_i^\perp, \wh{V}_i, U_i^\perp$. 
\item For $i \in \{0,r\}$: $\wt{U}_i, \wt{V}_i, \wt{W}_i$, which are derived vector bundles over $\Hk_S^r$ obtained by pulling back $U_i,V_i,W_i$ respectively along $h_i$; and their respective dual derived vector bundles $\wt{W}_i^\perp, \wh{\wt{V}}_i, \wt{U}_i^\perp$. 
\item The Hecke stacks $\Hk_U^\perp, \Hk_U^\sharp, \Hk_V^\flat, \Hk_V^\sharp, \Hk_W^\flat, \Hk_W^\sharp$, which are derived vector bundles over $\Hk_S^r$; and their respective dual derived vector bundles $\Hk_{W^\perp}^{\sharp}, \Hk_{W^\perp}^{\flat}, \Hk_{\wh{V}}^\sharp, \Hk_{\wh{V}}^\flat, \Hk_{U^\perp}^\sharp, \Hk_{U^\perp}^\flat$. 
\end{enumerate}
We give an informal summary of the definitions. Recall that for an animated $\F_q$-algebra $R$, an $R$-point of $S$ is a Hermitian bundle $\cF$ on $X'_R$. In these terms, 
\begin{itemize}
\item The fiber of $U_i$ over $\cF \in S(R)$ is $\RHom_{X_R'}(\cF^*, \cE_1^* \otimes R)$. 
\item The fiber of $V_i$ over $\cF \in S(R)$ is $\RHom_{X_R'}(\cF^*, \wt{Q}_1 \otimes R)$. 
\item The fiber of $W_i$ over $\cF \in S(R)$ is $\RHom_{X_R'}(\cF^*, \sigma^* \cE_2[1] \otimes R)$. 
\item The fiber of $U_i^\perp$ over $\cF \in S(R)$ is $\RHom_{X_R'}(\cF^*, \cE_2^* \otimes R)$.
\item The fiber of $\wh{V}_i$ over $\cF \in S(R)$ is $\RHom_{X_R'}(\cF^*, \wt{Q}_2 \otimes R)$. 
\item The fiber of $W_i^\perp$  over $\cF \in S(R)$ is $\RHom_{X_R'}(\cF^*, \sigma^* \cE_1[1] \otimes R)$. 
\end{itemize}
Note that the definitions of the six spaces above are, in fact, independent of $i$. However, it the notation is useful for indexing purposes. 

For an animated $\F_q$-algebra $R$, an $R$-point of $\Hk_S^r$ is a sequence of modifications of Hermitian bundles $(\cF_0 \dashrightarrow \ldots \dashrightarrow \cF_r)$ on $X_R'$ that we abbreviate $\cF_\star$. For $i \in \{0,r\}$ and $? \in \{U_i,V_i, W_i, U_i^\perp, \wh{V}_i, W_i^\perp \}$, the fiber of $\wt{?}$ over $(\cF_\star) \in \Hk_S^r(R)$ is obtained by replacing $\cF$ with $\cF_i$ in the descriptions of $?(R)$ above. 

From $\cF_\star \in \Hk_S^r(R)$ we can define perfect complexes $\cF_\bu^\flat$ and $F_\bu^\sharp$ on $X_R'$ as in \cite[\S 9.1.3, \S 9.2.2]{FYZ3}. For $? \in \{U,V,W\}$ the fiber of $\Hk_?^{\flat}$ is obtained by replacing $\cF$ with $\cF_{\bu}^\flat$ in the description of $?(R)$ above, while the fiber of $\Hk_?^{\sharp}$ over $\cF_\star \in \Hk_S^r(R)$ is obtained by replacing $\cF$ with $\cF_{\bu}^\sharp$ in the description of $?(R)$ above.

\begin{remark}The vanishing assumption \eqref{eq: vanishing assumption 1} implies that $U,V$ and $W$ are all classical vector bundles over $S$, and we have a short exact sequence of classical vector bundles over $S$
\begin{equation}\label{eq: UVW SES}
0\to U_i \to V_i \to W_i\to 0.
\end{equation}
Similarly, we have a short exact sequence of classical vector bundles over $S$
\begin{equation}
0\to U^{\bot}_i \to \wh V_i \to W^{\bot}_i\to 0.
\end{equation}
\end{remark}

Thus we have a pair of commutative diagrams:
\begin{equation}\label{eq: big diagram for E_1}
\adjustbox{scale=0.8, center}{\begin{tikzcd}
& & \Hk_U^{\flat} \ar[dl, "\wt{a}_0"', color=blue] \ar[dr, "\wt{a}_r"] \ar[ddd, bend left, "f"', color=orange] \\
& \ar[dl, "h_0^U"'] \wt{U}_0 \ar[ddd, "\wt{f}_0"', color=red] \ar[dr, "\wt{a}_r'"'] && \wt{U}_r \ar[dl, "\wt{a}_0'"] \ar[ddd, "\wt{f}_r"] \ar[dr, "h_r^U"] \\
U_0 \ar[ddd, "f_0"]  & & \Hk_U^{\sharp}  \ar[ddd, bend right, "f^{\sharp}"]   &  &  U_r  \ar[ddd, "f_r"] \\
& & \Hk_V^{\flat} \ar[dl, "\wt{b}_0"', color=green] \ar[dr, "\wt{b}_r"] \ar[ddd, bend left, "g"']   \\
& \wt{V}_0 \ar[ddd, "\wt{g}_0"']  \ar[dr, "\wt{b}_r'"']  \ar[dl, "h_0^V"'] && \wt{V}_r \ar[dl, "\wt{b}_0'"]  \ar[ddd, "\wt{g}_r"]  \ar[dr, "h_r^V"] \\
V_0 \ar[ddd, "g_0"] & & \Hk_V^{\sharp} \ar[ddd, bend right, "g^{\sharp}"]  & & V_r  \ar[ddd, "g_r"] \\
& & \Hk_W^{\flat}  \ar[dl, "\wt{c}_0"'] \ar[dr, "\wt{c}_r"] \\
& \wt{W}_0 \ar[dr, "\wt{c}_r'"'] \ar[dl, "h_0^W"']  & &  \wt{W}_r \ar[dl, "\wt{c}_0'"]  \ar[dr, "h_r^W"]  \\
W_0 & & \Hk_W^{\sharp}  & & W_r 
\end{tikzcd}\hspace{1cm}
\begin{tikzcd}
& & \Hk_{U^{\perp}}^{\flat} \ar[dl, "\wt{a}_0^{\perp}"'] \ar[dr, "\wt{a}_r^{\perp}"] \ar[ddd, bend left, "f^{\perp}"'] \\
& \ar[dl, "h_0^{U^{\bot}}"'] \wt{U}_0^{\perp} \ar[ddd, "\wt{f}_0^{\perp}"'] \ar[dr, "(\wt{a}_r')^{\perp}"'] && \wt{U}_r^{\perp} \ar[dl, "(\wt{a}_0')^{\perp}"] \ar[ddd, "\wt f_r^{\perp}"] \ar[dr, "h_r^{U^{\bot}}"]  \\
U_0^{\perp} \ar[ddd, "f_0^{\perp}"] & & \Hk_{U^{\perp}}^{\sharp}  \ar[ddd, bend right, "(f^{\sharp})^{\perp}"]   & & U_r^{\perp} \ar[ddd, "f_r^{\perp}"] \\
& & \Hk_{\wh{V}}^{\flat} \ar[dl, "\wt{\b}_0"'] \ar[dr, "\wt{\b}_r"] \ar[ddd, bend left, "g^{\bot}"']   \\
& \ar[dl, "h_0^{\wh V}"']  \wh{\wt{V}}_0 \ar[ddd, "\wt{g}_0^{\perp}"', color=red]  \ar[dr, "\wt\b'_{r}"', color=green]  && \wh{\wt{V}}_r \ar[dl, "\wt\b'_{0}"]  \ar[ddd, "\wt{g}_r^{\perp}"]  \ar[dr, "h_r^{\wh V}"]  \\
\wh{V}_0 \ar[ddd, "g_0^{\perp}"]  & & \Hk_{\wh{V}}^{\sharp} \ar[ddd, bend right, "(g^{\sharp})^{\perp}", color=orange] & & \wh{V}_r \ar[ddd, "g_r^{\perp}"]  \\	
& & \Hk_{W^{\perp}}^{\flat}  \ar[dl, "\wt{c}_0^{\perp}"'] \ar[dr, "\wt{c}_r^{\perp}"] \\
& \ar[dl, "h_0^{W^{\bot}}"']  \wt{W}_0^{\perp} \ar[dr, "(\wt{c}_r')^{\perp}"', color=blue]  & &  \wt{W}_r^{\perp} \ar[dl, "(\wt{c}_0')^{\perp}"]  \ar[dr, "h_r^{W^{\bot}}"]  \\
W_0^{\perp} & & \Hk_{W^{\perp}}^{\sharp} & & W_r^{\perp} 
\end{tikzcd}
}
\end{equation}
In each diagram: 
\begin{itemize}
\item The maps in the columns come from exact triangles of perfect complexes. 
\item The three diamonds in the middle are derived Cartesian. 
\item The four parallelograms on the left and right sides are derived  Cartesian. 
\end{itemize}
The diagram on the right is dual to the diagram on the left. The duality exchanges $U$ with $W^{\perp}$, $V$ with $\wh{V}$, and $W$ with $U^{\perp}$, and exchanges $\flat$ and $\sharp$ superscripts. Sample examples of dual morphisms are colored with the same color. By \cite[\S 9.3]{FYZ3}, each of these diagrams is globally presented, so that we may apply the results of \S \ref{ssec: global presentation} to them.

\subsection{Calculation of motivic Fourier transforms}\label{ssec: calculate FT of cc} We refer to the diagrams in \eqref{eq: big diagram for E_1}. By \cite[Corollary 9.1.4]{FYZ3}, the map $a_r$ is quasi-smooth, hence it has a relative fundamental class, which defines as in \S\ref{ssec: VFC as trace} a cohomological correspondence
\begin{equation}
\cc_{U}=[a_{r}]\in \Corr_{\quo{\Hk^{\flat}_{U}}}(\Qsh{U_0}, \Qsh{U_{r}}\tw{-d(a_r)}).
\end{equation}
Similarly, the relative fundamental class of $a^{\bot}_{r}$ defines a cohomological correspondence
\begin{equation}
\cc_{U^{\bot}}=[a^{\bot}_{r}]\in \Corr_{\quo{\Hk^{\flat}_{U^{\bot}}}}(\Qsh{U^{\bot}_0}, \Qsh{U^{\bot}_{r}}\tw{-d(a^{\bot}_r)}).
\end{equation}
By \cite[Corollary 9.1.5]{FYZ3} the pushforward of cohomological correspondences (\S \ref{ssec: pushforward functoriality for CC}) along the morphism of correspondences $f: \Hk^{\flat}_{U}\to \Hk^{\flat}_{V}$ is defined, giving
\begin{equation}
f_{!}(\cc_{U})\in \Corr_{\quo{\Hk^{\flat}_{V}}}(f_{0!}\Qsh{U_{0}}, f_{r!}\Qsh{U_{r}}\tw{-d(a_{r})}).
\end{equation}
Similarly, 
\begin{equation}
f^{\bot}_{!}(\cc_{U^{\bot}})\in \Corr_{\quo{\Hk^{\flat}_{\wh V}}}(f^{\bot}_{0!}\Qsh{U^{\bot}_{0}}, f^{\bot}_{r!}\Qsh{U^{\bot}_{r}}\tw{-d(a^{\bot}_{r})})
\end{equation}
is defined.

For a $\G_m$-equivariant cohomological correspondence $\cc$ on a derived vector bundle $E$, we denote by $\quo{\cc}$ the descended cohomological correspondence on $\quo{E}$. Recall the notion of homogeneous Fourier transform of cohomological correspondences from \S \ref{sssec: FT of cc}. We have (simplifying some dimensions as in \cite[\S 9.4]{FYZ3})
\begin{equation}
\rFT_{\Hk^{\flat}_{V}}(\quo{f}_{!}(\quo{\cc_{U}}))\in \Corr_{\quo{\Hk^{\flat}_{\wh V}}}(\rFT_{V_{0}}(f_{0!}\Qsh{U_{0}}), \rFT_{V_{r}}(f_{r!}\Qsh{U_{r}})\tw{d(\wt b_{r})-d(a_{r})}).
\end{equation}

Since $U^{\bot}$ is the orthogonal complement of $U$ relative to $V$ (in the derived sense), by \S\ref{ssec:main/fun} we have canonical isomorphisms for $i=0,r$:
\begin{eqnarray*}
\rFT_{V_{i}}(f_{i!}\Qsh{U_{i}})\cong (f^{\bot}_i)_!\Qsh{U^{\bot}_{i}}\sm{[\rank(\cV)]}\tw{-\rank(\cU)}.
\end{eqnarray*}
Note the shift and twist on the right side is the same for $i=0$ and $i=r$. Therefore we may view (simplifying some dimensions as in \cite[\S 9.4]{FYZ3}) 
\[
\rFT_{\Hk^{\flat}_{V}}(\quo{f}_{!}(\quo{\cc_{U}})) \in \Corr_{\quo{\Hk^{\flat}_{\wh V}}}((f^{\bot}_{0})_! \Qsh{U^{\bot}_{0}}, (f^{\bot}_r)_! \Qsh{U^{\bot}_{r}}\tw{-d(a_{r}^{\bot})}).
\]

\begin{thm}\label{thm: FT of pushforward correspondence} Recall the shift and twist notation from \S \ref{ssec: trace variants}. Let $\pi_i \co U_i \rightarrow S$ be the bundle projection. Then we have 
\[
\TT_{[d(f_0) + d(\pi_0)](d(\pi_0))}\rFT_{\quo{\Hk_V^{\flat}}}(\quo{f}_{!} (\quo{\cc_U}))   = (\quo{f^{\perp}})_{!} (\quo{\cc_{U^\perp}}) \in \Corr_{\quo{\Hk_{\wh V}^\flat}}((f_{0}^\perp)_! \Qsh{U^{\bot}_0}, (f_{r}^\perp)_! \Qsh{U^{\bot}_r} \tw{-d(a_r^\perp)}). 
\]
\end{thm}

\begin{proof}
Let $\frs\in \Corr_{\quo{\Hk^{r}_{S}}}(\Qsh{S}, \Qsh{S}\tw{-d(h_{r})})$ be the cohomological correspondence obtained from the relative fundamental class of $h_{r}$ as in \S \ref{ssec: VFC as trace}. Let 
\[
\pi: \Hk^{\flat}_{U}\to \Hk_{S}^{r}, \quad \pi^{\bot}: \Hk^{\flat}_{U^{\bot}}\to \Hk_{S}^{r},
\]
be the bundle projections viewed as maps of correspondences, and also recall the maps of correspondences
\begin{eqnarray*}
z^{\perp}: \Hk^{r}_{S}\to \Hk^{\flat}_{W^\perp}, \quad g^{\perp}: \Hk_{\wh{V}}^{\flat}\to \Hk_{W^\perp}^{\flat}.
\end{eqnarray*}

The proof is completed by the sequence of equalities of cohomological correspondences
\begin{align*}
\TT_{[d(f_0) + d(\pi_0)](d(\pi_0))}\rFT(\quo{f}_{!}\quo{\cc_{U}}) &\stackrel{(1)}{=} 
\TT_{[d(f_0) + d(\pi_0)](d(\pi_0))} \rFT(\quo{f}_{!}\quo{\pi}^{*}\quo{\frs}) \stackrel{(2)}{=} (\quo{g^{ \perp}})^{*}(\quo{z^\perp})_{!}\rFT(\frs)  \\
& \stackrel{(3)}{=}(\quo{g^\perp})^{*} (\quo{z^{\perp}})_! \frs\stackrel{(4)}{=}(\quo{f^\perp})_!(\quo{\pi^\perp})^{*}\frs\stackrel{(5)}{=} (\quo{f^\perp})_! \quo{\cc_{U^\perp}}.
\end{align*}
Here: 
\begin{enumerate}
\item[(1), (5)] follow from the $\G_m$-equivariant identifications 
\begin{equation*}
\pi^{*}\frs=\cc_{U}, \quad (\pi^{\bot})^{*}\frs=\cc_{U^{\bot}}
\end{equation*}
which are proved exactly as in \cite[Lemma 9.4.2]{FYZ3}. 

\item[(2)] involves two applications of Proposition \ref{prop: cc push/pull functoriality}, namely
\begin{equation*}
\TT_{[d(f_0)]}\rFT\c \quo{f}_{!}=(\quo{g}^\perp)^*\c\rFT, \quad \TT_{[d(\pi_0)](d(\pi_0)) } \rFT\c \quo{\pi}^{*}= \quo{z}^\perp_{!}\c\rFT.
\end{equation*}
\item[(3)] is the trivial equality $\frs=\rFT_{\Hk_{S}^{r}}(\frs)$, where $\Hk_{S}^{r}$ is regarded as a trivial vector bundle over itself.

\item[(4)] follows from Theorem \ref{thm: descent of pushforward correspondence}. (Note that the hypotheses of Theorem \ref{thm: descent of pushforward correspondence} hold in this situation by \cite[Corollary 9.1.5]{FYZ3}.) 
\end{enumerate}

\end{proof}

\subsection{Calculations with the homogeneous arithmetic Fourier transform}
We denote
\begin{eqnarray}
X^{\c}&\coloneqq&X-\nu(|Q|\cup |\cT_{1}|\cup |\cT_{2}|)=X-\nu(|\wt Q_{1}|)=X-\nu(|\wt Q_{2}|);\\
X'^{\c}&\coloneqq &\nu^{-1}(X^{\c}).
\end{eqnarray}
For a stack $A$ over $X^{\c}$, we denote
\begin{equation}
A^{\c}\coloneqq A\times_{X^{r}}(X^{\c})^{r}.
\end{equation}
In particular,   $\Hk_S^{r, \circ} \subset \Hk_S^r$ denotes the open substack where the legs are all disjoint from $|\wt{Q}_1|\cup|\wt Q_{2}|$.

By \cite[Lemma 10.13]{FYZ3}, for each $i$ the restriction $\wt b_i^{\c} \co \Hk_{V}^{\flat,\c} \rightarrow \wt{V}^{\c}_i$ of $\wt b_{i}$ and the restriction $\wt b'^{\c}_{i} \co  \wt{V}^{\c}_i\to \Hk_{V}^{\sh,\c} $ of $\wt b'_{i}$ are isomorphisms. This implies \cite[Corollary 10.1.4]{FYZ3} that the projection map $\Sht^{r,\c}_{V}\to \Sht^{r,\c}_{S}=\Sht^{r,\le \mu,\c}_{U(n)}$ is an \'{e}tale $\F_{q}$-vector space bundle. Hence the theory of the homogeneous arithmetic Fourier transform (\S \ref{sec: arithmetic FT}) applies to it.

 \subsubsection{Virtual fundamental classes}\label{ssec: VFC}
By \cite[Remark 9.1.1]{FYZ3} the spaces $U_i$ from \S \ref{ssec: moduli spaces} can be viewed as the derived fiber of the derived Hitchin stack $\sM_{H_1, H_2}$ from \cite[\S 5]{FYZ2} over $\{ \cE_1 \} \times S   \rightarrow \Bun_{\GL(m)'} \times \Bun_{U(n)}$, where $H_1 = \GL(m)'$ and $H_2 = U(n)$. Similarly, by \cite[Remark 9.1.2]{FYZ3} $\Hk_U^\flat$ is an open substack of the derived fiber of the derived Hecke stack $\Hk_{\sM_{H_1, H_2}}^r$ from \cite[\S 5]{FYZ2} over $\{ \cE_1 \} \times S   \rightarrow \Bun_{\GL(m)'} \times \Bun_{U(n)}$. Therefore, the derived fibered product 
\[
\begin{tikzcd}
\Sht_U^r \ar[r] \ar[d] & \Hk_U^\flat \ar[d, "{(a_0, a_r)}"] \\
U_0 \ar[r, "{(\Id, \Frob)}"] & U_0 \times U_r
\end{tikzcd}
\]
is equipped with an open embedding in $\Sht_{\sM_{H_1, H_2}}^r$, and in particular is of virtual dimension $d(a_r)$. We then have two natural cycles in $\CH_{d(a_r)}(\Sht_U^r )$:
\begin{enumerate}
\item The intrinsic derived fundamental class $[\Sht_U^r] \in \CH_{d(a_r)}(\Sht_U^r )$.
\item The trace of the cohomological correspondence $\cc_U$ (calculated using the canonical Weil structure), denoted $\Tr^{\Sht}(\cc_U) \in \CH_{d(a_r)}(\Sht_U^r )$ (cf. \S\ref{sssec: fix vs sht}). 
\end{enumerate}

We assemble the earlier results to calculate the trace of our cohomological correspondences. The assumptions \eqref{eq: vanishing assumption 1} imply that the maps $U_i \rightarrow S$ and $U_i^\perp \rightarrow S$ are smooth, hence $U_i$ and $U_i^\perp$ are smooth. Then by Corollary~\ref{cor: smooth derived local terms Sht}  we have 
\begin{equation}\label{eq: trace of cc_U}
\Tr^{\Sht}(\cc_U) = [\Sht_U^r] \in \CH_{d(a_r)}(\Sht_U^r).
\end{equation}
In particular, $\Sht_U^r$ is an open substack of $\Sht_{\cM_{\cE_1}}^r$, so it is quasi-smooth and $[\Sht_U^r]$ is the restriction of what was called $[\cZ_{\cE_1}^r]$ in \cite{FYZ2}. 

Similarly, we have
\begin{equation}\label{eq: trace of cc_U perp}
\Tr^{\Sht}(\cc_{U^\perp}) = [\Sht_{U^\perp}^r] \in \CH_{d(a_r^\perp)}(\Sht_{U^{\perp}}^r),
\end{equation}
where $\Sht_{U^\perp}^r$ is defined by the derived Cartesian square 
\[
\begin{tikzcd}
\Sht_{U^\perp}^r \ar[r] \ar[d] & \Hk_{U^\perp}^\flat \ar[d, "{(a_0^\perp, a_r^\perp)}"]\\
U_0^\perp \ar[r, "{(\Id, \Frob)}"] & U_0^\perp \times U_r^\perp
\end{tikzcd}
\]

Next, the assumptions \eqref{eq: vanishing assumption 2} imply that the maps $f_i \co U_i \rightarrow V_i$, $f \co \Hk_U^\flat \rightarrow \Hk_V^\flat$, $f_i^{\perp} \co U_i^{\perp} \rightarrow \wh{V}_i$, and $f^\perp \co \Hk_{U^{\perp}}^\flat \rightarrow \Hk_{\wh{V}}^\flat$ are all closed embeddings. Then Proposition \ref{prop: trace commutes with proper push} applies to give 
\begin{equation}\label{eq: trace f_!}
\Tr^{\Sht}(f_! \cc_U) = \Sht(f)_! \Tr^{\Sht}(\cc_U) \stackrel{\eqref{eq: trace of cc_U}}= \Sht(f)_! [\Sht_{U}^r] \in \CH_{d(a_r)} (\Sht_V^r),
\end{equation}
where we write $\Sht(f) \coloneqq \Fix(f^{(1)}) \co \Sht_U^r \rightarrow \Sht_V^r$ for the map induced by taking fixed points of the twisted cohomological correspondence $\mf{\cc}_U^{(1)}$, and similarly for other cohomological correspondences. We similarly have 
\begin{equation}\label{eq: trace f_! perp}
\Tr^{\Sht}(f_!^\perp \cc_{U^\perp}) = \Sht(f^\perp )_! \Tr^{\Sht}(\cc_{U^\perp}) \stackrel{\eqref{eq: trace of cc_U perp}}= \Sht(f^\perp )_! [\Sht_{U^\perp}^r] \in \CH_{d(a_r^\perp)} (\Sht_{\wh V}^r). 
\end{equation}

\begin{notat}For an $\F_q$-vector space stack $Y \rightarrow T$ as in \S \ref{sssec: AFT setup} and a class $\alpha \in \CH_*(Y)$ (resp. $\CH^{*}(Y)$), we denote
\[
\quo{\alpha} = \Av(\alpha) \coloneqq \frac{1}{q-1} \pr_! (\alpha) \in \CH_*(\quo{Y}) \quad \text{ (resp. $\CH^{*}(\quo{Y})$)}
\]
where $\pr  \co Y \rightarrow \quo{Y}$ is the quotient map. 
\end{notat}

\begin{example}[Homogeneous cycles]\label{ex: homogeneous cycle}
We say that $\alpha$ on $Y$ is \emph{homogeneous} if $\alpha = \pr^* (\quo{\alpha})$. Note that $[\Sht^r_U ] \in \CH_{d(a_r)} (\Sht_U^r)$ is homogeneous, hence $\Sht(f)_! [\Sht^r_U ]  \in \CH_{d(a_r)} (\Sht_V^r)$ is homogeneous.
\end{example}

\subsubsection{Arithmetic Fourier transform of cycles} Recall that $\Sht^{r,\c}_{V}\to \Sht^{r,\c}_{S}$ is an \'etale $\F_{q}$-vector space bundle. We now relate the cycle classes \eqref{eq: trace f_!} and \eqref{eq: trace f_! perp} under the homogeneous arithmetic Fourier transform on $\Sht^{r,\c}_{V}$ as defined in \S \ref{sec: arithmetic FT}. 

\begin{thm}\label{thm: AFT of distribution}  
We have 
\[
\aFT(\Sht(f)^{\c}_! (\quo{[\Sht_U^{r,\c}]})) =  (-1)^{d(U/S)+d(f_0)} q^{d(U/S)} \cdot  \Sht(f^{\perp})^{\c}_! (\quo{[\Sht_{U^{\perp}}^{r,\c}]}) \in \CH_{d(a_r)}(\quo{\Sht_{\wh{V}}^{r,\c}}).
\]
Here $\Sht(f)^{\c}: \Sht_U^{r,\c}\to \Sht^{r,\c}_{V}$ is the restriction of $\Sht(f)$, and similarly for $\Sht(f^{\perp})^{\c}$. We use the same notation for induced maps on the homogeneous quotients $\quo{(-)}$. (We are using \cite[Remark 10.2.2]{FYZ3} to match the degrees of homology.) 
\end{thm}

\begin{proof} We apply Theorem \ref{thm: trace compatible with FT} with $E = V$, $C_E = \Hk_V^{\flat,\c}$ and $\cc =( f_! \cc_U)|_{\Hk_V^{\flat,\c}}$.  Then Theorem \ref{thm: trace compatible with FT} tells us that 
\begin{equation}\label{eq: aft trace 1}
\aFT_{\Sht_V^{r,\c}}\left(\Tr^{\Sht}( \quo{f}_!  \quo{\cc_U})|_{\quo{\Sht^{r,\c}_{V}}} \right)= \left(\Tr^{\Sht} \rFT_{\Hk_V^\flat}(\quo{f}_!  \quo{\cc_U}) \right)|
_{\quo{\Sht^{r,\c}_{\wh V}}} \in \CH_{d(a_r)}(\quo{\Sht_{\wh{V}}^{r,\c}}).
\end{equation}
By Theorem \ref{thm: FT of pushforward correspondence} we have
\[
\rFT_{\Hk_V^{\flat}}(\quo{f}_{!} \quo{\cc_U})   = \TT_{[-d(U/S)-d(f_0)](-d(U/S))} (\quo{f}^{\perp})_{!} \quo{\cc_{U^\perp}}.
\]
Putting this into \eqref{eq: aft trace 1} and then taking the trace, using \eqref{eq: Tr Sht sh tw}, \eqref{eq: trace f_!} and \eqref{eq: trace f_! perp}, yields the result. 
\end{proof}

\subsubsection{Test functions} We introduce some notation for functions on $\Sht_{V}$ and $\Sht_{\wh{V}}$. The decompositions $\wt{Q}_1 \coloneqq Q^* \oplus \cT_1^* \oplus \sigma^* \cT_2$ and $\wt{Q}_2 \coloneqq \sigma^* Q \oplus \sigma^* \cT_1 \oplus \cT_2^*$ induce the following decompositions defined in \cite[\S 10.2.1]{FYZ3} (with the same notation): 
\[
\Sht_V^r = \Sht_{V^{(0)}}^r \times_{\Sht_S^r} \Sht_{V^{(1)}}^r \times_{\Sht_S^r} \Sht_{V^{(2)}}^r,
\]
\[
\Sht_{\wh{V}}^r = \Sht_{\wh{V}^{(0)}}^r \times_{\Sht_S^r} \Sht_{\wh{V}^{(1)}}^r \times_{\Sht_S^r}\Sht_{\wh{V}^{(2)}}^r.
\]
We note that $\Sht_{\wh{V}^{(i)}}^{r,\c}$ is dual to $\Sht_{V^{(i)}}^{r,\c}$ as $\F_{q}$-vector spaces over $\Sht^{r,\c}_{S}$ in the sense of \S \ref{sssec: AFT setup}.

We denote $\mf{q}_{12} \co \Sht^{r}_{V^{(0)}} \rightarrow \ul{\F_q}$ and $\mf{q}_{21} \co \Sht^{r}_{V^{(0)}} \rightarrow \ul{\F_q}$ the two quadratic forms induced by the Hermitian structures $h_{12}$ and $h_{21}$ on $Q$ from \cite[\S 2.3.3]{FYZ3}, respectively. Namely, $\frq_{12}$ is the composition
\begin{equation}
\frq_{12}: \Sht^{r}_{V^{(0)}}\xr{(\Id, h_{12})}\Sht^{r}_{V^{(0)}}\times_{\Sht^{r}_{S}}\Sht^{r}_{\wh V^{(0)}}\to \ul{\F_{q}},
\end{equation}
and similarly for $\frq_{21}$. They are related by $\mf{q}_{12} = - \mf{q}_{21}$. Recall the additive character $\psi \co \F_q \rightarrow \ol{\Q}^\times$ from \S \ref{ssec: higher theta series}. 
\begin{itemize}
\item We let $\mf{q}_{12}^* \psi$ be the pullback of $\psi$ to $\Sht_{V^{(0)}}^r$ via $\mf{q}_{12}$, and similarly for $\mf{q}_{21}$. Abusing notation, we will also use the same notation $\mf{q}^*_{12} \psi$ to denote its pullback to $\Sht_V^r$ and to $\Sht^{r,\c}_{\wh V}$. The meaning will be clear from context. 
\item We let $\delta_{\Sht_{V^{(i)}}^{r,\c}}$ be the indicator function of the zero-section of the \'etale $\F_{q}$-vector space bundle $\Sht_{V^{(i)}}^{r,\c} \rightarrow \Sht_{S}^{r,\c}$. Abusing notation, we will also use this same notation to denote its pullback to $\Sht_V^{r,\c}$. 
\item We let $\bbm{1}_{\Sht_{V^{(i)}}^{r,\c}}$ be the constant function of $\Sht_{V^{(i)}}^{r,\c}$ with value $1$. Abusing notation, we will also use this same notation to denote its pullback to $\Sht_V^{r,\c}$. 
\item We use similar notation on $\Sht_{\wh{V}}^{r,\c}$ and $\quo{\Sht_{\wh{V}}^{r,\c}}$. 
\end{itemize}

\begin{lemma}\label{lem: FT of test functions}
Let $d^{(i)}$ be the rank of $\Sht_{V^{(i)}}^{r,\c}$ as an \'etale $\F_q$-vector space bundle over $\Sht_{S}^{r,\c}$. Let $d = d^{(0)} +  d^{(1)} + d^{(2)}$ be the rank of $ \Sht_V^{r,\c}$ as an \'etale $\F_q$-vector space bundle over $\Sht_S^{r,\c}$. Then we have 
\begin{align*}
 \aFT_{\Sht_V^{r,\c}} \left( \Av({(\mf{q}_{12}^*  \psi)} \cdot {\delta_{\Sht^{r,\c}_{V^{(1)}}}} \cdot {\bbm{1}_{\Sht^{r,\c}_{V^{(2)}}}}) \right) 
 = (-1)^{d} q^{d^{(2)} + \frac{1}{2}d^{(0)}} \eta_{F'/F}(D_Q)^n \cdot \Av ( (\mf{q}_{12}^* [-1]^* \psi) \cdot \bbm{1}_{\Sht^{r,\c}_{V^{(1)}}} \cdot \delta_{\Sht^{r,\c}_{V^{(2)}}})
\end{align*}
as functions on $\quo{\Sht_{\wh{V}}^{r,\c}}$.  
\end{lemma}

\begin{proof} This follows from \cite[Corollary 10.2.4]{FYZ3} by applying $\Av(-)$.
\end{proof}

\begin{lemma}\label{lem: theta series as product} For $i=1,2$, let $\cA_{\wt{\cE}_i}$ be the Hitchin base as in \cite[\S 3.3]{FYZ2}. For $a\in \cA_{\wt\cE_{i}}(k)$, recall that $\cZ^{r, \leq \mu,\c}_{\wt{\cE_i}}(a) \coloneqq \cZ^{r}_{\wt\cE_{i}}(a)\times_{\Sht^{r}_{U(n)}}\Sht^{r,\c}_{S}$.
	
(1) We have an equality in $\CH_{d(a_r)} (\quo{\Sht_V^{r,\c}})$:
\[
(\Sht(f)^{\c}_! \quo{[\Sht_{U}^{r,\c}]}) \cdot  \Av ({(\mf{q}_{12}^* \psi)} \cdot  {\delta_{\Sht_{V^{(1)}}^{r,\c}}} \cdot {\bbm{1}_{\Sht_{V^{(2)}}^{r,\c}}})=  \Av\left(\sum_{a \in \cA_{\wt{\cE}_1}(k)} \psi(\langle e_{\cG, \cE_1}, a \rangle ) \Sht(f)^{\c}_! [\cZ^{r, \leq \mu,\c}_{\wt{\cE}_{1}}(a)] \right).
\]

(2) We have an equality in $\CH_{d(a_r)} (\quo{\Sht_{\wh{V}}^{r,\c}})$:
\[
(\Sht(f^{\perp})^{\c}_! 
\quo{[\Sht_{U^{\perp}}^{r, \c}
]})\cdot \Av({(\mf{q}_{21}^* \psi )} \cdot {\bbm{1}_{\Sht_{\wh{V}^{(1)}}^{r,\c}}} \cdot {\delta_{\Sht_{\wh{V}^{(2)}}^{r,\c}}})   =   \Av\left(\sum_{a \in \cA_{\wt{\cE}_2}(k)} \psi(\langle e_{\cG, \cE_2}, a \rangle )\Sht(f^{\perp})^{\c}_! [\cZ^{r, \leq \mu,\c}_{\wt{\cE}_{2}}(a)]\right).
\]

\end{lemma}

\begin{proof} This follows from \cite[Lemma 10.2.7]{FYZ3} by applying $\Av(-)$. 
\end{proof}

\subsection{Proof of Theorem \ref{thm: generic modularity}}\label{ssec: conclusion}

We will now complete the proof of Theorem \ref{thm: generic modularity}. Let $d, d^{(i)}$, for $i \in \{0,1,2\}$, be be the rank of $\Sht_{V^{(i)}}^{r,\c}$ as an \'etale $\F_q$-vector space bundle over $\Sht_{S}^{r,\c}$.  Note that $d^{(i)}$ is also the rank of $V^{(i)}$ as a vector bundle over $S$. 

From \cite[(10.3.1)]{FYZ3} we may rewrite the higher theta series in the following way, using the non-homogeneous variant of Notation \ref{notat: pairing},
\begin{equation}\label{eq: conclusion -3}
\wt{Z}_m^r( \wt{\cE}_1 , \cG) |_{\Sht^{r,\c}_{S}}= \chi(\det \wt{\cE}_1) q^{n (\deg \wt{\cE}_1  - \deg \omega_X)/2} \langle  \Sht(f)^{\c}_! [\Sht_U^{r,\c}], \mf{q}_{12}^* \psi \cdot \delta_{\Sht_{V^{(1)}}^{r,\c}} \cdot \bbm{1}_{\Sht_{V^{(2)}}^{r,\c}} \rangle
\end{equation}
and
\begin{equation}\label{eq: conclusion -2}
\wt{Z}_m^r(\wt{\cE}_2 , \cG)|_{\Sht^{r,\c}_{S}} = \chi(\det \wt{\cE}_2) q^{n (\deg \wt{\cE}_2  - \deg \omega_X)/2} \langle \Sht(f^\perp)^{\c}_! [\Sht_{U^\perp}^{r,\c}], \mf{q}_{21}^*  \psi \cdot \bbm{1}_{\Sht_{\wh{V}^{(1)}}^{r,\c}} \cdot \delta_{\Sht_{\wh{V}^{(2)}}^{r,\c}} \rangle.
\end{equation}
Since $\Sht(f)^{\c}_! [\Sht_U^{r,\c}]$ and $\Sht(f^\perp)^{\c}_! [\Sht_{U^\perp}^{r,\c}]$ are homogeneous (cf. Example \ref{ex: homogeneous cycle}), it suffices to show that
\begin{align}\label{eq: averaged modularity}
 &\chi(\det \wt{\cE}_1) q^{n (\deg \wt{\cE}_1  - \deg \omega_X)/2} \langle \Sht(f)^{\c}_! \ \quo{[\Sht_U^{r,\c}]}, \Av(\mf{q}_{12}^* \psi \cdot {\delta_{\Sht_{V^{(1)}}^{r,\c}}} \cdot {\bbm{1}_{\Sht_{V^{(2)}}^{r,\c}}} ) \rangle \nonumber \\
 =& 
 \chi(\det \wt{\cE}_2) q^{n (\deg \wt{\cE}_2  - \deg \omega_X)/2} \langle \Sht(f^\perp)^{\c}_! \ \quo{[\Sht_{U^\perp}^{r,\c}]}, \Av(\mf{q}_{21}^*  \psi \cdot {\bbm{1}_{\Sht_{\wh{V}^{(1)}}^{r,\c}}} \cdot {\delta_{\Sht_{\wh{V}^{(2)}}^{r,\c}}}  ) \rangle.
\end{align}

Let $d = d^{(0)} +  d^{(1)} + d^{(2)}$ be the rank of $ \Sht_V^{r,\c}$ as an $\F_q$-vector space over $\Sht_S^{r,\c}$. By the Plancherel formula from Lemma \ref{lem: Plancherel for cycles} and the involutivity of $\aFT$ from Lemma \ref{lem: AFT involutivity}, we have 
\begin{eqnarray}\label{eq: conclusion 1}
&& \langle \Sht(f)^{\c}_! \quo{[\Sht_U^{r,\c}]}, \Av(\mf{q}_{12}^*  \psi \cdot {\delta_{\Sht_{V^{(1)}}^{r,\c}}} \cdot {\bbm{1}_{\Sht_{V^{(2)}}^{r,\c}}} ) \rangle \\
\notag &=& \frac{1}{q^d} \langle \aFT ( \Sht(f)^{\c}_! \quo{[\Sht_U^{r,\c}]}) , \aFT( \Av(\mf{q}_{12}^*  \psi \cdot {\delta_{\Sht_{V^{(1)}}^{r,\c}}} \cdot {\bbm{1}_{\Sht_{V^{(2)}}^{r,\c}}} ))\rangle.
\end{eqnarray}
Using Theorem \ref{thm: AFT of distribution} and Lemma \ref{lem: FT of test functions}, we rewrite the RHS of \eqref{eq: conclusion 1} as 

\begin{eqnarray}\label{eq: conclusion 2}
 &   \frac{1}{q^d} q^{d(U/S)} (-1)^{d(U/S) + d(f_0)} (-1)^d q^{d^{(2)}+  \frac{1}{2}d^{(0)}} \eta_{F'/F}(D_Q)^n \\
& \hspace{1cm} \cdot  \notag \langle \Sht(f^\perp)^{\c}_!  \quo{[\Sht_{U^\perp}^{r,\c}]} , \Av(\mf{q}_{12}^* [-1]^* \psi \cdot {\bbm{1}_{\Sht_{\wh{V}^{(1)}}^{r,\c}}} \cdot {\delta_{\Sht_{\wh{V}^{(2)}}^{r,\c}}} )\rangle.
\end{eqnarray}
Since $\mf{q}_{12} =- \mf{q}_{21}$, we have $\mf{q}_{12}^* [-1]^* = \mf{q}_{21}^*$. Then clearly \eqref{eq: conclusion 2} agrees with the RHS of \eqref{eq: averaged modularity} up to sign and exponent of $q$. These signs and the exponents of $q$ on either side are matched exactly as in \cite[\S 10.3]{FYZ3}.

\bibliographystyle{amsalpha}
\bibliography{Bibliography}

\end{document}